\documentclass{amsart}
\usepackage[utf8]{inputenc}

\usepackage{amsmath}
\usepackage{mathtools}
\usepackage{amssymb}
\usepackage{mismath}
\usepackage{dirtytalk}
\usepackage{amsthm}
\usepackage{float}

\DeclareRobustCommand{\SkipTocEntry}[5]{}

\usepackage{blindtext}

\usepackage[margin=1.05in]{geometry}

\usepackage{array}
\usepackage{tabularx}
\usepackage{longtable}
\usepackage{chngpage}

\usepackage[table,dvipsnames]{xcolor}
\usepackage{hyperref}

\definecolor{bytreebluevertex}{RGB}{128,128,255}
\definecolor{bytreeblueline}{RGB}{77,77,255}
\definecolor{bytreeyellowline}{RGB}{255,169,82}
\definecolor{bytreeyellowvertex}{RGB}{255,232,191}

\usepackage{pdflscape}
\usepackage{comment}

\newtheorem{theorem}{Theorem}[section]
\newtheorem{corollary}[theorem]{Corollary}
\newtheorem{lemma}[theorem]{Lemma}
\newtheorem{fact}[theorem]{Fact}
\newtheorem{proposition}[theorem]{Proposition}

\usepackage[toc]{appendix}

\theoremstyle{definition}

\newtheorem{notation}[theorem]{Notation}
\newtheorem{definition}[theorem]{Definition}

\newtheorem{example}[theorem]{Example}

\newtheorem{remark}[theorem]{Remark}

\usepackage{enumerate}

\usepackage{wrapfig}

\usepackage{caption}
\usepackage{mathrsfs}

\usepackage{dsfont}

\usepackage{graphicx}

\usepackage[OT2,T1]{fontenc}

\usepackage{tikz}
\usetikzlibrary{shapes,positioning}
\usetikzlibrary{arrows,chains,matrix,positioning,scopes}
\usetikzlibrary {decorations.pathmorphing, decorations.pathreplacing, decorations.shapes}

 \usepackage[UKenglish]{babel}
\usepackage{graphicx}	
\usepackage{amsthm}
\usepackage{amsmath}
\usepackage{amssymb}
\usepackage{mathtools}
\usepackage{enumerate}
\usepackage{mathrsfs}
\usepackage{thmtools, thm-restate}
\usepackage[table]{xcolor}
\usepackage{tabularx, booktabs}
\usepackage{tikz}
\usepackage{pgfplots}

\usepackage{tikz-cd} 
\usetikzlibrary {arrows.meta,bending,positioning,patterns,decorations}
\usepackage{pbox}
\usepackage{tkz-graph}
\usetikzlibrary{snakes,fit,positioning,calc,shapes}
\usepackage{graphicx}
\usepackage{cleveref,tikz-cd,pbox,wrapfig}
\definecolor{amethyst}{rgb}{0.6, 0.4, 0.8}
\definecolor{atomictangerine}{rgb}{1.0, 0.6, 0.4}
\definecolor{deeppeach}{rgb}{1.0, 0.8, 0.64}
\definecolor{eggshell}{rgb}{0.94, 0.92, 0.84}
\definecolor{lightapricot}{rgb}{0.99, 0.84, 0.69}
\definecolor{lemonchiffon}{rgb}{1.0, 0.98, 0.8}
\definecolor{roundabout}{rgb}{1.0, 0.91, 0.75}
\definecolor{atomictangerine}{rgb}{1.0, 0.6, 0.4}
\definecolor{ruby}{rgb}{0.88, 0.07, 0.37}
\definecolor{sapphire}{rgb}{0.03, 0.15, 0.4}

\def\rootsep{0.03}               
\def\clustersep{0.06}            
\def\cnamescale{0.4}             
\def\cdepthscale{0.4}            
\def\cltopskip{1pt}              
\def\clbottomskip{1pt}           

\def\rootscaleDot{0.1}  \def\rootcolorDot{black}

\def\rootscale{0.5}   \def\rootcolor{gray}
\def\rootscaleA{0.7}  \def\rootcolorA{yellow}
\def\rootscaleB{0.5}  \def\rootcolorB{green}
\def\rootscaleC{0.4}  \def\rootcolorC{sapphire}
\def\rootscaleD{0.45}  \def\rootcolorD{ruby}
\tikzset{
  clA/.style = {very thick,black},
  clB/.style = {thick,purple},rootDot/.style = {circle,scale=\rootscaleDot,fill=\rootcolorDot},
    rcDot/.style 2 args = {right=#1*1.5*\clustersep of {#2.east|-first},rootDot}, rrDot/.style = {right=\rootsep of {#1.east|-first},rootDot}
}

\def\graphdslabelscale{0.6}

\def\GraphScale{0.6}

\tikzset{
  root/.style = {circle,scale=\rootscale,fill=\rootcolor},
    rc/.style 2 args = {right=#1*1.5*\clustersep of {#2.east|-first},root}, rr/.style = {right=\rootsep of {#1.east|-first},root},
  roott/.style = {circle,inner sep=-2pt,minimum size=5pt,black,font=\ttfamily\footnotesize},
    rct/.style 2 args = {right=#1*1.5*\clustersep of {#2.east|-first},roott}, rrt/.style = {right=\rootsep of {#1.east|-first},roott},
  rootA/.style = {circle,scale=\rootscaleA,ball color=\rootcolorA},
    rcA/.style 2 args = {right=#1*1.5*\clustersep of {#2.east|-first},rootA}, rrA/.style = {right=\rootsep of {#1.east|-first},rootA},
  rootB/.style = {circle,scale=\rootscaleB,ball color=\rootcolorB},
    rcB/.style 2 args = {right=#1*1.5*\clustersep of {#2.east|-first},rootB}, rrB/.style = {right=\rootsep of {#1.east|-first},rootB},
  rootC/.style = {diamond,scale=\rootscaleC,ball color=\rootcolorC},
    rcC/.style 2 args = {right=#1*1.5*\clustersep of {#2.east|-first},rootC}, rrC/.style = {right=\rootsep of {#1.east|-first},rootC},
  rootD/.style = {circle,scale=\rootscaleD,ball color=\rootcolorD},
    rcD/.style 2 args = {right=#1*1.5*\clustersep of {#2.east|-first},rootD}, rrD/.style = {right=\rootsep of {#1.east|-first},rootD},
  cluster/.style = {draw=black!90,thick,rounded corners,inner sep=22*\clustersep,outer xsep=22*\clustersep,fit=#1},
  clabel/.style  = {anchor=west,scale=\cdepthscale,black,inner sep=0,outer xsep=1,outer ysep=0},
  clabelL/.style = {above right=-\clustersep of #1t.north east,clabel},
  clabelD/.style = {below right=-\clustersep of #1t.south east,clabel},
  clouter/.style = {inner sep=0,outer sep=0,fit=#1}
}

\def\Cluster #1 = #2;{\node[cluster=#2] (#1) {};}
\def\ClusterL #1[#2] = #3;{
  \node[cluster=#3] (#1t) {}; \node[clabelL=#1] (#1l) {$#2$}; \node[clouter=(#1t)(#1l)] (#1) {};}
\def\ClusterD #1[#2] = #3;{
  \node[cluster=#3] (#1t) {}; \node[clabelD=#1] (#1d) {$#2$}; \node[clouter=(#1t)(#1d)] (#1) {};}
\def\ClusterLD #1[#2][#3] = #4;{
  \node[cluster=#4] (#1t) {}; \node[clabelL=#1] (#1l) {$#2$}; 
  \node[clabelD=#1] (#1d) {$#3$}; \node[clouter=(#1t)(#1l)(#1d)] (#1) {};}
\def\ClusterLDName #1[#2][#3][#4] = #5;{
  \node[cluster=#5] (#1t) {}; \node[clabelL=#1] (#1l) {$#2$}; 
  \node[clabelD=#1] (#1d) {$#3$}; 
  \node[scale=\cnamescale,above=\clustersep/3 of #1t,inner sep=0, outer sep=0] (#1n) {$#4$}; 
  \node[clouter=(#1l)(#1d)(#1t)] (#1) {};}

\newcommand{\Root}[4][]{
  \ifx\relax#2\relax\node[rr#1=#3] (#4) {};\else\node[rc#1={#2}{#3}] (#4) {};\fi}
\newcommand{\RootT}[5][]{
  \ifx\relax#2\relax\node[rrt#1=#3] (#4) {#5};\else\node[rct#1={#2}{#3}] (#4) {#5};\fi}

\def\frob(#1)(#2){\path[draw,thick,shorten <=-22*\clustersep,shorten >=-22*\clustersep](#1.east)--(#2.west|-#1){};}

\usepackage{pbox}
\def\pb#1{\pbox[c]{\textwidth}{\hfil #1\hfil}}

\long\def\clusterpicture#1\endclusterpicture{\pb{\vbox to \cltopskip{\vfill}\\%
  \begin{tikzpicture}\node[coordinate] (first) {};#1\end{tikzpicture}\\[-11pt]\vbox to \clbottomskip{\vfill}}}   
\long\def\clusterpictureopt#1#2\endclusterpicture{\pb{\vbox to \cltopskip{\vfill}\\%
  \begin{tikzpicture}[#1]\node[coordinate] (first) {};#2\end{tikzpicture}\\[-11pt]\vbox to \clbottomskip{\vfill}}}

\def\pb#1{\pbox[c]{\textwidth}{\hfil #1\hfil}}

\def\GraphVertices{\SetVertexNormal[Shape=circle, FillColor=blue!50, LineColor=blue!50, LineWidth=0.8pt]
  \tikzset{VertexStyle/.append style = {inner sep=0.5pt,minimum size=0.3em,font = \tiny\bfseries}}}

\def\BlueEdges{  \SetUpEdge[lw=0.8pt,color=blue!70]
   \tikzset{EdgeStyle/.append style = {shorten <=0.5pt,shorten >=0.5pt}}}

\def\LoopW(#1){
  \path[draw,-,thick,color=blue!70] (#1) edge[out=155,in=90] ($(#1)-(1.3,0)$);
  \path[draw,-,thick,color=blue!70] (#1) edge[out=210,in=270] ($(#1)-(1.3,0)$);
}
\def\LoopE(#1){
  \path[draw,-,thick,color=blue!70] (#1) edge[out=25,in=90] ($(#1)+(1.3,0)$);
  \path[draw,-,thick,color=blue!70] (#1) edge[out=-25,in=270] ($(#1)+(1.3,0)$);
}
\def\LoopS(#1){
  \path[draw,-,thick,color=blue!70] (#1) edge[out=115,in=180] ($(#1)+(0,1.2)$);
  \path[draw,-,thick,color=blue!70] (#1) edge[out=65,in=0] ($(#1)+(0,1.2)$);
}
\def\LoopN(#1){
  \path[draw,-,thick,color=blue!70] (#1) edge[out=-115,in=180] ($(#1)-(0,1.2)$);
  \path[draw,-,thick,color=blue!70] (#1) edge[out=-65,in=0] ($(#1)-(0,1.2)$);
}
\def\EdgeW(#1){
  \path[draw,-,thick,color=blue!70] (#1+) edge[out=180,in=90] ($(#1+)-(1.3,0.3)$);
  \path[draw,-,thick,color=blue!70] (#1-) edge[out=180,in=270] ($(#1+)-(1.3,0.3)$);
}
\def\EdgeE(#1){
  \path[draw,-,thick,color=blue!70] (#1+) edge[out=0,in=90] ($(#1-)+(1.3,0.3)$);
  \path[draw,-,thick,color=blue!70] (#1-) edge[out=0,in=270] ($(#1-)+(1.3,0.3)$);
}
\def\EdgeS(#1){
  \path[draw,-,thick,color=blue!70] (#1+) edge[out=90,in=0] ($(#1-)+(0.3,1.3)$);
  \path[draw,-,thick,color=blue!70] (#1-) edge[out=90,in=180] ($(#1-)+(0.3,1.3)$);
}
\def\EdgeN(#1){
  \path[draw,-,thick,color=blue!70] (#1+) edge[out=270,in=0] ($(#1+)-(0.3,1.3)$);
  \path[draw,-,thick,color=blue!70] (#1-) edge[out=270,in=180] ($(#1+)-(0.3,1.3)$);
}
\def\GCircle(#1,#2)(#3,#4){
  \path(#1,#2) node[coordinate] (1) {};
  \path(#3,#4) node[coordinate] (2) {};
  \path[draw,-,thick,color=blue!70] (1) edge[out=90,in=90] (2);
  \path[draw,-,thick,color=blue!70] (2) edge[out=270,in=270] (1);
}

\def\EdgeSign(#1)(#2)#3(#4)#5{
  \node at ($(#1)!#3!(#2) + (#4)$) [color=black, scale=\graphdslabelscale] {$\scriptstyle #5$};
}

\def\GraphEdgeSignDist{0.55}

\def\GraphEdgeSignS(#1)(#2)#3#4{\EdgeSign(#1)(#2)#3(0,-\GraphEdgeSignDist){#4}}

\def\VSwap#1#2#3#4{\path[draw](#1) edge[<->,#3,shorten >=#4pt,shorten <=#4pt] (#2){};}
\def\VArr#1#2#3#4{\path[draw](#1) edge[->,#3,shorten >=#4pt,shorten <=#4pt] (#2){};}

\def\ESwapOfs#1#2#3#4#5#6#7#8{\VSwap{$(#1)!0.5!(#2) + (#6)$}{$(#3)!0.5!(#4) + (#7)$}{#5}{#8}}

\def\EArrOfs#1#2#3#4#5#6#7#8{\VArr{$(#1)!0.5!(#2) + (#6)$}{$(#3)!0.5!(#4) + (#7)$}{#5}{#8}}

\def\tgrGB{\raise-7pt\hbox{\begin{tikzpicture}[scale=\GraphScale]
  \GraphVertices
  \Vertex[x=1.50,y=0.000,L=1]{1};
  \coordinate (2) at (0.000,0.000);
  \BlueEdges
  \LoopW(1)
\GraphEdgeSignS(1)(2){0.5}{n}\end{tikzpicture}}}

\def\tgrGBex{\raise-7pt\hbox{\begin{tikzpicture}[scale=\GraphScale]
  \GraphVertices
  \Vertex[x=1.50,y=0.000,L=1]{1};
  \coordinate (2) at (0.000,0.000);
  \BlueEdges
  \LoopW(1)
\GraphEdgeSignS(1)(2){0.5}{1}\end{tikzpicture}}}

\def\tgrGC{\raise-7pt\hbox{\begin{tikzpicture}[scale=\GraphScale]
  \GraphVertices
  \Vertex[x=1.50,y=0.000,L=1]{1};
  \coordinate (2) at (0.000,0.000);
  \BlueEdges
  \LoopW(1)
\GraphEdgeSignS(1)(2){0.5}{n}\ESwapOfs1212{}{0,-0.25}{0,0.25}{0.5}\end{tikzpicture}}}

\def\tgrGD{\raise-7pt\hbox{\begin{tikzpicture}[scale=\GraphScale]
  \GraphVertices
  \Vertex[x=1.50,y=0.000,L=\relax]{1};
  \coordinate (2) at (3.00,0.000);
  \coordinate (3) at (0.000,0.000);
  \BlueEdges
  \LoopE(1)
  \LoopW(1)
\GraphEdgeSignS(1)(3){0.5}{n}\GraphEdgeSignS(1)(2){0.5}{n}\end{tikzpicture}}}

\def\tgrGE{\raise-7pt\hbox{\begin{tikzpicture}[scale=\GraphScale]
  \GraphVertices
  \Vertex[x=1.50,y=0.000,L=\relax]{1};
  \coordinate (2) at (3.00,0.000);
  \coordinate (3) at (0.000,0.000);
  \BlueEdges
  \LoopE(1)
  \LoopW(1)
\GraphEdgeSignS(1)(3){0.5}{n}\GraphEdgeSignS(1)(2){0.5}{n}\ESwapOfs1212{}{0,-0.25}{0,0.25}{0.5}\end{tikzpicture}}}

\def\tgrGF{\raise-7pt\hbox{\begin{tikzpicture}[scale=\GraphScale]
  \GraphVertices
  \Vertex[x=1.50,y=0.000,L=\relax]{1};
  \coordinate (2) at (3.00,0.000);
  \coordinate (3) at (0.000,0.000);
  \BlueEdges
  \LoopE(1)
  \LoopW(1)
\GraphEdgeSignS(1)(3){0.5}{n}\GraphEdgeSignS(1)(2){0.5}{n}\ESwapOfs1313{}{0,-0.25}{0,0.25}{0.5}\ESwapOfs1212{}{0,-0.25}{0,0.25}{0.5}\end{tikzpicture}}}

\def\tgrGG{\raise-7pt\hbox{\begin{tikzpicture}[scale=\GraphScale]
  \GraphVertices
  \Vertex[x=1.50,y=0.000,L=\relax]{1};
  \coordinate (2) at (3.00,0.000);
  \coordinate (3) at (0.000,0.000);
  \BlueEdges
  \LoopE(1)
  \LoopW(1)
\GraphEdgeSignS(1)(3){0.5}{n}\GraphEdgeSignS(1)(2){0.5}{n}\ESwapOfs1312{in=160,out=20}{0.2,0.3}{-0.2,0.3}{0.5}\end{tikzpicture}}}

\def\tgrGH{\raise-7pt\hbox{\begin{tikzpicture}[scale=\GraphScale]
  \GraphVertices
  \Vertex[x=1.50,y=0.000,L=\relax]{1};
  \coordinate (2) at (3.00,0.000);
  \coordinate (3) at (0.000,0.000);
  \BlueEdges
  \LoopE(1)
  \LoopW(1)
\GraphEdgeSignS(1)(3){0.5}{n}\GraphEdgeSignS(1)(2){0.5}{n}\EArrOfs1312{in=150,out=30}{0.1,0.29}{0,0.35}{0.5}\EArrOfs1213{in=-60,out=-60}{0,0.2}{0.3,-0.25}{0.5}\end{tikzpicture}}}

\def\tgrGA{\raise-3pt\hbox{\begin{tikzpicture}[scale=\GraphScale]
  \GraphVertices
  \Vertex[x=0.000,y=0.000,L=2]{1};
  \BlueEdges
\end{tikzpicture}}}

\makeatletter
\tikzset{join/.code=\tikzset{after node path={%
\ifx\tikzchainprevious\pgfutil@empty\else(\tikzchainprevious)%
edge[every join]#1(\tikzchaincurrent)\fi}}}
\makeatother

\tikzset{>=stealth',every on chain/.append style={join},
         every join/.style={->}}

\DeclareSymbolFont{cyrletters}{OT2}{wncyr}{m}{n}
\DeclareMathSymbol{\Sha}{\mathalpha}{cyrletters}{"58}

\mathtoolsset{showonlyrefs}

\pgfplotsset{compat=1.18}

\makeatletter
\let\@wraptoccontribs\wraptoccontribs
\makeatother

\title{Invariants recovering the reduction type of a hyperelliptic curve}
\author{Lilybelle Cowland Kellock}
\contrib[with an appendix by]{Elisa Lorenzo García}
\address{University College London, London WC1H 0AY, UK}
\email{lilybelle.kellock.20@ucl.ac.uk}

\begin{document}

\begin{abstract}
Tate's algorithm tells us that for an elliptic curve $E$ over a local field $K$ of residue characteristic $\geq 5$, $E/K$ has potentially good reduction if and only if $\textup{ord}(j_E)\geq 0$. It also tells us that when $E/K$ is semistable the dual graph of the special fibre of the minimal regular model of $E/K^{\textup{unr}}$ can be recovered from $\textup{ord}(j_E)$. We generalise these results to hyperelliptic curves of genus $g\geq 2$ over local fields of odd residue characteristic $K$ by defining a list of absolute invariants that determine the potential stable model of a genus $g$ hyperelliptic curve $C$. They also determine the dual graph of the special fibre of the minimal regular model of $C/K^{\textup{unr}}$ if $C/K$ is semistable. This list depends only on the genus of $C$, and the absolute invariants can be written in terms of the coefficients of a Weierstrass equation for $C$. We explicitly describe the method by which the valuations of the invariants recover the dual graphs. Additionally, we show by way of a counterexample that if $g \geq 2$, there is no list of invariants whose valuations determine the dual graph of the special fibre of the minimal regular model of a genus $g$ hyperelliptic curve $C$ over a local field $K$ of odd residue characteristic when $C$ is not assumed to be semistable.
\end{abstract}

\maketitle

\tableofcontents

\section{Introduction}
Reduction types of curves are used to classify and understand curves with bad reduction over local fields. As for Kodaira types of elliptic curves, the object of concern in the study of reduction types is the dual graph of the special fibre of a model of the curve. The dual graph represents how curves in the special fibre intersect and their multiplicities, and is required to calculate many local quantities associated to a curve with bad reduction, for example Tamagawa numbers (see \cite{boschliu} Theorem 1.11). In this paper, we address the question of how the dual graph of the special fibre of the minimal regular model of a semistable hyperelliptic curve $C$ over a local field $K$ of odd residue characteristic can be obtained directly from absolute invariants of $C$. This generalises a corollary of Tate's algorithm \cite{tatealgorithm}, that when $E/K$ is known to be a semistable elliptic curve its Kodaira type can be read off from the valuation of $j_E$ (see \cite{advancedsilverman} p.365). We also show that if $C$ is a genus $g \geq 2$ hyperelliptic curve over a local field $K$ that is not assumed to be semistable, there is no list of invariants whose valuations determine the dual graph of the special fibre of the minimal regular model of $C/K^{\textup{unr}}$. This shows that there is no analogue in the setting of hyperelliptic curves to the fact that the dual graph of the special fibre of the minimal regular model of an elliptic curve $E/K^{\textup{unr}}$ can be read off from the valuation of $\Delta_E$ and $j_E$ (see \cite{advancedsilverman} p.365).

In this article, we consider \textit{absolute invariants} of hyperelliptic curves $C$ of genus $g\geq 2$ over local fields $K$. A polynomial $I\in\mathbb{Z}[a_0,\dots,a_{2g+2}]$ in the coefficients of a binary form $f(X,Z)$ is a weight $k$ \textup{invariant} of genus $g$ hyperelliptic curves if for all $M\in \textup{GL}_2(\Bar{K})$, one has $I(f(M\cdot (X,Z))) = \det(M)^k I(f(X,Z))$ (see \cite{hilbert} for an introduction to invariants of binary forms). An \textit{absolute invariant} $A\in\mathbb{Z}(a_0,\dots,a_{2g+2})$ is a quotient of invariants of the same weight, so that $A(f(X,Z))=A(f(M\cdot (X,Z)))$ and they are invariant in geometric isomorphism classes of hyperelliptic curves. For this reason, we write $A(C)$ for $A(f(X,Z))$, where $C:y^2=f(x)$ is any Weierstrass equation for $C$ and $f(X,Z)$ is the associated binary form.

Throughout, we assume that the genus of $C$ is at least $2$, since this is necessary to use the results of \cite{m2d2} on recovering the dual graph of the special fibre of the minimal regular model. We restrict ourselves to local fields for the same reason, but it is believed by the author that the construction works for hyperelliptic curves over general discretely valued fields that have perfect residue fields with characteristic $\neq 2$. The main result that we prove is the following theorem, where the list of absolute invariants is explicitly described in \S\ref{mainthminvproof}. The list depends only on the genus $g$ of $C$ and does not depend on $K$, the field of definition of $C$, as long as $K$ has odd residue characteristic. 

\begin{theorem}[=Corollary \ref{invthm}]\label{1.1}
There exists an explicit finite list of absolute invariants $A_{g}^{(1)},\dots,A_{g}^{(n_g)}$ for which the valuations $\textup{ord}(A_{g}^{(i)}(C))$ for $i=1,\dots,n_g$, when evaluated on a hyperelliptic curve $C$ of genus $g\geq 2$ over a local field $K$ of odd residue characteristic, uniquely determine:
\begin{enumerate}[(i)]
    \item The dual graph of the special fibre of the minimal regular model of $C/K^{\textup{unr}}$ if $C/K$ is semistable;
    \item The dual graph of the special fibre of the potential stable model of $C$ if $C/K$ is not semistable. 
\end{enumerate}
\end{theorem}

Consider the classical setting of an elliptic curve $E$ over a local field $K$. By Tate's algorithm (see \cite{advancedsilverman} p.365), when $E/K$ has multiplicative reduction the number of components in the special fibre is given by $-\textup{ord}(j_E)$, and $E/K$ has potentially good reduction if and only if $\textup{ord}(j_E)\geq 0$. Theorem \ref{1.1} generalises these two statements to hyperelliptic curves in $(i)$ and $(ii)$ respectively. When the methods of this paper for obtaining absolute invariants are applied to an elliptic curve $E$ over a local field $K$, one obtains a single absolute invariant $I_{1}^{1}=\frac{j_E}{16}-3$ (see Example \ref{ellipticcurvesex}). Since we assume throughout that $K$ has odd residue characteristic, $\textup{ord}(I_{1}^{1})\geq 0$ if and only if $\textup{ord}(j_E)\geq 0$, and if $\textup{ord}(I_{1}^{1})< 0$ then $\textup{ord}(I_{1}^{1})=\textup{ord}(j_E)$, thus the methods of this paper `recover' these results for elliptic curves.

If an elliptic curve $E/K$ is not assumed to have multiplicative reduction and $K$ has residue characteristic $\geq 5$, the dual graph of the special fibre of the minimal regular model can be recovered from the valuation of $j_E$ and $\Delta_E$, which are both invariants of the curve (again, see \cite{advancedsilverman} p.365). We prove the following theorem, which shows that there is no analogue to this fact in the setting of hyperelliptic curves. The theorem implies that if $C$ is a genus $g \geq 2$ hyperelliptic curve over a local field $K$ that is not assumed to be semistable, there is no list of invariants whose valuations determine the dual graph of the special fibre of the minimal regular model of $C/K^{\textup{unr}}$.

\begin{theorem}[=Theorem \ref{nosilverman}]\label{secondtheorem}
Let $K$ be a local field of odd residue characteristic. For $g\geq 2$ there exists a semistable hyperelliptic curve $C:y^2=f_1(x)$ of genus $g$ defined over $K$ and a non-semistable hyperelliptic curve $C':y^2=f_2(x)$ of genus $g$ defined over $K$ for which $A(C)=A(C')$ for every absolute invariant $A$ and $\textup{ord}(I(f_1))=\textup{ord}(I(f_2))$ for every invariant $I$ of weight $k\in\mathbb{Z}^+$, but the dual graph of the special fibre of the minimal regular models of $C/K^{\textup{unr}}$ and $C'/K^{\textup{unr}}$ do not coincide.
\end{theorem}

We wish to highlight that if one is not restricted to using invariants of of the curve, it is possible to recover the dual graph of the special fibre of the minimal strict normal crossings model of a hyperelliptic curve $C$ over a local field $K$ of odd residue characteristic from polynomials in the coefficients of a Weierstrass equation for the curve (see \cite{lilybellepolys} Theorem 1.11).

There have been many previous works on recovering the reduction type of curves from invariants. It was described by Mestre in \cite{mestre} and Liu in \cite{liu} how the dual graph of the special fibre of the potential stable model of a genus $2$ curve can be recovered from the Igusa--Clebsch invariants defined by Igusa in \cite{igusa} and Clebsch in \cite{clebsch}. For genus $3$ hyperelliptic curves, there is a list of invariants describing their isomorphism classes given by Shioda in \cite{shioda} and Tsuyumine in \cite{tsuyumine}. In \cite{elisa}, it is shown that Shioda invariants can be expressed in terms of differences of roots of a Weierstrass equation and that this has applications to studying the reduction type of the curve, which is an approach that we take in this paper. There is also a paper \cite{helminck} that uses tropical invariants to recover the Berkovich skeleta of superelliptic curves. In terms of the general study of dual graphs of special fibres of curves and stable models, a similar construction to the `stable model tree' defined in this paper was used by Bosch in \cite{bosch} to determine the stable type of hyperelliptic curves. There are also the papers \cite{pink1} and \cite{pink2} for calculating the stable reduction type in residue characteristic $\neq 2$ and $2$ respectively. For computing dual graphs, there is a Magma \cite{magma} package based on the papers \cite{tim1}, \cite{tim2} and \cite{simone}.

We use the machinery of cluster pictures introduced in \cite{m2d2} in order to prove the results in this paper. Let $K$ be a local field of odd residue characteristic and let $C/K$ be a hyperelliptic curve given by a Weierstrass equation $C:y^2 = f(x)$ of genus $g$. Write $\mathcal{R}$ for the set of roots of $f(x)$ in $\Bar{K}$, $d=\deg(f)=|\mathcal{R}|\in \{2g + 1, 2g + 2\}$ and $c_f$ for the leading coefficient of $f$ so that
\begin{equation}
    C:y^2=f(x)=c_f\prod_{r\in \mathcal{R}} (x-r).
\end{equation}
To this Weierstrass equation for $C$, one associates its \textit{cluster picture}, a pictorial object encoding the $\pi$-adic distances between the roots of $f(x)$, where $\pi$ is a uniformiser of $K$. Cluster pictures are now a classical approach to studying the arithmetic of hyperelliptic curves over local fields. The main result of \cite{m2d2} that is of relevance to this article is that if $C/K$ is semistable, the dual graph of the special fibre of the minimal regular model of $C/K^{\textup{unr}}$ can be recovered from the cluster picture (see \cite{m2d2} Theorem 8.5). 

\begin{example}\label{jinvariant}
To demonstrate the connection between the $\pi$-adic distances between the roots of $f(x)$ and the reduction type, note that the $j$-invariant of an elliptic curve $E:y^2=f(x)$ can be written as 
\begin{equation}
j_E=-8\cdot \frac{\left((x_1-x_2)^2+(x_1-x_3)^2+(x_2-x_3)^2\right)^3}{(x_1-x_2)^2(x_1-x_3)^2(x_2-x_3)^2},
\end{equation}
where $x_1$, $x_2$ and $x_3$ are the roots of $f(x)$ over $\Bar{K}$. Since $j_E$ is invariant under a change of model, we can assume that the Weierstrass equation is minimal, and so $f(x)$ has a repeated root of multiplicity $2$ mod $\pi$ (i.e. $E/K$ has multiplicative reduction) if and only if $\textup{ord}(j_E)<0$. 
\end{example}

We define absolute invariants that generalise the fact that the $j$-invariant detects how $\pi$-adically close the roots of $f(x)$ are. We concoct absolute invariants that encode the distances between the roots of $f(x)$, where $f(x)$ is any even degree polynomial. A complication of this is that the set of distances between roots depends on the choice of model of $C$, so instead of studying the cluster picture directly we study the \textit{stable model tree} (see Definition \ref{stablemodeltreedef}) which is defined in terms of the cluster picture but is model independent; this is similar to the tree studied in \cite{srinivasan}. From this we can use the results of \cite{m2d2} to recover the relevant dual graphs (see \S\ref{dualgraphstreesection}), since the stable model tree encodes the same information as the BY tree defined in \cite{m2d2} which is used to recover the dual graph. The absolute invariants that recover the stable model tree are constructed by studying the possibilities for the stable model tree and the possible orderings on the distances between their vertices (see \S\ref{invariantssection}).

The absolute invariants constructed in this paper can be written as rational functions of elementary symmetric polynomials in the variables $X_1,\dots,X_d$, and so when evaluated on the roots of a polynomial $f(x)$ over a local field $K$ they can be written in terms of the Weierstrass coefficients of $f(x)$. This means that in order to calculate the absolute invariants one does not need to know the roots of $f(x)$ over $\Bar{K}$ a priori, which could be defined over large extensions. Alternatively, if starting with a hyperelliptic curve $C:y^2=f(x)$ defined over a number field, $f(x)$ can be efficiently factorised over $\mathbb{C}$ and the invariants evaluated on the roots of $f(x)$ over the global field. 

We define the absolute invariants in this paper for a hyperelliptic curve given by a Weierstrass equation of even degree, but if $C:y^2=f(x)$ where $d=\deg(f)$ is odd and $x_1,\dots,x_d$ are the roots of $f(x)$ over $\Bar{K}$, we can evaluate the absolute invariants letting the variable $X_{d+1}$ go to infinity. One can check that this is consistent with evaluating the absolute invariants on a model of this curve that has even degree.

\subsection{Layout of the paper}

The paper is laid out as follows. 

In \S\ref{notation}, we list the notation and graph-theoretic terminology used throughout. In \S\ref{clusterpics}, we give the background involving cluster pictures that will be needed to prove the main results, and we prove Theorem \ref{secondtheorem} which tells us that for hyperelliptic curves of genus $\geq 2$, one needs more information than the valuation of invariants in order to recover the dual graph of the special fibre of the minimal regular model. 

In \S\ref{dualgraphstreesection}, we define an object called the `stable model tree' (see Definition \ref{stablemodeltreedef}) and explain how the dual graph of the special fibre of the minimal regular model of a semistable hyperelliptic curve can be read off from it using the results of \cite{m2d2}.

In \S\ref{invariantssection}, we define the absolute invariants associated to a stable model tree with an ordering on the distances between the vertices (see Definition \ref{invariantsdefinition}). In \S\ref{invariantsofcurves} and \S\ref{comparingvaluations}, we prove results on the valuations of the absolute invariants from \S\ref{invariantssection} when evaluated on a hyperelliptic curve, including a result that compares the valuations of the absolute invariants associated to different possibilities for the stable model tree. 

In \S\ref{mainthminvproof}, we prove the main theorem of this paper, Theorem \ref{wholealgorithm}. This theorem describes how the stable model tree (and thus the dual graph) is uniquely determined inductively by comparing the valuations of the absolute invariants defined in \S\ref{invariantssection}. We give an explicit description of the list of absolute invariants needed to recover the dual graph of the special fibre of the relevant models of $C/K^{\textup{unr}}$ for $C$ of fixed genus. The description is given in terms of the possible stable model trees and the possible orderings on the distances between the vertices in the stable model tree (see Corollary \ref{invthm}).

In \S\ref{genus2section}, we give an explicit description of the absolute invariants that recover the relevant dual graphs for genus $2$ curves. We give an example using the methods of this paper to recover the dual graph of the special fibre of the minimal regular model of a semistable genus $2$ curve. We also explicitly write down the genus $2$ absolute invariants defined in this paper, and give a table that describes the lengths of the chains in the dual graph of a semistable genus $2$ curve in terms of the valuations of these absolute invariants.

\subsection{Notation and terminology}\label{notation}
We will use the following notation.
\begin{align}
& K && \text{ a local field of odd residue characteristic}; \\
& \textup{ord} && \text{ the valuation with respect to a uniformiser of } K; \\
& \Bar{K} && \text{ the algebraic closure of } K; \\
& K^{\textup{unr}} && \text{ the maximal unramified extension of } K; \\
& f(x) && \text{ an even degree polynomial with coefficients in }K; \\
& C && \text{ a hyperelliptic curve over $K$ given by a Weierstrass equation $C:y^2=f(x)$}; \\
& \Delta_C && \text{ the discriminant of $C$, see Definition \ref{discriminant}}; \\ 
& T_C && \text{ the stable model tree of $C$, see Definition \ref{stablemodeltreedef}};\\
& (T,I) && \text{ a possible stable model tree $T$ with ordering $I$, see Definition \ref{setofleafgraphs};} \\
& (T_C,I_C) && \text{ the stable model tree $T_C$ with ordering $I_C$, see Definition \ref{hyperellipticordering};} \\
& \delta(v,w) && \text{ the distance between two vertices $v$ and $w$ in a weighted tree, see Definition \ref{distancetrees}}; \\
& \delta(C_{ij},C_{kl}) && \text{ the distance between two paths $C_{ij}$ and $C_{kl}$ in a weighted tree, see Definition \ref{distancepaths}}; \\
& \delta(\mathfrak{s},\mathfrak{t}) && \text{ the distance between two clusters, see Definition \ref{delta}}; \\
& \delta_n(C) && \text{ the $n$-th largest distance in $T_C$, see Definition \ref{treedefinitions}}; \\
& \delta_{n}(T,I) && \text{ the $n$-th largest distance in $T$ with ordering $I$, see Definition \ref{stablemodeltreedefs}.} \\
& K_n(C) && \text{ see Definition \ref{treedefinitions}};\\
& K_{n}(T,I) && \text{ see Definition \ref{bn+1}}; \\
& S_{n}(T,I) && \text{ a set of `pairs of pairs' of singletons in $(T,I)$, see Definition \ref{stablemodeltreedefs}}; \\
& \textup{Inv}_{T,I,n} && \text{ the $n$-th absolute invariant associated to $(T,I)$, see Definition \ref{invariantsdefinition}}; \\
& B_{n+1}(T,I,C) && \text{ the `averaging function' at the $(n+1)$-st step, see Definition \ref{bn+1}}; \\
& \mathcal{T}_{n+1}(C) && \text{ the set of possible stable models trees at the $(n+1)$-st step, see Definition \ref{possiblen+1st};} \\
& \mathbf{T}_d && \text{ the set of possible stable model trees with orderings for genus $g$, see Definition \ref{setofleafgraphs}.} \\
\end{align}

Where $C:y^2=f(x)$ is a Weierstrass equation for $C/K$, we will assume that $f(x)$ has even degree. We will sometimes use the term `dual graph' to mean the dual graph of the special fibre of the minimal regular model of a semistable hyperelliptic curve, or the dual graph of the special fibre of the potential stable model of a non-semistable hyperelliptic curve. We adopt the convention that $v(0)=\infty$. 

\begin{definition}[As in \cite{liudisc}]\label{discriminant}
For a hyperelliptic curve of genus $g$ given by a Weierstrass equation $C:y^2=f(x)$, define 
\begin{equation}
\Delta_C=16^g c_f^{4g+2}\textup{disc}\left(\frac{1}{c_f}f(x)\right),
\end{equation}
where $c_f$ is the leading coefficient of $f$. 
\end{definition}

In this paper we will consider unweighted trees $T=(V,E)$ and weighted trees $T=(V,E,L)$, where $L:E\rightarrow \mathbb{Q}^{+}$. All graphs will be considered as unlabelled unless otherwise stated, i.e. $G=(\{v_1,\dots,v_d\},\{v_1v_2\})$ is the same graph as $G'=(\{v_1,\dots,v_d\},\{v_1v_3\})$. We will need the following graph-theoretic definitions.

\begin{definition}
For a tree $T$, we call the vertices of degree $1$ in $T$ \textit{leaves}. We call two edges \textit{adjacent} if they are connected to a common vertex. We call two vertices \textit{adjacent} if there is an edge between them.
\end{definition}

\begin{definition}\label{bipartite}
Define $K_{1,d}=(V,E)$, where $V=\{v,v_1\dots,v_{d}\}$ and $E=\{vv_1,\dots,vv_d\}$.
\end{definition}

\begin{definition}\label{distancetrees}
Let $T=(V,E,L)$ be a weighted tree where $L:E\rightarrow \mathbb{Q}^{+}$ is a length function on the edges. For $v,w\in E$, we call the unique path between $v$ and $w$ with the smallest number of edges in $T$ the \textit{shortest path} between $v$ and $w$ and denote it by $P_{vw}$. Define $\delta(v,w)=\sum_{e\in P_{vw}} L(e)$.
\end{definition}

\begin{definition}\label{distancepaths}
Let $s_i$, $s_j$, $s_k$ and $s_l$ be four distinct leaves in a weighted tree $T=(V,E,L)$. Denote by $C_{ij}$ the path in $T$ between $s_i$ and $s_j$ and by $C_{kl}$ the path between $s_k$ and $s_l$. Since $T$ is a tree, if $C_{ij}$ and $C_{kl}$ do not intersect, there is a unique shortest path between $C_{ij}$ and $C_{kl}$ in $T$ that goes between vertex $v$ in $C_{ij}$ and vertex $w$ in $C_{kl}$. Define 
\begin{equation}
    \delta(C_{ij},C_{kl})=\delta(v,w). 
\end{equation}
\end{definition}

\begin{example}
Suppose we are given the tree $T$ below, where the edges extending to the leaves do not have an assigned length. We consider such a tree since this is what the stable model trees defined in \S\ref{dualgraphstreesection} look like. We have drawn the paths $C_{12}$ and $C_{35}$ by making the edges on those paths bold. 
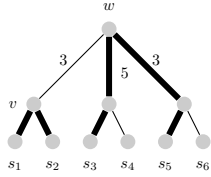
\begin{center}
		\begin{figure}[H]
			\begin{tikzpicture}
				[scale=0.5, auto=left,every node/.style={circle,fill=black!20,scale=0.6}]
                \node[label=above: {$w$}] (n10) at (2,2) {};
				\node[label=left: {$v$}] (n1) at (0,0) {};
				\node (n2) at (2,0)  {};
				\node (n3) at (4,0)  {};

                \node[label=below: {$s_1$}] (n4) at (-0.5,-1) {};
                \node[label=below: {$s_2$}] (n5) at (0.5,-1) {};
                
                \node[label=below: {$s_3$}] (n6) at (1.5,-1) {};
                \node[label=below: {$s_4$}] (n7) at (2.5,-1) {};

                \node[label=below: {$s_5$}] (n8) at (3.5,-1) {};
                \node[label=below: {$s_6$}] (n9) at (4.5,-1) {};

                \draw[line width=0.8mm] (n1) -- (n4); 
                \draw[line width=0.8mm] (n1) -- (n5); 
                
                \draw[line width=0.8mm] (n2) -- (n6); 
                \draw (n2) -- (n7); 

                \draw[line width=0.8mm] (n3) -- (n8); 
                \draw (n3) -- (n9);
				
				\draw (n10) -- (n1) node [text=black, pos=0.4, left, fill=none] {$3$};
                \draw[line width=0.8mm] (n10) -- (n2) node [text=black, pos=0.6, right,fill=none] {$5$};
				\draw[line width=0.8mm] (n10) -- (n3) node [text=black, pos=0.4, right,fill=none] {$3$};
			\end{tikzpicture}
		\end{figure}
\end{center}
The unique path between $C_{12}$ and $C_{35}$ is between $v$ and $w$, hence $\delta(C_{12},C_{35})=\delta(v,w)=3$.
\end{example}

\addtocontents{toc}{\SkipTocEntry}
	\section*{Acknowledgements} 

The author would like to thank Elisa Lorenzo García for posing the problems considered in this paper at the Seminari de Teoria de Nombres de Barcelona in 2022 and for numerous informative conversations. She would also like to thank Vladimir Dokchitser for the generosity of his support and guidance and for many useful discussions, and Holly Green for proofreading the manuscript. The author was supported by the Engineering and Physical Sciences Research Council [EP/L015234/1], the EPSRC Centre for Doctoral Training in Geometry and Number Theory (The London School of Geometry and Number Theory), University College London.

\section{Cluster pictures and reduction types}\label{clusterpics}
We will now give a brief introduction to cluster pictures of hyperelliptic curves, which were introduced in \cite{m2d2}. We also prove Theorem \ref{secondtheorem}, which tells us that for hyperelliptic curves of genus $\geq 2$, one needs more information than the valuation of invariants of the curve in order to recover the dual graph of the special fibre of the minimal regular model. 

Let $C$ be a hyperelliptic curve of genus $g$ over a complete discretely valued field $K$ given by a Weierstrass equation
\begin{equation}
    C:y^2=f(x)=c_f(x-x_1)\cdots(x-x_d),
\end{equation}
where $d=2g+1$ or $d=2g+2$. Let $\mathcal{R}=\{x_1,\dots,x_d\}$ denote the set of roots of $f(x)$ in $\Bar{K}$.  

\begin{definition}[Cluster, from \cite{m2d2} Definition 1.1]
A \textit{cluster} is a non-empty subset $\mathfrak{s}\subseteq\mathcal{R}$ of the form $\mathfrak{s}=D\cap\mathcal{R}$ for some disc $D = \{x \in \Bar{K} \mid \textup{ord}(x- z) \geq d\}$ for some $z \in \Bar{K}$ and $d \in \mathbb{Q}$.
\end{definition}

\begin{definition}[Cluster picture, from \cite{m2d2} Definitions 1.1 and 1.5]
For a cluster $\mathfrak{s}$ with $|\mathfrak{s}| > 1$, its depth $d_{\mathfrak{s}}$ is the maximal $d$ for which $\mathfrak{s}$ is cut out by such a disc $D$ as above. That is,
\begin{equation}
    d_\mathfrak{s} = \text{{min}}_{r,r'\in\mathfrak{s}} \textup{ord}(r-r').
\end{equation}
If $\mathfrak{s}\neq \mathcal{R}$, then its relative
depth is $\delta_\mathfrak{s}=d_\mathfrak{s}-d_{P(\mathfrak{s})}$, where $P(\mathfrak{s})$ is the smallest cluster with $\mathfrak{s}\subsetneq P(\mathfrak{s})$. We refer to this data of the clusters and relative depths as the \textit{cluster picture} of $C$. The clusters of size $1$ are called \textit{singletons}.
\end{definition}

\begin{definition}[From \cite{m2d2} Definition 1.3]
If $\mathfrak{s}\subsetneq\mathfrak{r}$ is a maximal subcluster, we call $\mathfrak{r}$ the \textit{parent} of $\mathfrak{s}$. We call $\mathfrak{s}$ a \textit{child} of $\mathfrak{r}$. 
\end{definition}

\begin{definition}[From \cite{m2d2} Definition 1.4]
If $\mathfrak{s}$ has size $>1$ it is \textit{proper}. If $\mathfrak{s}$ has odd/even size we call it \textit{odd/even}. If $\mathfrak{s}$ is proper, even and has only even children it is \textit{übereven}.
\end{definition}

\begin{definition}[From \cite{semistable} Definition 3.45 and \cite{m2d2} Example D.2.]\label{delta}
Let $\mathcal{C}$ be a cluster picture and let $\Sigma$ be the set of clusters of $\mathcal{C}$. We can define a distance function $\delta:\Sigma\times\Sigma\rightarrow\mathbb{Q}$ between pairs of clusters of size $>1$ as follows. Let $\mathfrak{s}$ and $\mathfrak{r}\in\Sigma$.
\begin{enumerate}[(i)]
    \item $\delta(\mathfrak{s},\mathfrak{s})=0$;
    \item If $\mathfrak{s}\subseteq \mathfrak{r}$ then $\delta(\mathfrak{s},\mathfrak{r})=d_{\mathfrak{s}}-d_{\mathfrak{r}}$;
    \item If $\mathfrak{u}$ is the least common ancestor of $\mathfrak{s}$ and $\mathfrak{r}$, then $\delta(\mathfrak{s},\mathfrak{r})=\delta(\mathfrak{u},\mathfrak{s})+\delta(\mathfrak{u},\mathfrak{r})=d_\mathfrak{s}-d_\mathfrak{u}+d_\mathfrak{r}-d_\mathfrak{u}$,
\end{enumerate}
where the least common ancestor of $\mathfrak{s}$ and $\mathfrak{r}$ is the unique cluster $\mathfrak{u}$ such that $\mathfrak{s},\mathfrak{r}\subset\mathfrak{u}$ and no child of $\mathfrak{u}$ contains both $\mathfrak{s}$ and $\mathfrak{r}$.
\end{definition}

\begin{definition}[BY tree, as in \cite{m2d2} Definition D.6]\label{bytreedefinition}
Let $\Sigma$ be a cluster picture. We define $T_{\Sigma}$, the BY
tree associated to $\Sigma$, as follows. First take the graph with:
\begin{itemize}
\item A vertex $v_{\mathfrak{s}}$ for every proper cluster $\mathfrak{s}$, excluding $\mathfrak{s} = \mathcal{R}$ when $\mathcal{R} = 2g + 2$ and has a child of size $2g + 1$, coloured yellow if $\mathfrak{s}$ is übereven and blue otherwise;
\item An edge linking $v_{\mathfrak{s}}$ to $v_{P(\mathfrak{s})}$ for every proper cluster $\mathfrak{s}\neq \mathcal{R}$, yellow of length $2\delta_{\mathfrak{s}}$ if $\mathfrak{s}$ is even, and blue of length $\delta_{\mathfrak{s}}$ if $\mathfrak{s}$ is odd.
\end{itemize}
To obtain $T_\Sigma$ from this graph we remove, if $\mathcal{R} = 2g + 2$ and $\mathcal{R}$ is a disjoint union of two proper children, the degree $2$ vertex $v_{\mathcal{R}}$ from the vertex set (keeping the underlying
topological space the same). We define the genus of a vertex $v_{\mathfrak{s}}$ as $g(v_{\mathfrak{s}}) = g(\mathfrak{s})$, where $\#\{\textup{odd children of }\mathfrak{s}\}=2g(\mathfrak{s})+1$ or $2g(\mathfrak{s})+2$.
\end{definition}

Theorem 5.18 of \cite{m2d2} states that the dual graph of the special fibre of the minimal regular model of a semistable hyperelliptic curve $C$ over a local field $K$ of odd residue characteristic is determined by the BY tree of $C/K$. Hence, it is the information contained in the BY tree that we want to capture from absolute invariants of the curve in order to obtain the dual graph from absolute invariants.

The table on p.365 of \cite{advancedsilverman} shows that when $E$ is an elliptic curve over a discretely valued field $K$ of residue characteristic $p\geq 5$, the dual graph of the special fibre of the minimal regular model of $E/K^{\textup{unr}}$ can be read off from $\textup{ord}(j_E)$ and $\textup{ord}(\Delta_E)$. In the theorem below, we use cluster pictures and Theorem 5.18 of \cite{m2d2} to show that for a hyperelliptic curve $C$ of genus $g\geq 2$ over a local field $K$, it is not sufficient to know the valuations of invariants of the curve in order to recover the dual graph of the special fibre of the minimal regular model of $C/K^{\textup{unr}}$. 

\begin{theorem}\label{nosilverman}
Let $K$ be a local field of odd residue characteristic. For $g\geq 2$ there exists a semistable hyperelliptic curve $C:y^2=f_1(x)$ of genus $g$ defined over $K$ and a non-semistable hyperelliptic curve $C':y^2=f_2(x)$ of genus $g$ defined over $K$ such that $\textup{ord}(\Delta_{C})=\textup{ord}(\Delta_{C'})$ and $A(C)=A(C')$ for every absolute invariant $A$. In particular, $\textup{ord}(I(f_1))=\textup{ord}(I(f_2))$ for every invariant $I$ of weight $k\in\mathbb{Z}^+$, but the dual graph of the special fibre of the minimal regular models of $C/K^{\textup{unr}}$ and $C'/K^{\textup{unr}}$ do not coincide.
\end{theorem}

\begin{proof}
Let $d=2g+2$ and denote by $\pi$ a uniformiser of $K$. We split the proof into three cases. 

\textit{Case $1$: $g$ is even.} Then $d/2$ is odd and the curves 
\begin{align}
C&:y^2= f_1(x)=\prod_{i=0}^{\frac{d}{2}-1}(x-\frac{\pi}{1+\pi^{2i+1}})\cdot(x+1)\cdot \prod_{j=1}^{\frac{d}{2}-1}(x-\frac{1}{1+\pi^{2j}})\quad \textup{ and } \\
C'&:y^2=f_2(x)=\prod_{i=0}^{\frac{d}{2}-1} (x-1-\pi^{2i+1})\cdot (x+\pi)\cdot \prod_{j=1}^{\frac{d}{2}-1} (x-\pi-\pi^{2j+1})
\end{align}
are defined over $K$ and are isomorphic over $\Bar{K}$ by taking $x\mapsto \pi/x$ and $y\mapsto \frac{(\alpha_1\cdots\alpha_d)^{\frac{1}{2}}}{x^{\frac{d}{2}}}y$ in the second Weierstrass equation, where $\alpha_1,\dots,\alpha_d$ are the roots of $f_2(x)$. By checking the valuations of the differences of the roots, the Weierstrass equations for these curves have the following cluster pictures
\begin{center}
\begin{minipage}[b]{0.4\textwidth}
		\begin{center}
		\begin{figure}[H]
\scalebox{1.5}{
\clusterpicture             
  \Root[D] {1} {first} {r1};
  \Root[D] {} {r1} {r2};
  \Root[Dot] {2} {r2} {r3};
  \Root[Dot] {} {r3} {r4};
    \Root[Dot] {} {r4} {r5};
    \Root[D] {1.5} {r5} {r6};
    \Root[D] {1} {r6} {r7};
    \Root[D] {5} {r7} {r8};
    \Root[D] {} {r8} {r9};
    \Root[Dot] {2} {r9} {r10};
    \Root[Dot] {} {r10} {r11};
    \Root[Dot] {} {r11} {r12};
    \Root[D] {1.5} {r12} {r13};
    \Root[D] {1} {r13} {r14};
  \ClusterLD c1[][2] = (r1)(r2);
  \ClusterLD c2[][2] = (c1)(r3)(r4)(r5);
  \ClusterLD c3[][2] = (c2)(r6);
  \ClusterLD c4[\mathfrak{s}_1][2] = (c3)(r7);
  \ClusterLD c5[][2] = (r8)(r9);
  \ClusterLD c6[][2] = (c5)(r10)(r11)(r12);
  \ClusterLD c7[\mathfrak{s}_2][2] = (c6)(r13);
  \ClusterD c5[0] = (c4)(c7)(r14);
\endclusterpicture}
\vspace{10pt}
\caption*{Cluster picture of $C/K$}
		\end{figure}
				\end{center}
    \end{minipage}\begin{minipage}[b]{0.4\textwidth}
		\begin{center}
		\begin{figure}[H]
\scalebox{1.5}{
\clusterpicture             
  \Root[D] {1} {first} {r1};
  \Root[D] {} {r1} {r2};
  \Root[Dot] {2} {r2} {r3};
  \Root[Dot] {} {r3} {r4};
    \Root[Dot] {} {r4} {r5};
    \Root[D] {1.5} {r5} {r6};
    \Root[D] {1} {r6} {r7};
    \Root[D] {6} {r7} {r8};
    \Root[D] {} {r8} {r9};
    \Root[Dot] {2} {r9} {r10};
    \Root[Dot] {} {r10} {r11};
    \Root[Dot] {} {r11} {r12};
    \Root[D] {1.5} {r12} {r13};
    \Root[D] {1} {r13} {r14};
  \ClusterLD c1[][2] = (r1)(r2);
  \ClusterLD c2[][2] = (c1)(r3)(r4)(r5);
  \ClusterLD c3[][2] = (c2)(r6);
  \ClusterLD c4[\mathfrak{s}_3][1] = (c3)(r7);
  \ClusterLD c5[][2] = (r8)(r9);
  \ClusterLD c6[][2] = (c5)(r10)(r11)(r12);
  \ClusterLD c7[][2] = (c6)(r13);
  \ClusterLD c8[\mathfrak{s}_4][1] = (c7)(r14);
  \ClusterD c5[0] = (c4)(c8);
\endclusterpicture}
\vspace{10pt}
\caption*{Cluster picture of $C'/K$}
		\end{figure}
				\end{center}
    \end{minipage}
\end{center}
where the dots represent clusters that decrease in size by $1$ at each step and each have a relative depth of $2$, and $\mathfrak{s}_1$ and $\mathfrak{s}_2$ have size $d/2$ and $d/2-1$ and $\mathfrak{s}_3$ and $\mathfrak{s}_4$ have size $d/2$. For a hyperelliptic curve $\mathcal{C}:y^2=f(x)$ with leading coefficient $c_f$, the discriminant is given by
\begin{equation}
\Delta_{\mathcal{C}} = 16^g c_f^{4g+2}\textup{disc}(\frac{1}{c_f}f(x)). 
\end{equation}
From the cluster picture, we can read off
\begin{equation}
\textup{ord}\left(\textup{disc}\left(\frac{1}{c_{f_1}}f_1(x)\right)\right)= \textup{ord}\left(\textup{disc}\left(\frac{1}{c_{f_2}}f_2(x)\right)\right)= 2\left(4\sum_{k=2}^{\frac{d}{2}-1}\left(\frac{k}{2}\right)+2\left(\frac{d/2}{2}\right)\right),
\end{equation}
and so $\textup{ord}(\Delta_C)=\textup{ord}(\Delta_{C'})$ since $\textup{ord}(c_{f_1})=\textup{ord}(c_{f_2})=0$ and $K$ was assumed to have odd residue characteristic. Since $C$ and $C'$ are isomorphic over $\Bar{K}$, $A(C)=A(C')$ for any absolute invariant $A$. Thus $\textup{ord}(I(f_1(x)))=\textup{ord}(I(f_2(x)))$ for any invariant $I$ of weight $k$, since we can find an $a,b\in\mathbb{Z}^{+}$ such that $I^a/\Delta_C^b$ is an absolute invariant. By \cite{m2d2} Theorem 1.8, $C/K$ is semistable and $C'/K$ is not semistable.

\textit{Case $2$: $g$ is odd and $n=\frac{d}{4}$ is even}. Denote by $i$ a primitive $4$-th root of unity in $\Bar{K}$. Then the curves
\begin{align}
C:y^2&=f_1(x)=\prod_{j=1}^{\frac{n}{2}}\left(x-\frac{1}{1+\pi^{2j}}\right)\left(x+\frac{1}{1+\pi^{2j}}\right) \cdot \prod_{k=0}^{n-1} \left(x+\frac{i\pi}{(1+\pi^{2k+1})}\right)\left(x-\frac{i\pi}{(1+\pi^{2k+1})}\right) \\
& \hspace{100pt}\cdot\prod_{l=0}^{\frac{n}{2}-1} \left(x-\frac{\pi}{(1+\pi^{2l+1})}\right)\left(x+\frac{\pi}{(1+\pi^{2l+1})}\right) \quad \textup{ and } \\
C':y^2&=f_2(x)=\prod_{j=1}^{\frac{n}{2}}(x-\pi^{\frac{3}{2}}-\pi^{\frac{4j+3}{2}})(x+\pi^{\frac{3}{2}}+\pi^{\frac{4j+3}{2}}) \cdot \prod_{k=0}^{n-1} (x-i(\pi^{\frac{1}{2}}+\pi^{\frac{4k+3}{2}}))(x+i(\pi^{\frac{1}{2}}+\pi^{\frac{4k+3}{2}})) \\
& \hspace{100pt}\cdot\prod_{l=0}^{\frac{n}{2}-1} (x-\pi^{\frac{1}{2}}-\pi^{\frac{4l+3}{2}})(x+\pi^{\frac{1}{2}}+\pi^{\frac{4l+3}{2}}), 
\end{align}
are defined over $K$ and are isomorphic over $\Bar{K}$ by taking $x\mapsto \pi^{\frac{3}{2}}/x$ and $y\mapsto\frac{(\alpha_1\cdots\alpha_d)^{\frac{1}{2}}}{x^{\frac{d}{2}}}y$ in the second Weierstrass equation, where $\alpha_1,\dots,\alpha_d$ are the roots of $f_2(x)$. The Weierstrass equations for these curves have the following cluster pictures

\begin{figure}[H]
\scalebox{1.5}{
\clusterpicture             
\Root[D] {1} {first} {r1};
  \Root[D] {} {r1} {r2};
  \Root[Dot] {2} {r2} {r3};
  \Root[Dot] {} {r3} {r4};
    \Root[Dot] {} {r4} {r5};
    \Root[D] {1.5} {r5} {r6};
    \Root[D] {1} {r6} {r7};
    \Root[D] {6} {r7} {r8};
    \Root[D] {} {r8} {r9};
    \Root[Dot] {2} {r9} {r10};
    \Root[Dot] {} {r10} {r11};
    \Root[Dot] {} {r11} {r12};
    \Root[D] {1.5} {r12} {r13};
    \Root[D] {1} {r13} {r14};
    \Root[D] {8} {r14} {r15};
  \Root[D] {} {r15} {r16};
  \Root[Dot] {2} {r16} {r17};
  \Root[Dot] {} {r17} {r18};
    \Root[Dot] {} {r18} {r19};
    \Root[D] {1.5} {r19} {r20};
    \Root[D] {1} {r20} {r21};
    \Root[D] {6} {r21} {r22};
    \Root[D] {} {r22} {r23};
    \Root[Dot] {2} {r23} {r24};
    \Root[Dot] {} {r24} {r25};
    \Root[Dot] {} {r25} {r26};
    \Root[D] {1.5} {r26} {r27};
    \Root[D] {1} {r27} {r28};
     \Root[D] {6} {r28} {r29};
    \Root[D] {} {r29} {r30};
    \Root[Dot] {2} {r30} {r31};
    \Root[Dot] {} {r31} {r32};
    \Root[Dot] {} {r32} {r33};
    \Root[D] {1.5} {r33} {r34};
    \Root[D] {1} {r34} {r35};
     \Root[D] {6} {r35} {r36};
    \Root[D] {} {r36} {r37};
    \Root[Dot] {2} {r37} {r38};
    \Root[Dot] {} {r38} {r39};
    \Root[Dot] {} {r39} {r40};
    \Root[D] {1.5} {r40} {r41};
    \Root[D] {1} {r41} {r42};
  \ClusterLD c1[][2] = (r1)(r2);
  \ClusterLD c2[][2] = (c1)(r3)(r4)(r5);
  \ClusterLD c3[][2] = (c2)(r6);
  \ClusterLD c4[\mathfrak{s}_1][2] = (c3)(r7);
  \ClusterLD c5[][2] = (r8)(r9);
  \ClusterLD c6[][2] = (c5)(r10)(r11)(r12);
  \ClusterLD c7[][2] = (c6)(r13);
   \ClusterLD c77[\mathfrak{s}_2][2] = (c7)(r14);
  \ClusterLD c8[][2] = (r15)(r16);
  \ClusterLD c9[][2] = (c8)(r17)(r18)(r19);
  \ClusterLD c10[][2] = (c9)(r20);
  \ClusterLD c11[\mathfrak{s}_3][1] = (c10)(r21);
  \ClusterLD c12[][2] = (r23)(r22);
  \ClusterLD c13[][2] = (c12)(r24)(r25)(r26);
  \ClusterLD c14[][2] = (c13)(r27);
  \ClusterLD c15[\mathfrak{s}_4][1] = (c14)(r28);
  \ClusterLD c16[][2] = (r29)(r30);
  \ClusterLD c17[][2] = (c16)(r31)(r32)(r33);
  \ClusterLD c18[][2] = (c17)(r34);
  \ClusterLD c19[\mathfrak{s}_5][1] = (c18)(r35);
  \ClusterLD c20[][2] = (r36)(r37);
  \ClusterLD c21[][2] = (c20)(r38)(r39)(r40);
  \ClusterLD c22[][2] = (c21)(r41);
  \ClusterLD c23[\mathfrak{s}_6][1] = (c22)(r42);
  \ClusterLD c25[][1] = (c11)(c15)(c23)(c19);
  \ClusterD c5[0] = (c77)(c4)(c25);
\endclusterpicture}
\vspace{10pt}
\caption*{Cluster picture of $C/K$}
		\end{figure}
    
\begin{figure}[H]
\scalebox{1.5}{
\clusterpicture             
\Root[D] {1} {first} {r1};
  \Root[D] {} {r1} {r2};
  \Root[Dot] {2} {r2} {r3};
  \Root[Dot] {} {r3} {r4};
    \Root[Dot] {} {r4} {r5};
    \Root[D] {1.5} {r5} {r6};
    \Root[D] {1} {r6} {r7};
    \Root[D] {6} {r7} {r8};
    \Root[D] {} {r8} {r9};
    \Root[Dot] {2} {r9} {r10};
    \Root[Dot] {} {r10} {r11};
    \Root[Dot] {} {r11} {r12};
    \Root[D] {1.5} {r12} {r13};
    \Root[D] {1} {r13} {r14};
    \Root[D] {8} {r14} {r15};
  \Root[D] {} {r15} {r16};
  \Root[Dot] {2} {r16} {r17};
  \Root[Dot] {} {r17} {r18};
    \Root[Dot] {} {r18} {r19};
    \Root[D] {1.5} {r19} {r20};
    \Root[D] {1} {r20} {r21};
    \Root[D] {6} {r21} {r22};
    \Root[D] {} {r22} {r23};
    \Root[Dot] {2} {r23} {r24};
    \Root[Dot] {} {r24} {r25};
    \Root[Dot] {} {r25} {r26};
    \Root[D] {1.5} {r26} {r27};
    \Root[D] {1} {r27} {r28};
     \Root[D] {6} {r28} {r29};
    \Root[D] {} {r29} {r30};
    \Root[Dot] {2} {r30} {r31};
    \Root[Dot] {} {r31} {r32};
    \Root[Dot] {} {r32} {r33};
    \Root[D] {1.5} {r33} {r34};
    \Root[D] {1} {r34} {r35};
     \Root[D] {6} {r35} {r36};
    \Root[D] {} {r36} {r37};
    \Root[Dot] {2} {r37} {r38};
    \Root[Dot] {} {r38} {r39};
    \Root[Dot] {} {r39} {r40};
    \Root[D] {1.5} {r40} {r41};
    \Root[D] {1} {r41} {r42};
  \ClusterLD c1[][2] = (r1)(r2);
  \ClusterLD c2[][2] = (c1)(r3)(r4)(r5);
  \ClusterLD c3[][2] = (c2)(r6);
  \ClusterLD c4[\mathfrak{s}_1'][2] = (c3)(r7);
  \ClusterLD c5[][2] = (r8)(r9);
  \ClusterLD c6[][2] = (c5)(r10)(r11)(r12);
  \ClusterLD c7[][2] = (c6)(r13);
   \ClusterLD c77[\mathfrak{s}_2'][2] = (c7)(r14);
   \ClusterLD c24[][1] = (c77)(c4);
  \ClusterLD c8[][2] = (r15)(r16);
  \ClusterLD c9[][2] = (c8)(r17)(r18)(r19);
  \ClusterLD c10[][2] = (c9)(r20);
  \ClusterLD c11[\mathfrak{s}_3'][1] = (c10)(r21);
  \ClusterLD c12[][2] = (r23)(r22);
  \ClusterLD c13[][2] = (c12)(r24)(r25)(r26);
  \ClusterLD c14[][2] = (c13)(r27);
  \ClusterLD c15[\mathfrak{s}_4'][1] = (c14)(r28);
  \ClusterLD c16[][2] = (r29)(r30);
  \ClusterLD c17[][2] = (c16)(r31)(r32)(r33);
  \ClusterLD c18[][2] = (c17)(r34);
  \ClusterLD c19[\mathfrak{s}_5'][1] = (c18)(r35);
  \ClusterLD c20[][2] = (r36)(r37);
  \ClusterLD c21[][2] = (c20)(r38)(r39)(r40);
  \ClusterLD c22[][2] = (c21)(r41);
  \ClusterLD c23[\mathfrak{s}_6'][1] = (c22)(r42);
  \ClusterD c5[\frac{1}{2}] = (c24)(c11)(c15)(c23)(c19);
\endclusterpicture}
\vspace{10pt}
\caption*{Cluster picture of $C'/K$}
		\end{figure}
\noindent where the dots represent clusters that decrease in size by $1$ at each step and each have a relative depth of $2$, and $\mathfrak{s}_1$, $\mathfrak{s}_1'$, $\mathfrak{s}_2$,  $\mathfrak{s}_2'$, $\mathfrak{s}_5$, $\mathfrak{s}_5'$, $\mathfrak{s}_6$ and $\mathfrak{s}_6'$ have size $n/2$ and $\mathfrak{s}_3$, $\mathfrak{s}_3'$, $\mathfrak{s}_4$ and $\mathfrak{s}_4'$ have size $n$. One can verify from the cluster picture that $\textup{ord}(\Delta_C)=\textup{ord}(\Delta_{C'})$ so, as in Case 1, $\textup{ord}(I(f_1(x)))=\textup{ord}(I(f_2(x)))$ for any invariant $I$ of weight $k$. Again, one can verify from \cite{m2d2} Theorem 1.8 that $C/K$ is semistable and $C'/K$ is not semistable.

\textit{Case $3$: $g$ is odd and $n=\frac{d}{4}$ is odd.} The curves
\begin{align}
C:y^2&=f_1(x)=(x+1)\cdot(x-1+\pi^2)\cdot\prod_{j=2}^{n-1} (x-1-\pi^2-\pi^{2j}) \prod_{k=1}^{n} (x-i(\pi^2+\pi^{2k}))(x+i(\pi^2+\pi^{2k})) \\
&\hspace{100pt} \cdot (x+\pi^2)\cdot\prod_{l=2}^{n}(x-\pi^2-\pi^{2l})\quad \textup{ and } \\
C':y^2&=f_2(x)= (x+\pi^3)\cdot(x-\frac{\pi^3}{1-\pi^2})\cdot \prod_{j=2}^{n-1}(x-\frac{\pi^3}{1+\pi^2+\pi^{2j}})\cdot \prod_{k=0}^{n-1} (x+\frac{i\pi}{1+\pi^{2k}})(x-\frac{i\pi}{1+\pi^{2k}}) \\
&\hspace{100pt} \cdot (x+\pi)\cdot\prod_{l=1}^{n-1}(x-\frac{\pi}{1+\pi^{2l}})
\end{align}
are defined over $K$ and are isomorphic over $\Bar{K}$ by taking $x\mapsto \pi^3/x$ and $y\mapsto\frac{(\alpha_1\cdots\alpha_d)^{\frac{1}{2}}}{x^{\frac{d}{2}}}y$ in the first Weierstrass equation, where $\alpha_1,\dots,\alpha_d$ are the roots of $f_1(x)$. The Weierstrass equations for these curves have the following cluster pictures
\begin{center}
		\begin{center}
		\begin{figure}[H]
\scalebox{1.5}{
\clusterpicture             
  \Root[D] {1} {first} {r1};
  \Root[D] {} {r1} {r2};
  \Root[Dot] {2} {r2} {r3};
  \Root[Dot] {} {r3} {r4};
    \Root[Dot] {} {r4} {r5};
    \Root[D] {1.5} {r5} {r6};
    \Root[D] {1} {r6} {r7};
    \Root[D] {6} {r7} {r8};
    \Root[D] {} {r8} {r9};
    \Root[Dot] {2} {r9} {r10};
    \Root[Dot] {} {r10} {r11};
    \Root[Dot] {} {r11} {r12};
    \Root[D] {1.5} {r12} {r13};
    \Root[D] {1} {r13} {r14};
    \Root[D] {6} {r14} {r15};
  \Root[D] {} {r15} {r16};
  \Root[Dot] {2} {r16} {r17};
  \Root[Dot] {} {r17} {r18};
    \Root[Dot] {} {r18} {r19};
    \Root[D] {1.5} {r19} {r20};
    \Root[D] {1} {r20} {r21};
    \Root[D] {5} {r21} {r22};
    \Root[D] {} {r22} {r23};
    \Root[Dot] {2} {r23} {r24};
    \Root[Dot] {} {r24} {r25};
    \Root[Dot] {} {r25} {r26};
    \Root[D] {1.5} {r26} {r27};
    \Root[D] {1} {r27} {r28};
  \ClusterLD c1[][2] = (r1)(r2);
  \ClusterLD c2[][2] = (c1)(r3)(r4)(r5);
  \ClusterLD c3[\mathfrak{s}_1][2] = (c2)(r6);
  \ClusterLD c5[][2] = (r8)(r9);
  \ClusterLD c6[][2] = (c5)(r10)(r11)(r12);
  \ClusterLD c7[][2] = (c6)(r13);
   \ClusterLD c77[\mathfrak{s}_2][2] = (c7)(r14);
  \ClusterLD c8[][2] = (r15)(r16);
  \ClusterLD c9[][2] = (c8)(r17)(r18)(r19);
  \ClusterLD c10[][2] = (c9)(r20);
  \ClusterLD c11[\mathfrak{s}_3][2] = (c10)(r21);
  \ClusterLD c12[][2] = (r23)(r22);
  \ClusterLD c13[][2] = (c12)(r24)(r25)(r26);
  \ClusterLD c14[\mathfrak{s}_4][2] = (c13)(r27);
  \ClusterLD c15[][2] = (c77)(c14)(c11)(r14)(r28);
  \ClusterD c5[0] = (c3)(r7)(c15);
\endclusterpicture}
\vspace{10pt}
\caption*{Cluster picture of $C/K$}
		\end{figure}
				\end{center}

		\begin{center}
		\begin{figure}[H]
\scalebox{1.5}{
\clusterpicture             
  \Root[D] {1} {first} {r1};
  \Root[D] {} {r1} {r2};
  \Root[Dot] {2} {r2} {r3};
  \Root[Dot] {} {r3} {r4};
    \Root[Dot] {} {r4} {r5};
    \Root[D] {1.5} {r5} {r6};
    \Root[D] {1} {r6} {r7};
    \Root[D] {6} {r7} {r8};
    \Root[D] {} {r8} {r9};
    \Root[Dot] {2} {r9} {r10};
    \Root[Dot] {} {r10} {r11};
    \Root[Dot] {} {r11} {r12};
    \Root[D] {1.5} {r12} {r13};
    \Root[D] {1} {r13} {r14};
    \Root[D] {6} {r14} {r15};
  \Root[D] {} {r15} {r16};
  \Root[Dot] {2} {r16} {r17};
  \Root[Dot] {} {r17} {r18};
    \Root[Dot] {} {r18} {r19};
    \Root[D] {1.5} {r19} {r20};
    \Root[D] {1} {r20} {r21};
    \Root[D] {5} {r21} {r22};
    \Root[D] {} {r22} {r23};
    \Root[Dot] {2} {r23} {r24};
    \Root[Dot] {} {r24} {r25};
    \Root[Dot] {} {r25} {r26};
    \Root[D] {1.5} {r26} {r27};
    \Root[D] {1} {r27} {r28};
  \ClusterLD c1[][2] = (r1)(r2);
  \ClusterLD c2[][2] = (c1)(r3)(r4)(r5);
  \ClusterLD c3[][2] = (c2)(r6);
  \ClusterLD c4[\mathfrak{s}_1'][2] = (c3)(r7);
  \ClusterLD c5[][2] = (r8)(r9);
  \ClusterLD c6[][2] = (c5)(r10)(r11)(r12);
  \ClusterLD c7[][2] = (c6)(r13);
   \ClusterLD c77[\mathfrak{s}_2'][2] = (c7)(r14);
  \ClusterLD c8[][2] = (r15)(r16);
  \ClusterLD c9[][2] = (c8)(r17)(r18)(r19);
  \ClusterLD c10[][2] = (c9)(r20);
  \ClusterLD c11[\mathfrak{s}_3'][2] = (c10)(r21);
  \ClusterLD c12[][2] = (r23)(r22);
  \ClusterLD c13[][2] = (c12)(r24)(r25)(r26);
  \ClusterLD c14[\mathfrak{s}_4'][2] = (c13)(r27);
  \ClusterD c5[1] = (c4)(c77)(c14)(c11)(r14)(r28);
\endclusterpicture}
\vspace{10pt}
\caption*{Cluster picture of $C'/K$}
		\end{figure}
				\end{center}

\end{center}
\noindent where the dots represent clusters that decrease in size by $1$ at each step and each have a relative depth of $2$, and where the clusters $\mathfrak{s}_1$, $\mathfrak{s}_4$ and $\mathfrak{s}_4'$ have size $n-1$ and
$\mathfrak{s}_1'$, $\mathfrak{s}_2$, $\mathfrak{s}_2'$, $\mathfrak{s}_3$ and $\mathfrak{s}_3'$ have size $n$. Again, one can verify from the cluster picture that $\textup{ord}(\Delta_C)=\textup{ord}(\Delta_{C'})$ so, as in Case 1, $\textup{ord}(I(f_1(x)))=\textup{ord}(I(f_2(x)))$ for any invariant $I$ of weight $k$. By \cite{m2d2} Theorem 1.8, $C/K$ is semistable and $C'/K$ is not semistable.
\end{proof}

\section{Dual graphs and stable model trees}\label{dualgraphstreesection}

Let $C:y^2=f(x)$ be a semistable hyperelliptic curve over a local field $K$ of odd residue characteristic, where $f(x)$ has even degree. We want to recover the dual graph of the special fibre of the minimal regular model of $C$ from absolute invariants of the curve. In \cite{m2d2} Theorem 5.18, it is shown that the dual graph is determined by the BY tree of $C/K$ (see Definition \ref{bytreedefinition} above). It will make our lives easier to work with a different tree, from which the BY tree can be recovered, called the \textit{stable model tree}. We define the stable model tree in this section, and describe how the BY tree, and hence dual graph, can be recovered from it. In \S\ref{mainthminvproof}, we describe how the stable model tree (and thus the dual graph) can be recovered from absolute invariants of the curve. 

\begin{definition}[Stable model tree]\label{stablemodeltreedef}
Let $C:y^2=f(x)$ be a Weierstrass model for a hyperelliptic curve over a local field $K$ of odd residue characteristic, where $f(x)$ has even degree $d$, and let $\Sigma$ denote the cluster picture of $C$. We call the following graph $T_C=(V,E,L)$ the \textit{stable model tree of }$C$. Take the graph with:
\begin{itemize}
\item A vertex $v_\mathfrak{s}$ for every cluster $\mathfrak{s}\in \Sigma$, including clusters of size $1$. 
\item An edge linking $v_{\mathfrak{s}}$ to $v_{P(\mathfrak{s})}$ for every cluster $\mathfrak{s}\in \Sigma$.
\item A length function so that the edge $v_{\mathfrak{s}}v_{P(\mathfrak{s})}$ has length $\delta_{\mathfrak{s}}$, excluding the edges extending to clusters of size $1$ which do not have an allocated length.
\end{itemize}
To obtain $T_C$ from this graph, we remove the vertex $v_{\mathcal{R}}$ if $\mathcal{R}$ is a disjoint union of two children $\mathfrak{s}_1$ and $\mathfrak{s}_2$. We create a single edge between $v_{\mathfrak{s}_1}$ and $v_{\mathfrak{s}_2}$ of length $\delta(\mathfrak{s}_1,\mathcal{R})+\delta(\mathfrak{s}_1,\mathcal{R})$ if $\mathfrak{s}_1$ and $\mathfrak{s}_2$ have size $\geq 2$ and with no assigned length if $\mathfrak{s}_1$ or $\mathfrak{s}_2$ have size $1$.

We call the leaf vertices (corresponding to clusters of size $1$) \textit{singletons} of $T_C$, and vertices that are not singletons \textit{proper vertices}. If a proper vertex $v$ becomes a leaf if the singletons of $T_C$ are removed we call it a \textit{proper leaf}.
\end{definition}

We wish to show that there is a one-to-one correspondence between BY trees and stable model trees. To do so, we will need the following propositions and definitions. We will also need the following observation about the stable model tree of a curve with potentially good reduction when proving the main results, as we deal with the case of potentially good reduction separately in Theorem \ref{wholealgorithm}.

\begin{proposition}\label{potgoodredlemma}
Let $C:y^2=f(x)$ be a hyperelliptic curve of genus $g$ over a local field $K$ of odd residue characteristic, where $d=\deg(f)=2g+2$. The following are equivalent:
\begin{enumerate}[(i)]
    \item $C/K$ has potentially good reduction. 
    \item The cluster picture of $C:y^2=f(x)$ has no proper clusters of size $<2g+1$. 
    \item $T_C$ is isomorphic to the complete bipartite graph $K_{1,d}$.
    \item $T_\Sigma$ consists of a single blue vertex of genus $g$. 
\end{enumerate}
\end{proposition}

\begin{proof}
$(i) \iff (ii)$ is given in \cite{m2d2} Theorem 1.8(3). For $(ii) \iff (iii)$, suppose the cluster picture of $C:y^2=f(x)$ has no proper clusters of size $<2g+1$. This means that the cluster picture either has one cluster, the top cluster $\mathcal{R}$, or it has the top cluster $\mathcal{R}$ and a subcluster of size $d-1$. By the definition of $T_C$, this means that $T_C$ is isomorphic to the complete bipartite graph $K_{1,d}$. Conversely, if $T_C$ is isomorphic to the complete bipartite graph $K_{1,d}$, this can only occur if the cluster picture either has one cluster, the top cluster, or it has the top cluster and a subcluster of size $d-1$. This is because otherwise there would be at least two vertices corresponding to proper leaves of $T_C$ which each have edges extending to at least $2$ singleton vertices of degree $1$, meaning that $T_C$ is not isomorphic to $K_{1,d}$. 

For $(ii)\iff (iv)$, it is clear by Definition \ref{bytreedefinition} that if the cluster picture of $C:y^2=f(x)$ has no proper clusters of size $<2g+1$ then $T_\Sigma$ consists of a single blue vertex of genus $g$. For $(iv) \implies (ii)$, suppose the cluster picture of $C:y^2=f(x)$ has a proper cluster $\mathfrak{s}$ of size $<2g+1$. Then $T_\Sigma$ contains at least two vertices $v_{\mathfrak{s}}$ and $v_{\mathcal{R}}$ if $\mathcal{R}$ is not a union of two clusters, or $v_{\mathfrak{s}}$ and $v_{\mathfrak{s}'}$ if $\mathcal{R}$ is the disjoint union of two proper clusters $\mathfrak{s}$ and $\mathfrak{s}'$.
\end{proof}

\begin{definition}
A directed tree $T$ is called \textit{rooted} if there is a vertex $v$ (the \textit{root}) for which for every vertex $w\neq v$, the path from $v$ to $w$ has all edges on the path directed away from $v$. \textit{Rooting a tree} $T$ \textit{at a vertex} $v$ is the process of designating $v$ to be the root by directing all edges away from it.
\end{definition}

\begin{proposition}\label{bytotc}
As unweighted graphs, $T_\Sigma$ is isomorphic to $T_C$ with the singletons removed. The stable model tree $T_C=(V,E,L)$ is uniquely determined by the BY tree $T_{\Sigma}$ using the following procedure: 
\begin{enumerate}[(i)]
\item Root $T_\Sigma$ at a vertex $v_r$. 
\item Let $k_v$ denote the number of blue edges directed away from a vertex $v$. 
\begin{itemize}
\item The number of singleton vertices attached to $v_r$ in $T_C$ is $2g(v_r)+2-k_{v_r}$.
\item The number of singleton vertices attached to $v\neq v_r$ in $T_C$ is 
\begin{equation}
\begin{cases}
2g(v)+1 -k_v &\textup{ if the edge extending to $v$ in $T_\Sigma$ is blue}; \\
2g(v)+2 -k_v &\textup{ if the edge extending to $v$ in $T_\Sigma$ is yellow}. \\
\end{cases}
\end{equation} 
\end{itemize}
\item The length function $L:E\setminus \{\textup{edges extending to leaves}\}\rightarrow \mathbb{Q}^+$ is given by
\begin{equation}
L(e)=\begin{cases}
L_{BY}(e) &\textup{if $e$ is blue in the BY tree}; \\
\frac{L_{BY}(e)}{2} &\textup{if $e$ is yellow in the BY tree},
\end{cases}
\end{equation}
where $L_{BY}(e)$ is the length of the edge $e$ in the BY tree. 
\end{enumerate}
\end{proposition}

\begin{proof}
Note that by Definition \ref{bytreedefinition} and \ref{stablemodeltreedef}, $T_C$ and $T_\Sigma$ are isomorphic as unweighted trees when the singletons of $T_C$ are removed. So we must check that the above process recovers the correct length of the edges of $T_C$ and the correct number of singletons attached to each proper vertex.

By Proposition \ref{potgoodredlemma}, if $T_\Sigma$ consists of one blue vertex of genus $g$ then $T_C$ is $K_{1,d}$, which is correctly recovered by the above process since $d=2g+2$.

Now suppose $T_\Sigma$ contains more than one vertex. If $v_\mathcal{R}$ is in the vertex set and $T_\Sigma$ is rooted at $v_\mathcal{R}$. Then by the definition of $T_C$, the above procedure recovers $T_C$ since for a vertex $v_{\mathfrak{s}}$
\begin{align}
\#\{\textup{singleton vertices}\}&=\#\{\textup{odd children of $\mathfrak{s}$}\}-\#\{\textup{proper odd children of $\mathfrak{s}$}\} \\
&= \begin{cases}
2g(v_{\mathfrak{s}})+1 -k_{v_{\mathfrak{s}}} &\textup{ if the edge extending to $v_{\mathfrak{s}}$ in $T_\Sigma$ is blue}; \\
2g(v_{\mathfrak{s}})+2 -k_{v_{\mathfrak{s}}} &\textup{ if the edge extending to $v_{\mathfrak{s}}$ in $T_\Sigma$ is yellow}. \\
\end{cases}
\end{align}

If $v_{\mathcal{R}}$ is excluded from the vertex set in $T_C$ and $T_{\Sigma}$, let $\Tilde{T}_\Sigma$ denote $T_\Sigma$ but with $v_{\mathcal{R}}$ added back in. Again by Definition \ref{stablemodeltreedef}, if $\Tilde{T}_\Sigma$ is rooted at $v_{\mathcal{R}}$ and then $v_{\mathcal{R}}$ is removed and the lengths of the adjacent edges are added, the above process correctly determines $T_C$. Let $v$ and $w$ be the vertices adjacent to $v_{\mathcal{R}}$ and root $T_\Sigma$ at $v$. We want to show that following the above process with $T_\Sigma$ rooted at $v$ recovers the same tree as with $\Tilde{T}_\Sigma$ rooted at $v_{\mathcal{R}}$ and then removing $v_{\mathcal{R}}$. From $w$ and the edges directed away from $v$ onward, the BY tree looks identical to if it were rooted at $v_{\mathcal{R}}$, so the above process correctly recovers the $T_C$ for these parts of the tree. So we must check that the number of singletons attached to $v$ is correctly recovered. Let $k_v$ be the number of blue edges directed away from $v$ in $T_\Sigma$ and $\Tilde{k}_v$ be the number of blue edges directed away from $v$ in $\Tilde{T}_\Sigma$. If $vw$ is yellow in $T_\Sigma$ then $v_\mathcal{R}v$ and $v_\mathcal{R}w$ are yellow in $\Tilde{T}_\Sigma$ and so $k_v=\Tilde{k}_v$ and the above process determines the number of singletons attached to $v$ in $\Tilde{T}_\Sigma$ and $T_\Sigma$ to be $2g(v)+2-k_v$. If $vw$ is blue in $T_\Sigma$ then $v_\mathcal{R}v$ and $v_\mathcal{R}w$ are blue in $\Tilde{T}_\Sigma$ and so $k_v=\Tilde{k}_v+1$. In this case the above process determines the number of singletons attached to $v$ in $\Tilde{T}_\Sigma$ is $2g(v)+1-\Tilde{k}_v$ and the number of singletons attached $v$ in $T_\Sigma$ to be $2g(v)+2-k_v=2g(v)+1-\Tilde{k}_v$.

Thus, if $T_\Sigma$ is rooted at $v_{\mathcal{R}}$ or an adjacent vertex the above process uniquely determines $T_C$. We now show that the tree obtained by procedure described in the proposition does not depend on the initial choice of root vertex. This is due to the fact that when choosing a different root vertex, if the edge directed towards a vertex $v$ goes from yellow to blue then $k_v$ goes down by $1$. Initially there are $2g(v)+2-k_v$ singletons and after choosing a different root there are $2g(v)+1-(k_v-1)=2g(v)+2-k_v$ singletons, so the number of singletons is unchanged. If it goes from blue to yellow then $k_v$ goes up by $1$ and similarly the number of singletons is unchanged. If a vertex $v$ becomes the root vertex, $k_v$ either stays the same if the edge directed towards $v$ was yellow or goes up by $1$ if the edge directed towards $v$ was blue. In both cases, the number of singleton vertices calculated by the above method is unchanged. This completes the proof.
\end{proof}

\begin{corollary}\label{stablemodelinvariant}
Let $C$ and $C'$ be hyperelliptic curves defined over $K$, a local field of odd residue characteristic, that are isomorphic over $\Bar{K}$. Then $T_{C}$ is isomorphic to $T_{C'}$.
\end{corollary}

\begin{proof}
In \cite{m2d2} Definition 14.1, the notion of equivalent cluster pictures is introduced. Let $F$ be the finite extension of $K$ over which the isomorphism between $C$ and $C'$ is defined. By Theorem 14.4 of \cite{m2d2} and its proof, the cluster pictures of $C$ and $C'$ are equivalent over $F$ which implies they are equivalent over $K$. By \cite{semistable} Theorem 1.1, curves with equivalent cluster pictures have the same BY tree, and by Proposition \ref{bytotc} they have the same stable model tree.
\end{proof}

The above corollary tells us that it makes sense to have defined the stable model tree associated to a hyperelliptic curve $C$, as opposed to associating it to a Weierstrass equation for $C$, since $T_C$ is invariant under a change of model. This also makes it clear as to why one might hope to be able to recover the stable model tree from absolute invariants of the curve.

\begin{example}
The curves
\begin{small}
\begin{align}
C:y^2&=f(x)=(x+p^2-p^5)(x+p^2+p^5)(x+p^3+p^7)(x+p^3-p^7)(x-1+p)(x-1-p) \quad \textup{ and }\\
C':y^2&=\left(x+\frac{1}{1+p^2-p^5}\right)\left(x+\frac{1}{1+p^2+p^5}\right)\left(x+\frac{1}{1+p^3+p^7}\right)\left(x+\frac{1}{1+p^3-p^7}\right)\left(x+\frac{1}{p}\right)\left(x-\frac{1}{p}\right)
\end{align}\end{small}

\noindent are isomorphic over $\Bar{\mathbb{Q}}_p$ by taking $x\mapsto \frac{1}{x}+1$ and $y\mapsto \frac{(\alpha_1\cdots\alpha_6)^{\frac{1}{2}}}{x^{3}}y$ in the first Weierstrass equation, where $\alpha_1,\dots,\alpha_6$ are the roots of $f(x)$. For $p\geq 3$ they have the following cluster pictures and the same stable model tree.
\begin{center}
\begin{minipage}[b]{0.35\textwidth}
		\begin{center}
		\begin{figure}[H]
\scalebox{1.5}{
\clusterpicture             
  \Root[D] {1} {first} {r1};
  \Root[D] {} {r1} {r2};
  \Root[D] {3} {r2} {r3};
  \Root[D] {} {r3} {r4};
    \Root[D] {3.5} {r4} {r5};
    \Root[D] {} {r5} {r6};
  \ClusterLD c1[][3] = (r1)(r2);
    \ClusterLD c2[][5] = (r3)(r4);
    \ClusterLD c3[][1] = (r5)(r6);
    \ClusterLD c4[][2] = (c1)(c2);
  \ClusterD c5[0] = (c1)(c2)(c3)(c4);
\endclusterpicture}
\vspace{10pt}
\caption*{Cluster picture of $C/\mathbb{Q}_p$}
		\end{figure}
				\end{center}
    \end{minipage}\begin{minipage}[b]{0.35\textwidth}
		\begin{center}
		\begin{figure}[H]
\scalebox{1.5}{
\clusterpicture             
  \Root[D] {1} {first} {r1};
  \Root[D] {} {r1} {r2};
  \Root[D] {3} {r2} {r3};
  \Root[D] {} {r3} {r4};
    \Root[D] {3.5} {r4} {r5};
    \Root[D] {} {r5} {r6};
  \ClusterLD c1[][3] = (r1)(r2);
    \ClusterLD c2[][5] = (r3)(r4);
    \ClusterLD c4[][3] = (c1)(c2);
  \ClusterD c5[-1] = (c1)(c2)(c4)(r5)(r6);
\endclusterpicture}
\vspace{10pt}
\caption*{Cluster picture of $C'/\mathbb{Q}_p$}
		\end{figure}
				\end{center}
    \end{minipage}\begin{minipage}[b]{0.35\textwidth}
		\begin{center}
		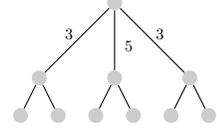
\begin{figure}[H]
			\begin{tikzpicture}
				[scale=0.5, auto=left,every node/.style={circle,fill=black!20,scale=0.6}]
                \node (n10) at (2,2) {};
				\node (n1) at (0,0) {};
				\node (n2) at (2,0)  {};
				\node (n3) at (4,0)  {};

                \node (n4) at (-0.5,-1) {};
                \node (n5) at (0.5,-1) {};
                
                \node (n6) at (1.5,-1) {};
                \node (n7) at (2.5,-1) {};

                \node (n8) at (3.5,-1) {};
                \node (n9) at (4.5,-1) {};

                \draw (n1) -- (n4); 
                \draw (n1) -- (n5); 
                
                \draw (n2) -- (n6); 
                \draw (n2) -- (n7); 

                \draw (n3) -- (n8); 
                \draw (n3) -- (n9);
				
				\draw (n10) -- (n1) node [text=black, pos=0.4, left, fill=none] {$3$};
                \draw (n10) -- (n2) node [text=black, pos=0.6, right,fill=none] {$5$};
				\draw (n10) -- (n3) node [text=black, pos=0.4, right,fill=none] {$3$};
			\end{tikzpicture}
 \caption*{The stable model tree $T_C=T_{C'}$}
		\end{figure}
				\end{center}
    \end{minipage}
\end{center}
\end{example}

\begin{remark}
Let $C$ be a hyperelliptic curve defined over a local field $K$ of odd residue characteristic and let $F/K$ be an extension of ramification degree $e$. Then $T_{C/F}$ is the same tree as $T_{C/K}$ but with all lengths multiplied by $e$.
\end{remark}

The stable model tree encodes information about the distances between clusters in the cluster picture of any Weierstrass equation for the curve. 

\begin{lemma}\label{distanceintc}
Let $C$ be a hyperelliptic curve with cluster picture $\Sigma$ and let $\mathfrak{s},\mathfrak{r}\in \Sigma$ be clusters. Then $\delta(\mathfrak{s},\mathfrak{r})=\delta(v_{\mathfrak{s}}, v_{\mathfrak{r}})$ as vertices in $T_C$.
\end{lemma}

\begin{proof}
This is clear from the definition of the distance between clusters (Definition \ref{delta}) and the definition of the stable model tree (Definition \ref{stablemodeltreedef}).
\end{proof}

We now show how the BY tree can be obtained from $T_C$. We will need the following definition. 

\begin{definition}\label{treeprelims}
Let $T_C$ be the stable model tree of a hyperelliptic curve and root $T_C$ at a proper vertex. 
\begin{itemize}
\item We call a vertex $w$ of $T_C$ a \textit{child} of vertex $v$ if there is an edge $vw$ in $T_C$ where the edge is directed from $v$ to $w$. In this case, we call $v$ the \textit{parent} of $w$.
\item We can allocate a proper vertex $v$ a size, which is 
\begin{equation}
\sum_{w} \#\{\textup{$s$ : $s$ is a singleton and there is an edge from $w$ to $s$}\},
\end{equation}
where we take the sum over the proper vertices $w$ for which there is a path from $v$ to $w$ that contains only forwards directed edges, including $w=v$. 
\begin{enumerate}[(i)]
\item We call a proper vertex $v$ \textit{even} if it has even size and \textit{odd} otherwise. 
\item We call a proper vertex $v$ \textit{übereven} if it is even and all of its children are even.
\end{enumerate}
\end{itemize}
\end{definition}

\begin{proposition}\label{byfromtc}
The BY tree of $C/K$ is uniquely determined by $T_C=(V,E,L)$ using the following procedure: 
\begin{enumerate}[(i)]
    \item Root $T_C$ at any proper vertex.
    \item Colour a proper vertex yellow if it is übereven and blue otherwise. 
    \item For any proper vertex $v$, define $g(v)$ such that $\#\{\textup{odd children of } v\}=2g(v)+1$ or $2g(v)+2$.
    \item Colour an edge $e$ yellow if the vertex it terminates at is even and blue otherwise; label it with length $2L(e)$ if $e$ is coloured yellow and $L(e)$ otherwise.
    \item Delete the singleton vertices of $T_C$ and the edges extending to them.
\end{enumerate}
\end{proposition}

\begin{proof}
Note that by Definition \ref{bytreedefinition} and \ref{stablemodeltreedef}, $T_C$ and $T_\Sigma$ are isomorphic as unweighted trees when the singletons of $T_C$ are removed. So we must check that the above process recovers the correct colouring and genus of the vertices and colouring of the edges.

By Proposition \ref{potgoodredlemma}, if $T_C$ is $K_{1,2g+2}$ then the BY tree is a single blue vertex of genus $g$, which is determined from $T_C$ using the above process.

Now suppose $T_C$ contains more than one proper vertex. If $v_{\mathcal{R}}$ is contained in the vertex set and $T_C$ is rooted at the $v_{\mathcal{R}}$ then by Definition \ref{bytreedefinition} the above process recovers the BY tree. If $v_{\mathcal{R}}$ is excluded from the vertex set in $T_C$ and $T_{\Sigma}$, let $\Tilde{T}_C$ denote $T_C$ but with $v_{\mathcal{R}}$ added back in. Again by Definition \ref{bytreedefinition}, if $\Tilde{T}_C$ is rooted at $v_{\mathcal{R}}$ and then $v_{\mathcal{R}}$ is removed and the lengths of the adjacent edges are added, the above process recovers the BY tree. Let $v$ and $w$ be the vertices adjacent to $v_{\mathcal{R}}$ and root $T_C$ at $v$. We want to show that following the above process with $T_C$ rooted at $v$ recovers the same tree as with $\Tilde{T}_C$ is rooted at $v_{\mathcal{R}}$ and then removing $v_{\mathcal{R}}$. From the edges directed away from $v$ onward the tree looks identical to if it were rooted at $v_{\mathcal{R}}$, so the above process correctly recovers the BY tree for these parts of the tree. So we must check that the colour and genus of $v$ is correctly recovered. Note that $w$ is even if and only if $v$ is even since we assumed that $f(x)$ has even degree, and so $v$ is is übereven when $T_C$ is rooted at $v$ if and only if it is übereven when $v_{\mathcal{R}}$ has not been removed and it is rooted at $v_{\mathcal{R}}$, hence the colour of $v$ is correctly determined. If the number of odd children of $v$ is odd when $\Tilde{T}_C$ is rooted at $v_{\mathcal{R}}$ the number of odd children goes up by $1$ when $v_{\mathcal{R}}$ is removed and $T_C$ is rooted at $v$ and so $g(v)$ stays the same. If the number of odd children of $v$ is even when $\Tilde{T}_C$ is rooted at $v_{\mathcal{R}}$ the number of odd children remains the same when $v_{\mathcal{R}}$ is removed and $T_C$ is rooted at $v$ and so $g(v)$ stays the same.

Thus, if $T_C$ is rooted at $v_{\mathcal{R}}$ or an adjacent vertex the above process uniquely determines the BY tree. We now show that the tree obtained by procedure described in the proposition does not depend on the initial choice of root vertex. This is due to the following observations. The first observation is that choosing a different root vertex does not change the parity of the size of a vertex. This means that the coloring of the edges remains unchanged, and thus their lengths. The second is that if an even vertex has only even children, then after designating a different vertex as the root it still has only even children. This means that the colouring of the vertices is unchanged. Finally, if a vertex has an odd number of odd children then after designating a different vertex as the root it either has the same number of odd children or one extra, and if a vertex has an even number of odd children then after changing to a different root vertex it has the same number of odd children or one fewer. This means that the genus of the vertex remains unchanged. This completes the proof.
\end{proof}

From Propositions \ref{bytotc} and \ref{byfromtc} we obtain the following, where by an isomorphism of BY trees we mean an isomorphism of the graph that preserves the colouring and genus of the vertices and the colouring of the edges.

\begin{corollary}
There is an explicit one-to-one correspondence between stable model trees up to isomorphism and BY trees up to isomorphism. 
\end{corollary}

Using the results of \cite{m2d2}, we will now see how the stable model tree $T_C$ can recover the dual graph of the special fibre of the minimal regular model of a semistable hyperelliptic curve $C/K$.

\begin{proposition}\label{semistablemain}\label{nonsemistablemain}
Let $C$ be a hyperelliptic curve over a local field $K$ of odd residue characteristic. Then $T_C$ uniquely determines
\begin{enumerate}[(i)]
    \item The dual graph of the special fibre of the minimal regular model of $C/K^{\textup{unr}}$ if $C$ is semistable;
    \item The dual graph of the special fibre of the potential stable model of $C$ if $C$ is not semsitable. 
\end{enumerate}
\end{proposition}

\begin{proof}
For $(i)$, note that by Theorem 5.18 of \cite{m2d2}, the dual graph of the special fibre of the minimal regular model of a semistable hyperelliptic curve is determined by the BY tree. By Proposition \ref{byfromtc}, the BY tree is uniquely determined by $T_C$.

For $(ii)$, let $F$ denote the finite extension of $K$ over which $C$ is semistable and let $e$ be the ramification degree. By Theorem 5.24 of \cite{m2d2}, the dual graph of the special fibre of the potential stable model of $C$ can be obtained from $T_{C/F}$ by allocating the length $1$ to all edges in the hyperelliptic graph $G_\Sigma$ (see Definition D.9 of \cite{m2d2}). Hence it can be obtained from $T_{C/K}$ by the same process, since $T_{C/F}$ is $T_{C/K}$ but with the lengths of the edges multiplied by $e$. 
\end{proof}

\begin{example}\label{clusterpictospecfib}
The curve
\begin{equation}
C:y^2=(x+p^2-p^5)(x+p^2+p^5)(x+p^3+p^7)(x+p^3-p^7)(x-1)(x-1-p)
\end{equation}
over $\mathbb{Q}_p$ for $p\geq 3$ is semistable and has the following cluster picture, stable model tree, BY tree and special fibre of its minimal regular model, where the lines indicate components that are isomorphic to $\mathbb{P}^1$.

\hspace{-34pt}\begin{minipage}[b]{0.31\textwidth}
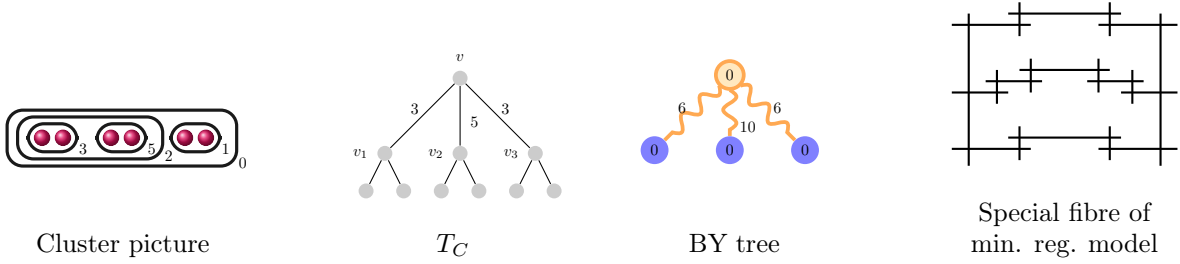
		\begin{center}
		\begin{figure}[H]
\scalebox{1.5}{
\clusterpicture             
  \Root[D] {1} {first} {r1};
  \Root[D] {} {r1} {r2};
  \Root[D] {3} {r2} {r3};
  \Root[D] {} {r3} {r4};
    \Root[D] {3.5} {r4} {r5};
    \Root[D] {} {r5} {r6};
  \ClusterLD c1[][3] = (r1)(r2);
    \ClusterLD c2[][5] = (r3)(r4);
    \ClusterLD c3[][1] = (r5)(r6);
    \ClusterLD c4[][2] = (c1)(c2);
  \ClusterD c5[0] = (c1)(c2)(c3)(c4);
\endclusterpicture}
\vspace{10pt}
\caption*{Cluster picture}
		\end{figure}
				\end{center}
    \end{minipage}\begin{minipage}[b]{0.23\textwidth}
		\begin{center}
		\begin{figure}[H]
			\begin{tikzpicture}
				[scale=0.5, auto=left,every node/.style={circle,fill=black!20,scale=0.6}]
                \node[label=above: {$v$}] (n10) at (2,2) {};
				\node[label=left: {$v_1$}] (n1) at (0,0) {};
				\node[label=left: {$v_2$}] (n2) at (2,0)  {};
				\node[label=left: {$v_3$}] (n3) at (4,0)  {};

                \node (n4) at (-0.5,-1) {};
                \node (n5) at (0.5,-1) {};
                
                \node (n6) at (1.5,-1) {};
                \node (n7) at (2.5,-1) {};

                \node (n8) at (3.5,-1) {};
                \node (n9) at (4.5,-1) {};

                \draw (n1) -- (n4); 
                \draw (n1) -- (n5); 
                
                \draw (n2) -- (n6); 
                \draw (n2) -- (n7); 

                \draw (n3) -- (n8); 
                \draw (n3) -- (n9);
				
				\draw (n10) -- (n1) node [text=black, pos=0.4, left, fill=none] {$3$};
                \draw (n10) -- (n2) node [text=black, pos=0.6, right,fill=none] {$5$};
				\draw (n10) -- (n3) node [text=black, pos=0.4, right,fill=none] {$3$};
			\end{tikzpicture}
 \caption*{$T_{C}$}
		\end{figure}
				\end{center}
    \end{minipage}\begin{minipage}[b]{0.23\textwidth}
		\begin{center}
		\begin{figure}[H]
\begin{tikzpicture}[scale=0.5, auto=left,every node/.style={circle,scale=0.6}]
                \node[line width=0.5mm,draw=bytreeyellowline,fill=bytreeyellowvertex] (n10) at (2,2) {$0$};
				\node[fill=bytreebluevertex] (n1) at (0,0) {$0$};
				\node[fill=bytreebluevertex] (n2) at (2,0)  {$0$};
				\node[fill=bytreebluevertex] (n3) at (4,0)  {$0$};

				\draw (n10) -- (n1)[line width=0.5mm,bytreeyellowline,decorate,decoration={coil,aspect=0}] node [text=black, pos=0.4, left, fill=none] {$6$};
                \draw (n10) -- (n2)[line width=0.5mm,bytreeyellowline,decorate,decoration={coil,aspect=0}] node [text=black, pos=0.8, right,fill=none] {$10$};
				\draw (n10) -- (n3)[line width=0.5mm,bytreeyellowline,decorate,decoration={coil,aspect=0}] node [text=black, pos=0.4, right,fill=none] {$6$};

\end{tikzpicture}
\vspace*{13pt}
\caption*{BY tree}
		\end{figure}
				\end{center}
    \end{minipage}\begin{minipage}[b]{0.31\textwidth}
		\begin{center}
		\begin{figure}[H]
			\begin{tikzpicture}
				[scale=0.3, auto=left,every node/.style={circle,fill=black!20,scale=0.6}]

                \draw[black, thick] (-0.5,5) -- (4.5,5) node [text=black, left, pos=0.5,fill=none] {};
                  \draw[black, thick] (0,5.5) -- (0,4) node [text=black, left, pos=0.5,fill=none] {};
                \draw[black, thick] (4,5.5) -- (4,4) node [text=black, left, pos=0.5,fill=none] {};
                \draw[black, thick] (0.5,4.5) -- (-3,4.5) node [text=black, left, pos=0.5,fill=none] {};
                \draw[black, thick] (3.5,4.5) -- (7,4.5) node [text=black, left, pos=0.5,fill=none] {};

                \draw[black, thick] (0,2.5) -- (4,2.5) node [text=black, left, pos=0.5,fill=none] {};
                \draw[black, thick] (0.5,1.5) -- (0.5,3) node [text=black, left, pos=0.5,fill=none] {};
                \draw[black, thick] (3.5,1.5) -- (3.5,3) node [text=black, left, pos=0.5,fill=none] {};               
                \draw[black, thick] (1,2) -- (-1.5,2) node [text=black, left, pos=0.5,fill=none] {};
                \draw[black, thick] (3,2) -- (5.5,2) node [text=black, left, pos=0.5,fill=none] {};
                \draw[black, thick] (-1,2.5) -- (-1,1) node [text=black, left, pos=0.5,fill=none] {};
                \draw[black, thick] (5,2.5) -- (5,1) node [text=black, left, pos=0.5,fill=none] {};
                \draw[black, thick] (-0.5,1.5) -- (-3,1.5) node [text=black, left, pos=0.5,fill=none] {};
                \draw[black, thick] (4.5,1.5) -- (7,1.5) node [text=black, left, pos=0.5,fill=none] {};

                \draw[black, thick] (-0.5,-0.5) -- (4.5,-0.5) node [text=black, left, pos=0.5,fill=none] {};
                  \draw[black, thick] (0,0) -- (0,-1.5) node [text=black, left, pos=0.5,fill=none] {};
                \draw[black, thick] (4,0) -- (4,-1.5) node [text=black, left, pos=0.5,fill=none] {};
                \draw[black, thick] (0.5,-1) -- (-3,-1) node [text=black, left, pos=0.5,fill=none] {};
                \draw[black, thick] (3.5,-1) -- (7,-1) node [text=black, left, pos=0.5,fill=none] {};

                \draw[black, thick] (-2.25,5) -- (-2.25,-1.75) node [text=black, left, pos=0.5,fill=none] {};
                 \draw[black, thick] (6.25,5) -- (6.25,-1.75) node [text=black, left, pos=0.5,fill=none] {};
        
			\end{tikzpicture}
 \caption*{Special fibre of min. reg. model}
		\end{figure}
				\end{center}
    \end{minipage}
We obtain the BY tree from $T_C$ as follows. Label the vertices $v$, $v_1$, $v_2$ and $v_3$ and root the tree at $v$. Following Proposition \ref{byfromtc}, $v$ is übereven so we colour it yellow, and $v_1$, $v_2$ and $v_3$ are not so we colour them blue. Then $g(v)=0$ and $g(v_i)=0$ for $i=1,2,3$ since $v$ has no odd children and $v_i$ has $2$ odd children for $i=1,2,3$. Since $v_1$, $v_2$ and $v_3$ are the children of $v$ and they all have even size, we colour the edges towards them yellow. To find the dual graph of the special fibre of the minimal regular model from the BY tree, according to \cite{m2d2} Theorem 5.18, one must take a second copy of the BY tree, attach the two copies at the blue vertices, ignore the genus $0$ blue vertices of degree $2$ and half the lengths. This tells us that the special fibre of the minimal regular model of $C/\mathbb{Q}_p^{\textup{unr}}$ looks like the formation of $\mathbb{P}^1$s above.
\end{example}

\section{The absolute invariants}\label{invariantssection}
In this section we will introduce the \textit{absolute invariants} associated to the possible stable model trees for a hyperelliptic curve of genus $g$. These will allow us to recover the stable model tree of a hyperelliptic curve in \S\ref{mainthminvproof}.

\begin{definition}\label{setofleafgraphs}
Let $\mathcal{T}_g$ denote the set of possible stable model trees, with the lengths of the edges omitted, for a hyperelliptic curve of genus $g$ that does not have potentially good reduction. The elements $T=(V,E)\in \mathcal{T}_g$, correspond to the possible stable models of a genus $g$ hyperelliptic curve over a local field, besides potentially good reduction. As in Definition \ref{stablemodeltreedefs}, we call the leaf vertices of $T$ \textit{singletons}, and vertices that are not singletons \textit{proper vertices}. If a proper vertex $v$ becomes a leaf if the singletons of $T$ are removed we call it a \textit{proper leaf}. 

Let $T\in \mathcal{T}_g$ and let $E_s$ denote the edges extending to the singleton vertices. Let
\begin{footnotesize}
\begin{equation}
I:\delta(v_1,v_2)=\dots=\delta(v_q,v_{q+1})>\delta(v_{q+3},v_{q+4})=\dots=\delta(v_{q+3},v_{q+4})> \dots > \delta(v_r,v_{r+1})=\dots=\delta(v_s,v_{s+1})
\end{equation}
\end{footnotesize}
\normalsize 
\noindent denote an ordering on the pairwise distances between proper vertices in $T$ for some $L:E\setminus E_s \rightarrow \mathbb{Q}^+$ and where there will be $v_i=v_j$ for some $i$ and $j$. An ordering
\begin{footnotesize}
\begin{equation}
I':\delta(v_1',v_2')=\dots=\delta(v_q',v_{q+1}')>\delta(v_{q+3}',v_{q+4}')=\dots=\delta(v_{q+3}',v_{q+4}')> \dots > \delta(v_r',v_{r+1}')=\dots=\delta(v_s',v_{s+1}')
\end{equation}
\end{footnotesize}
\normalsize 

\noindent is said to be equivalent to $I$ if there exists an isomorphism $\iota$ of $T$ for which $\iota(v_i)=v_i'$ for all proper vertices $v_i$ in $T$.

For a fixed labelling of the proper vertices of $T$, denote by $\textup{Ineq}_T$ the set of possible orderings on the pairwise distances between the vertices in $T$ as a labelled graph, up to equivalence under an isomorphism of the tree. Define
\begin{align}
\mathbf{T}_g=\{(T,I):T\in \mathcal{T}_g \textup{ and } I \in \textup{Ineq}_T\}. 
\end{align}
In other words, $\mathbf{T}_g$ is the set of all stable model trees with all possible orderings on the distances between vertices for a hyperelliptic curve of genus $g$. 
\end{definition}

\begin{definition}\label{hyperellipticordering}
Let $C$ by a hyperelliptic curve of genus $g$ over a local field $K$ of odd residue characteristic. Denote by $I_C\in\textup{Ineq}_{T_C}$ the ordering on the distances between the proper vertices of the stable model tree $T_C$ of $C$ for a fixed labelling of the edges. We can consider $(T_C,I_C)$ as an element of $\mathbf{T}_g$ by ignoring the lengths of the edges of $T_C$.
\end{definition}

\begin{definition}\label{stablemodeltreedefs}
Let $(T,I)\in\mathbf{T}_g$ with a fixed labelling of the singletons $s_1,\dots,s_{2g+2}$ of $T$.
\begin{enumerate}[(i)]
\item Denote by $\delta_n(T,I)$ the $n$-th largest distance between proper vertices in $T$ with the ordering $I$. 
\item Define
\begin{equation}
S_n(T,I)=\{\{\{s_i,s_j\},\{s_k,s_l\}\}: \textup{$i$, $j$, $k$ and $l$ are distinct and } \delta(C_{ij},C_{kl})=\delta_n(T,I)\},
\end{equation}
where we take all sets to be unordered, so $\{\{s_i,s_j\},\{s_k,s_l\}\}=\{\{s_j,s_i\},\{s_k,s_l\}\}=\{\{s_k,s_l\},\{s_i,s_j\}\}$.
\end{enumerate}
\end{definition}

We will now define the absolute invariants associated to each possible $(T,I)\in\mathbf{T}_g$, from which we will show the dual graph can be recovered.

\begin{definition}[Absolute invariants]\label{invariantsdefinition}
Let $(T,I)\in\mathbf{T}_g$ be a possible stable model tree for a hyperelliptic curve of genus $g$ with an ordering $I$ on the distances between the vertices. Let $d=2g+2$ and fix a labelling of the singletons $s_1,\dots,s_d$ of $T$ corresponding to variables $X_1,\dots,X_d$. Define
\begin{equation}
    (ij,kl)=\frac{(X_i-X_k)(X_i-X_l)(X_j-X_k)(X_j-X_l)}{(X_i-X_j)^2(X_k-X_l)^2}. 
\end{equation}
Denote by $n_{T,I}$ the number of distinct distances between vertices in $T$ with ordering $I$. For $1\leq n\leq n_{T,I}$, define
\begin{equation}
  R_{T,I,n}=\prod_{m=1}^{n}\prod_{\substack{\{\{s_i,s_j\},\{s_k,s_l\}\} \\ \in S_m(T,I)}} (ij,kl)^{n-m+1}.
\end{equation}
There is a natural action of $S_d$ on $R_{T,I,n}$ by letting $\sigma\in S_d$ take $X_i$ to $ X_{\sigma(i)}$. Define
\begin{equation}
    \textup{Inv}_{T,I,n}=\sum_{\sigma} R_{T,I,n}^{\sigma},
\end{equation}
where we take the sum over $\sigma\in S_d/\textup{Stab}_{S_d}(R_{T,I,n})$ under the action of $S_d$ on $R_{T,I,n}$. For a hyperelliptic curve $C:y^2=f(x)=c(x-x_1)\cdots(x-x_d)$ and a fixed labelling of the singletons of $T$ corresponding to the roots of $f(x)$, write $R_{T,I,n}(f)^\sigma$ and $\textup{Inv}_{T,I,n}(C)$ for $R_{T,I,n}^\sigma(x_1,\dots,x_d)$ and $\textup{Inv}_{T,I,n}(x_1,\dots,x_d)$ respectively. 
\end{definition}

\begin{notation}
We will write $\sum_{S_d/\textup{Stab}}R$ for $\sum_{\sigma\in S_d/\textup{Stab}(R)}R^{\sigma}$.
\end{notation}

Note that the invariants $\textup{Inv}_{T,I,n}(C)$ are model independent, so it makes sense to evaluate them on a curve $C$, as opposed to on a Weierstrass equation for the curve:

\begin{lemma}\label{actuallyareinvariants}
If $C/K$ and $C'/K$ are two genus $g$ hyperelliptic curves that are isomorphic over $\Bar{K}$ then $\textup{Inv}_{T,I,n}(C)=\textup{Inv}_{T,I,n}(C')$ for every $(T,I)\in \mathbf{T}_g$.
\end{lemma}

\begin{proof}
In terms of the Weierstrass equation, any isomorphism from $C:y^2=f(x)$ to $C':y^2=g(x)$ defined over $\Bar{K}$ takes $x$ to $\frac{ax+b}{cx+d}$, where $a$, $b$, $c$ and $d\in \Bar{K}$ (see, for example, \cite{mobius} \S 1.5.1). A simple check shows that $(ij,kl)$ is invariant under this action, and thus so is $\textup{Inv}_{T,I,n}(C)$.
\end{proof}

\begin{lemma}\label{finitenumberofleafgraphs}
The set $\mathbf{T}_g$ is finite. 
\end{lemma}

\begin{proof}
For fixed $g$, there are a finite number of cluster pictures of a polynomial of degree $2g+2$. Hence, there are a finite number of stable model trees with lengths omitted in $\mathcal{T}_g$. Since each stable model tree has a finite number of vertices, there are a finite number of possible orderings on the distances between the proper vertices, and so $\mathbf{T}_g$ is finite. 
\end{proof}

The main theorem that we prove in this paper is the following, that the absolute invariants defined in this section recover the dual graph of the special fibre of the minimal regular model of a semistable hyperelliptic curve. The remaining sections are concerned with the proof of this theorem. 

\begin{theorem}[=Corollary \ref{invthm}]
Let $C$ be a hyperelliptic curve of genus $g$ over a local field of odd residue chatacteristic $K$. The finite set 
\begin{equation}
\left\{(T,I,\textup{ord}(\textup{Inv}_{T,I,n}(C))):(T,I)\in \mathbf{T}_{g} \textup{ and } n\leq n_{T,I}\right\},
\end{equation}
where $n_{T,I}$ is the number of distinct distances between the vertices in $T$ with ordering $I$, uniquely determines:
\begin{enumerate}[(i)]
      \item The dual graph of the special fibre of the minimal regular model of $C/K^{\textup{unr}}$ if $C$ is semistable;
    \item The dual graph of the special fibre of the potential stable model of $C$ if $C/K$ is not semistable. 
\end{enumerate}
\end{theorem}

There are two main ideas behind the construction of the invariants. The first was to write an expression in terms of the differences of roots that is indeed an absolute invariant of the curve, as show above in Lemma \ref{actuallyareinvariants}. The second was to fashion absolute invariants that `pick out' the distances between the vertices in $T_C$ when evaluated on $C$, as described in Proposition \ref{distancevalues}.

\begin{example}\label{stablemodeltreedefex}
By writing down the possible cluster pictures without depths for a degree $6$ polynomial (see, for example, \cite{m2d2} Table 8), the trees in $\mathcal{T}_2$ are: 
\vspace{10pt}
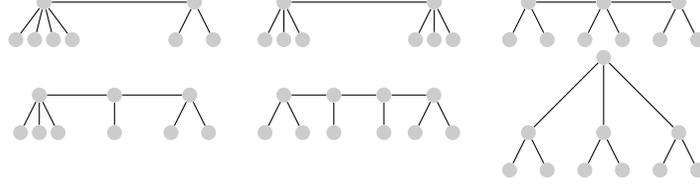
\begin{center}
\begin{minipage}[m]{0.2\textwidth}
		\begin{center}
\begin{tikzpicture}
				[scale=0.5, auto=left,every node/.style={circle,fill=black!20,scale=0.6}]
				\node (n1) at (0,0) {};
				\node (n3) at (4,0)  {};

                \node (n4) at (-0.25,-1)  {};
                \node (n5) at (0.25,-1)  {};
                \node (n6) at (3.5,-1)  {};
                \node (n7) at (4.5,-1)  {};

                \node (n8) at (-0.75,-1)  {};
                \node (n9) at (0.75,-1)  {};

				\draw (n1) -- (n4) node [text=black, pos=0.4, left, fill=none] {};
                \draw (n1) -- (n5) node [text=black, pos=0.4, left, fill=none] {};
                \draw (n1) -- (n8) node [text=black, pos=0.4, left, fill=none] {};
                \draw (n1) -- (n9) node [text=black, pos=0.4, left, fill=none] {};
                \draw (n3) -- (n6) node [text=black, pos=0.4, left, fill=none] {};
                \draw (n3) -- (n7) node [text=black, pos=0.4, left, fill=none] {};
                
				\draw (n1) -- (n3) node [text=black, pos=0.4, left, fill=none] {};
        
			\end{tikzpicture}
\end{center}
\end{minipage}\begin{minipage}[m]{0.2\textwidth}
		\begin{center}
\begin{tikzpicture}
				[scale=0.5, auto=left,every node/.style={circle,fill=black!20,scale=0.6}]
				\node (n1) at (0,0) {};
				\node (n3) at (4,0)  {};

                \node (n4) at (-0.5,-1)  {};
                \node (n5) at (0.5,-1)  {};
                \node (n6) at (3.5,-1)  {};
                \node (n7) at (4.5,-1)  {};
                \node (n8) at (0,-1)  {};
                 \node (n9) at (4,-1)  {};

                \draw (n1) -- (n4) node [text=black, pos=0.4, left, fill=none] {};
                \draw (n1) -- (n5) node [text=black, pos=0.4, left, fill=none] {};
                 \draw (n1) -- (n8) node [text=black, pos=0.4, left, fill=none] {};
                \draw (n3) -- (n6) node [text=black, pos=0.4, left, fill=none] {};
                \draw (n3) -- (n7) node [text=black, pos=0.4, left, fill=none] {};
			\draw (n3) -- (n9) node [text=black, pos=0.4, left, fill=none] {};
            
				\draw (n1) -- (n3) node [text=black, pos=0.4, left, fill=none] {};
        
			\end{tikzpicture}
\end{center}
\end{minipage}\begin{minipage}[m]{0.2\textwidth}
		\begin{center}
\begin{tikzpicture}
				[scale=0.5, auto=left,every node/.style={circle,fill=black!20,scale=0.6}]
				\node (n1) at (0,0) {};
				\node (n2) at (2,0)  {};
				\node (n3) at (4,0)  {};

                \node (n4) at (-0.5,-1)  {};
                \node (n5) at (0.5,-1)  {};
                \node (n6) at (1.5,-1)  {};
                \node (n7) at (2.5,-1)  {};
                \node (n8) at (3.5,-1)  {};
                 \node (n9) at (4.5,-1)  {};

                 \draw (n1) -- (n4) node [text=black, pos=0.4, left, fill=none] {};
                 \draw (n1) -- (n5) node [text=black, pos=0.4, left, fill=none] {};

                 \draw (n2) -- (n6) node [text=black, pos=0.4, left, fill=none] {};
                 \draw (n2) -- (n7) node [text=black, pos=0.4, left, fill=none] {};

                 \draw (n3) -- (n8) node [text=black, pos=0.4, left, fill=none] {};
                 \draw (n3) -- (n9) node [text=black, pos=0.4, left, fill=none] {};
				
				\draw (n1) -- (n2) node [text=black, pos=0.4, left, fill=none] {};
    \draw (n2) -- (n3) node [text=black, pos=0.4, left, fill=none] {};
        
			\end{tikzpicture}
\end{center}
\end{minipage}

\begin{minipage}[m]{0.2\textwidth}
		\begin{center}
\begin{tikzpicture}
				[scale=0.5, auto=left,every node/.style={circle,fill=black!20,scale=0.6}]
				\node (n1) at (0,0) {};
				\node (n2) at (2,0)  {};
				\node (n3) at (4,0)  {};

                \node (n4) at (-0.5,-1)  {};
                \node (n5) at (0.5,-1)  {};
                \node (n6) at (0,-1)  {};
                
                \node (n7) at (2,-1)  {};
                
                \node (n8) at (3.5,-1)  {};
                 \node (n9) at (4.5,-1)  {};

                 \draw (n1) -- (n4) node [text=black, pos=0.4, left, fill=none] {};
                 \draw (n1) -- (n5) node [text=black, pos=0.4, left, fill=none] {};

                 \draw (n1) -- (n6) node [text=black, pos=0.4, left, fill=none] {};
                 
                 \draw (n2) -- (n7) node [text=black, pos=0.4, left, fill=none] {};

                 \draw (n3) -- (n8) node [text=black, pos=0.4, left, fill=none] {};
                 \draw (n3) -- (n9) node [text=black, pos=0.4, left, fill=none] {};
				
				\draw (n1) -- (n2) node [text=black, pos=0.4, left, fill=none] {};
    \draw (n2) -- (n3) node [text=black, pos=0.4, left, fill=none] {};
        
			\end{tikzpicture}
\end{center}
\end{minipage}\begin{minipage}[m]{0.2\textwidth}
		\begin{center}
\begin{tikzpicture}
				[scale=0.5, auto=left,every node/.style={circle,fill=black!20,scale=0.6}]
				\node (n1) at (0,0) {};
				\node (n2) at (1.33,0)  {};
				\node (n3) at (2.66,0)  {};
                    \node (n4) at (3.99,0)  {};

                    \node (n5) at (-0.5,-1)  {};
                \node (n6) at (0.5,-1)  {};
                
                \node (n7) at (1.33,-1)  {};
                
                \node (n8) at (2.66,-1)  {};
                
                \node (n9) at (3.49,-1)  {};
                 \node (n10) at (4.49,-1)  {};

                 \draw (n1) -- (n5) node [text=black, pos=0.4, left, fill=none] {};
                 \draw (n1) -- (n6) node [text=black, pos=0.4, left, fill=none] {};

                 \draw (n2) -- (n7) node [text=black, pos=0.4, left, fill=none] {};
                 
                 \draw (n3) -- (n8) node [text=black, pos=0.4, left, fill=none] {};

                 \draw (n4) -- (n9) node [text=black, pos=0.4, left, fill=none] {};
                 \draw (n4) -- (n10) node [text=black, pos=0.4, left, fill=none] {};
				
				\draw (n1) -- (n2) node [text=black, pos=0.4, left, fill=none] {};
    \draw (n2) -- (n3) node [text=black, pos=0.4, left, fill=none] {};
    \draw (n3) -- (n4) node [text=black, pos=0.4, left, fill=none] {};
        
			\end{tikzpicture}
\end{center}
\end{minipage}\begin{minipage}[m]{0.2\textwidth}
		\begin{center}
\begin{tikzpicture}
				[scale=0.5, auto=left,every node/.style={circle,fill=black!20,scale=0.6}]
                \node (n10) at (2,2) {};
				\node (n1) at (0,0) {};
				\node (n2) at (2,0)  {};
				\node (n3) at (4,0)  {};

                 \node (n4) at (-0.5,-1) {};
                \node (n5) at (0.5,-1) {};
                
                \node (n6) at (1.5,-1) {};
                \node (n7) at (2.5,-1) {};

                \node (n8) at (3.5,-1) {};
                \node (n9) at (4.5,-1) {};

                \draw (n1) -- (n4); 
                \draw (n1) -- (n5); 
                
                \draw (n2) -- (n6); 
                \draw (n2) -- (n7); 

                \draw (n3) -- (n8); 
                \draw (n3) -- (n9);
				
				\draw (n10) -- (n1) node [text=black, pos=0.4, left, fill=none] {};
                \draw (n10) -- (n2) node [text=black, pos=0.6, right,fill=none] {};
				\draw (n10) -- (n3) node [text=black, pos=0.4, right,fill=none] {};
			\end{tikzpicture}
\end{center}
\end{minipage}
\end{center}
Let us take $T$ to be the tree on the bottom right, and let us label the edges non extending to singleton vertices with arbitrary lengths $a$, $b$ and $c$. The possibilities for the ordering on the distances between the proper vertices in $T$ are shown in $\textup{Ineq}_T$ below. These are written in the simplest form possible in terms of $a$, $b$ and $c$. 
\begin{center}
\begin{minipage}{0.2\textwidth}
\begin{figure}[H]
			\begin{tikzpicture}
				[scale=0.5, auto=left,every node/.style={circle,fill=black!20,scale=0.6}]
                \node (n10) at (2,2) {};
				\node (n1) at (0,0) {};
				\node (n2) at (2,0)  {};
				\node (n3) at (4,0)  {};

                \node[label=below: {$s_1$}] (n4) at (-0.5,-1) {};
                \node[label=below: {$s_2$}] (n5) at (0.5,-1) {};
                
                \node[label=below: {$s_3$}] (n6) at (1.5,-1) {};
                \node[label=below: {$s_4$}] (n7) at (2.5,-1) {};

                \node[label=below: {$s_5$}] (n8) at (3.5,-1) {};
                \node[label=below: {$s_6$}] (n9) at (4.5,-1) {};

                \draw (n1) -- (n4); 
                \draw (n1) -- (n5); 
                
                \draw (n2) -- (n6); 
                \draw (n2) -- (n7); 

                \draw (n3) -- (n8); 
                \draw (n3) -- (n9);
				
				\draw (n10) -- (n1) node [text=black, pos=0.4, left, fill=none] {$a$};
                \draw (n10) -- (n2) node [text=black, pos=0.6, right,fill=none] {$b$};
				\draw (n10) -- (n3) node [text=black, pos=0.4, right,fill=none] {$c$};
			\end{tikzpicture}
   \caption*{$T$}
   \end{figure}
   \end{minipage}\begin{minipage}{0.7\textwidth}
\begin{align}
\textup{Ineq}_T=\{&\{a=b=c\},\{a>b=c, \ a>2b\},\{a>b=c, \ a=2b\}, \\
&\{a>b=c, \ a<2b\}, \{a=b>c\},\{a>b>c, \ a>b+c\}, \\
&\{a>b>c,\  a=b+c\},\{a>b>c, \ a<b+c\}\}.
\end{align}
   \end{minipage}
\end{center}
Note that, for example, $\{a>b>c, \ a=b+c\}$ is equivalent to $\{a>c>b, \ a=b+c\}$ since there is an isomorphism of the tree that takes the edge labelled with $b$ to the edge labelled with $c$, and we only include orderings up to equivalence under isomorphisms of the tree.

Let us take $(T,I)\in \mathcal{T}_2$, where $T$ is the tree above and $I=\{a>b>c, \ a=b+c\}\in \textup{Ineq}_T$. In this labelling of the edges of $T$, we have $\delta_1(T,I)=a+b$ since $a>b>c$ so $a+b$ is the greatest distance in $T$. In the same vein, $\delta_2(T,I)=a+c$, $\delta_3(T,I)=a=b+c$, $\delta_4(T,I)=b$ and $\delta_5(T,I)=c$. In the fixed labelling of the singletons shown above, $\delta_1(T,I)=a+b$ appears as the distance between the paths $C_{12}$ and $C_{34}$ only, whence $S_{1}(T,I)= \{\{\{s_1,s_2\},\{s_3,s_4\}\}\}$. Similarly, 
\begin{align}
S_{2}(T,I)&= \{\{\{s_1,s_2\},\{s_5,s_6\}\}\} ;\\
S_{3}(T,I)&=\{\{\{s_1,s_2\},\{s_3,s_5\}\},\{\{s_1,s_2\},\{s_3,s_6\}\}, \\
&\hspace{103pt}\{\{s_1,s_2\},\{s_4,s_5\}\},\{\{s_1,s_2\},\{s_4,s_6\}\},\{\{s_3,s_4\},\{s_5,s_6\}\}\}; \\
S_{4}(T,I)&=\{\{\{s_1,s_5\},\{s_3,s_4\}\},\{\{s_1,s_6\},\{s_3,s_4\}\},\{\{s_2,s_5\},\{s_3,s_4\}\},\{\{s_2,s_6\},\{s_3,s_4\}\}\}; \\
S_{5}(T,I)&=\{\{\{s_1,s_3\},\{s_5,s_6\}\},\{\{s_1,s_4\},\{s_5,s_6\}\},\{\{s_2,s_3\},\{s_5,s_6\}\},\{\{s_2,s_4\},\{s_5,s_6\}\}\}.
\end{align}
Now, we have 
\begin{equation}
R_{T,I,1}=\frac{(X_1-X_3)(X_1-X_4)(X_2-X_3)(X_2-X_4)}{(X_1-X_2)^2(X_3-X_4)^2},
\end{equation}
and so 
\begin{equation}
\textup{Inv}_{T,I,1}=\frac{(X_1-X_3)(X_1-X_4)(X_2-X_3)(X_2-X_4)}{(X_1-X_2)^2(X_3-X_4)^2}+\frac{(X_1-X_5)(X_1-X_6)(X_2-X_5)(X_2-X_6)}{(X_1-X_2)^2(X_5-X_6)^2}+\dots,
\end{equation}
where the sum runs over elements of $S_6$ modulo the elements that stabilise $R_{T,I,1}$; there will be $45$ summands in $\textup{Inv}_{T,I,1}$ for this example. Then for $n=2$ and $n=3$, in this labelling of the singletons the summands are
\begin{align}
R_{T,I,2}&=\frac{(X_1-X_3)^2(X_1-X_4)^2(X_2-X_3)^2(X_2-X_4)^2}{(X_1-X_2)^4(X_3-X_4)^4} \cdot \frac{(X_1-X_5)(X_1-X_6)(X_2-X_5)(X_2-X_6)}{(X_1-X_2)^2(X_5-X_6)^2} ;\\
R_{T,I,3}&=\frac{(X_1-X_3)^3(X_1-X_4)^3(X_2-X_3)^3(X_2-X_4)^3}{(X_1-X_2)^6(X_3-X_4)^6} \cdot \frac{(X_1-X_5)^2(X_1-X_6)^2(X_2-X_5)^2(X_2-X_6)^2}{(X_1-X_2)^4(X_5-X_6)^4} \\
&\cdot \prod_{\substack{\{\{s_i,s_j\},\{s_k,s_l\}\} \\ \in S_{3}(T,I)}} \frac{(X_i-X_k)(X_i-X_l)(X_j-X_k)(X_j-X_l)}{(X_i-X_j)^2(X_k-X_l)^2}. 
\end{align}
Again, in order to form $\textup{Inv}_{T,I,2}$ and $\textup{Inv}_{T,I,3}$ we take the sum over $S_6$ modulo stabilisers of $R_{T,I,2}$ and $R_{T,I,3}$ respectively.  
\end{example}

\begin{remark}
Writing down all the possible orderings on the distances between vertices in a tree appears to be a non-trivial problem. For instance, if one has a weighted tree $T$ and four vertices $v_1$, $v_2$, $v_3$ and $v_4$ in $T$, it is not possible that $\delta(v_1,v_2)>\delta(v_3,v_4)>\delta(v_1,v_3),\delta(v_1,v_4),\delta(v_2,v_3),\delta(v_2,v_4)$, since this contradicts the four-point condition of trees (see, for example, the introduction to \cite{buneman}). It should be viable to axiomatise the possible orderings using the four-point condition. 
\end{remark}

\section{Valuation of the absolute invariant associated to a stable model tree}\label{invariantsofcurves}

The main result that we prove in this section is Theorem \ref{valuationofyprime} below, which relates the valuation of the $n$-th absolute invariant associated to $(T_C,I_C)$ to the distances between vertices $\delta_1(C),\dots,\delta_n(C)$ in $T_C$.

\begin{definition}\label{treedefinitions}
Let $T_C$ be the stable model tree of a hyperelliptic curve $C$ with a fixed labelling $s_1,\dots,s_d$ of the singletons of $T_C$.
\begin{enumerate}[(i)]
\item Denote the $n$-th largest distance between non-singleton vertices in $T_C$ by $\delta_n(C)$.
\item Define
\begin{equation}
K_{n}(C)=\#\{\{\{s_i,s_j\},\{s_k,s_l\}\}: \textup{$i$, $j$, $k$ and $l$ are distinct and } \delta(C_{ij},C_{kl})=\delta_n(C)\},
\end{equation}
where we take all sets to be unordered, see Definition \ref{distancepaths} for the definition of $\delta(C_{ij},C_{kl})$.
\end{enumerate}
\end{definition}

\begin{theorem}\label{valuationofyprime}
Let $T_C$ be the stable model tree of a hyperelliptic curve $C$ over a local field $K$. For $n\geq 0$
\begin{align}
    \textup{ord}(\textup{Inv}_{T_C,I_c,n}(C))=-2\sum_{m=1}^{n} K_{m}(C)(n+1-m)\delta_{m}(C).
\end{align}
\end{theorem}

\begin{example}
For a stable model tree given by 

\begin{center}
		\begin{figure}[H]
			\begin{tikzpicture}
				[scale=0.5, auto=left,every node/.style={circle,fill=black!20,scale=0.6}]
                \node[label=above: {$v$}] (n10) at (2,2) {};
				\node[label=left: {$v_1$}] (n1) at (0,0) {};
				\node[label=right: {$v_2$}] (n2) at (2,0)  {};
				\node[label=right: {$v_3$}] (n3) at (4,0)  {};

                \node[label=below: {$s_1$}] (n4) at (-0.5,-1) {};
                \node[label=below: {$s_2$}] (n5) at (0.5,-1) {};
                
                \node[label=below: {$s_3$}] (n6) at (1.5,-1) {};
                \node[label=below: {$s_4$}] (n7) at (2.5,-1) {};

                \node[label=below: {$s_5$}] (n8) at (3.5,-1) {};
                \node[label=below: {$s_6$}] (n9) at (4.5,-1) {};

                \draw (n1) -- (n4); 
                \draw (n1) -- (n5); 
                
                \draw (n2) -- (n6); 
                \draw (n2) -- (n7); 

                \draw (n3) -- (n8); 
                \draw (n3) -- (n9);
				
				\draw (n10) -- (n1) node [text=black, pos=0.4, left, fill=none] {$3$};
                \draw (n10) -- (n2) node [text=black, pos=0.6, right,fill=none] {$5$};
				\draw (n10) -- (n3) node [text=black, pos=0.4, right,fill=none] {$3$};
			\end{tikzpicture}
 \caption*{$T_{C}$}
		\end{figure}
\end{center}
we have $\delta_1(C)=8$, $\delta_2(C)=6$, $\delta_3(C)=5$ and $\delta_4(C)=3$. Then
\begin{equation}
K_1(C)=\#\{\{\{s_1,s_2\},\{s_3,s_4\}\},\{\{s_3,s_4\},\{s_5,s_6\}\} \}=2
\end{equation}
since there is a path of length $\delta_1(C)=8$ between $C_{12}$ and $C_{34}$ and between $C_{34}$ and $C_{56}$. Similarly, 
\begin{align}
K_2(C)&=\#\{\{\{s_1,s_2\},\{s_5,s_6\}\}\}=1; \\
K_3(C)&=\#\{\{\{s_3,s_4\},\{s_1,s_5\}\},\{\{s_3,s_4\},\{s_1,s_6\}\},\{\{s_3,s_4\},\{s_2,s_5\}\},\{\{s_3,s_4\},\{s_2,s_6\}\}\}=4; \\
K_{4}(C)&=\#\{\{\{s_1,s_2\},\{s_3,s_5\}\},\{\{s_1,s_2\},\{s_3,s_6\}\}, \{\{s_1,s_2\},\{s_4,s_5\}\}, \{\{s_1,s_2\},\{s_4,s_6\}\}, \\
&\hspace{27pt}\{\{s_5,s_6\},\{s_1,s_3\}\},\{\{s_5,s_6\},\{s_1,s_4\}\}, \{\{s_5,s_6\},\{s_2,s_3\}\}, \{\{s_5,s_6\},\{s_2,s_4\}\} \}=8.
\end{align}
\end{example}

In order to prove Theorem \ref{valuationofyprime}, we will need the following results, the first of which shows that the valuation of the building blocks $\textup{ord}((ij,kl))$ of the absolute invariants `pick out' the distances between vertices in $T_C$.

\begin{proposition}\label{distancevalues}
Let $C$ be a hyperelliptic curve of genus $g$ over a local field $K$ of odd residue characteristic and fix a model $C:y^2=f(x)$, where $f(x)$ has degree $2g+2$. Fix a labelling of the singletons $s_1,\dots,s_d$ of $T_C$ corresponding to the roots $x_1,\dots,x_d$ of $f(x)$. Denote by $C_{ij}$ the unique path between $s_i$ and $s_j$ and by $C_{kl}$ the unique path between $s_k$ and $s_l$ in $T_C$. Then
\begin{equation}
\textup{ord}\left(\frac{(x_i-x_k)(x_i-x_l)(x_j-x_k)(x_j-x_l)}{(x_i-x_j)^2(x_k-x_l)^2}\right)=\begin{cases} 
-2\delta(v,w) & \textup{ if } C_{ij}\textup{ and } C_{kl}\textup{ have no common edges}; \\
> 0 & \textup{ otherwise}, \\
\end{cases}
\end{equation}
where the unique path between $C_{ij}$ and $C_{kl}$ and intersecting each path at one vertex starts at vertex $v$ and ends at $w$. 
\end{proposition}

\begin{proof}
First, suppose $C_{ij}$ and $C_{kl}$ have no common edges. In this case, the stable model tree looks like the tree below, where the dashed lines represent the fact there could be any tree structure in between the vertices, and where it's possible that the vertices connected by a dotted edge are equal. We have indicated the location of the singletons $s_i$, $s_j$, $s_k$ and $s_l$ inside the vertex to which there is an edge going from said singleton. We will check that the correct valuation is recovered for different possibilities of where the vertex $v_\mathcal{R}$ could be in the tree, before it is possibly removed to form the stable model tree (see Definition \ref{stablemodeltreedef}). The three possibilities for the location of $v_{\mathcal{R}}$ have been indicated by black vertices in the tree below, where it is possible that $v_{\mathfrak{s}}=v_{\mathcal{R}}$. 
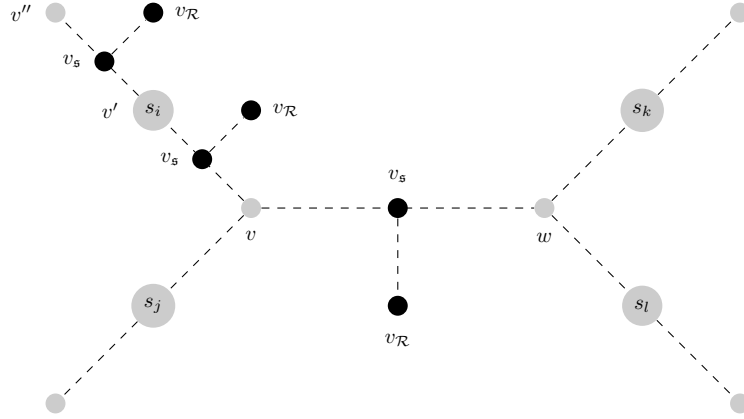
\begin{center}
			\begin{tikzpicture}
				[scale=1.3, auto=left,every node/.style={circle,fill=black!20,scale=0.8}]
               
			\node[label=below: {$v$}] (n1) at (0,0)  {};
                \node[label=below: {$w$}] (n2) at (3,0)  {};
                
                \node[color=black,label=above: {$v_{\mathfrak{s}}$}] (n12) at (1.5,0)  {};
                \node[color=black, label=below: {$v_{\mathcal{R}}$}] (n13) at (1.5,-1)  {};
                \draw[dashed] (n12) -- (n13);

                \node[color=black,label=left: {$v_{\mathfrak{s}}$}] (n14) at (-0.5,0.5)  {};
                \node[color=black,label=right: {$v_{\mathcal{R}}$}] (n15) at (0,1)  {};
                \draw[dashed] (n14) -- (n15);
                     
                \node[label=left: {$v'$}] (n3) at (-1,1) {$s_i$};
                
                \node (n4) at (-1,-1) {$s_j$};

                \node (n5) at (4,1) {$s_k$};
                
                \node (n6) at (4,-1) {$s_l$};

                \node[color=black,label=left: {$v_{\mathfrak{s}}$}] (n16) at (-1.5,1.5)  {};
                \node[color=black,label=right: {$v_{\mathcal{R}}$}] (n17) at (-1,2)  {};
                \draw[dashed] (n16) -- (n17);

                \node[label=left: {$v''$}] (n7) at (-2,2) {};
                \node (n8) at (-2,-2) {};

                \node (n9) at (5,2) {};
                \node (n10) at (5,-2) {};

				\draw[dashed]  (n1) -- (n2);

                \draw[dashed] (n1) -- (n3); 
                \draw[dashed] (n1) -- (n4);

                \draw[dashed] (n2) -- (n5); 
                \draw[dashed] (n2) -- (n6);

                \draw[dashed] (n3) -- (n7); 
                \draw[dashed] (n4) -- (n8); 

                \draw[dashed] (n5) -- (n9); 
                \draw[dashed] (n6) -- (n10); 
                
			\end{tikzpicture}
\end{center}
If $v_{\mathcal{R}}$ is between $v$ and $w$ then by the definition of the stable model tree and Lemma \ref{distanceintc}
\begin{align}
\textup{ord}\left(\frac{(x_i-x_k)(x_i-x_l)(x_j-x_k)(x_j-x_l)}{(x_i-x_j)^2(x_k-x_l)^2}\right)&= 4(d_{\mathcal{R}}+\delta(v_{\mathcal{R}},v_{\mathfrak{s}}))-2(d_{\mathcal{R}}+\delta(v_{\mathcal{R}},v))-2(d_{\mathcal{R}}+\delta(v_{\mathcal{R}},w)) \\
&=-2\delta(v,w). 
\end{align}
If $v_{\mathcal{R}}$ is between $v'$ and $v$, possibly with $v=v_{\mathcal{R}}$, then  
\begin{align}
\textup{ord}\left(\frac{(x_i-x_k)(x_i-x_l)(x_j-x_k)(x_j-x_l)}{(x_i-x_j)^2(x_k-x_l)^2}\right)=2(d_{\mathcal{R}}&+\delta(v_{\mathcal{R}},v_{\mathfrak{s}}))+2(d_{\mathcal{R}}+\delta(v_{\mathcal{R}},v))-2(d_{\mathcal{R}}+\delta(v_{\mathcal{R}},v_{\mathfrak{s}}))\\
&-2(d_{\mathcal{R}}+\delta(v_{\mathcal{R}},v)+\delta(v,w))=-2\delta(v,w). 
\end{align}
If $v_{\mathcal{R}}$ is between $v''$ and $v'$, then 
\begin{align}
\textup{ord}\left(\frac{(x_i-x_k)(x_i-x_l)(x_j-x_k)(x_j-x_l)}{(x_i-x_j)^2(x_k-x_l)^2}\right)&=2(d_{\mathcal{R}} +\delta(v_{\mathcal{R}},v'))+ 2(d_{\mathcal{R}} +\delta(v_{\mathcal{R}},v))-2(d_{\mathcal{R}}+\delta(v_{\mathcal{R}},v')) \\
&-2(d_{\mathcal{R}}+\delta(v_{\mathcal{R}},v)+\delta(v,w))=-2\delta(v,w). 
\end{align}

Now suppose $C_{ij}$ and $C_{kl}$ have common edges. In this case, the stable model tree looks like the following.
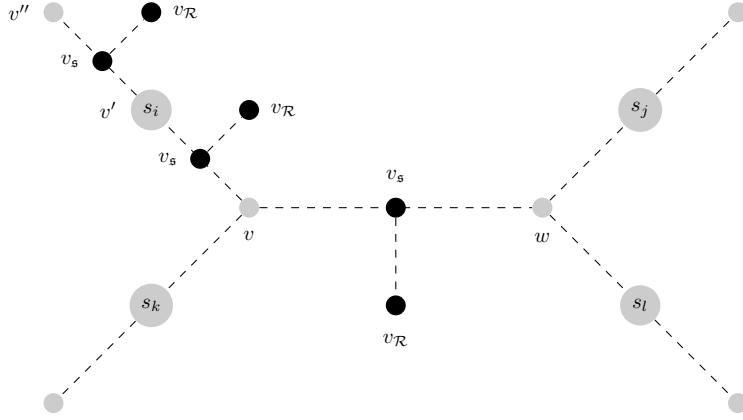
\begin{center}
			\begin{tikzpicture}
				[scale=1.3, auto=left,every node/.style={circle,fill=black!20,scale=0.8}]
               
			\node[label=below: {$v$}] (n1) at (0,0)  {};
                \node[label=below: {$w$}] (n2) at (3,0)  {};
                
                \node[color=black,label=above: {$v_{\mathfrak{s}}$}] (n12) at (1.5,0)  {};
                \node[color=black, label=below: {$v_{\mathcal{R}}$}] (n13) at (1.5,-1)  {};
                \draw[dashed] (n12) -- (n13);

                \node[color=black,label=left: {$v_{\mathfrak{s}}$}] (n14) at (-0.5,0.5)  {};
                \node[color=black,label=right: {$v_{\mathcal{R}}$}] (n15) at (0,1)  {};
                \draw[dashed] (n14) -- (n15);
                     
                \node[label=left: {$v'$}] (n3) at (-1,1) {$s_i$};
                
                \node (n4) at (-1,-1) {$s_k$};

                \node (n5) at (4,1) {$s_j$};
                
                \node (n6) at (4,-1) {$s_l$};

                \node[color=black,label=left: {$v_{\mathfrak{s}}$}] (n16) at (-1.5,1.5)  {};
                \node[color=black,label=right: {$v_{\mathcal{R}}$}] (n17) at (-1,2)  {};
                \draw[dashed] (n16) -- (n17);

                \node[label=left: {$v''$}] (n7) at (-2,2) {};
                \node (n8) at (-2,-2) {};

                \node (n9) at (5,2) {};
                \node (n10) at (5,-2) {};

				\draw[dashed]  (n1) -- (n2);

                \draw[dashed] (n1) -- (n3); 
                \draw[dashed] (n1) -- (n4);

                \draw[dashed] (n2) -- (n5); 
                \draw[dashed] (n2) -- (n6);

                \draw[dashed] (n3) -- (n7); 
                \draw[dashed] (n4) -- (n8); 

                \draw[dashed] (n5) -- (n9); 
                \draw[dashed] (n6) -- (n10); 
                
			\end{tikzpicture}
\end{center}
We have labelled the start and end vertices on the path that is the intersection of $C_{ij}$ and $C_{kl}$ as $v$ and $w$. As in the previous case it is possible that vertices connected by dotted edges are equal, except in this case $v\neq w$ since $C_{ij}$ and $C_{kl}$ are assumed to have common edges. If $v_{\mathcal{R}}$ is between $v$ and $w$, 
\begin{align}
\textup{ord}\left(\frac{(x_i-x_k)(x_i-x_l)(x_j-x_k)(x_j-x_l)}{(x_i-x_j)^2(x_k-x_l)^2}\right)=d_{\mathcal{R}}&+\delta(v_{\mathcal{R}},v)+2(d_{\mathcal{R}}+\delta(v,v_{\mathcal{R}})) + d_{\mathcal{R}}+\delta(v_{\mathcal{R}},w) \\
&-4(d_{\mathcal{R}}+\delta(v_{\mathcal{R}},v_{\mathfrak{s}}))=\delta(v,w)>0. 
\end{align}
If $v_{\mathcal{R}}$ is between $v'$ and $v$ (possibly $v=v_{\mathcal{R}}$) then 
\begin{align}
\textup{ord}\left(\frac{(x_i-x_k)(x_i-x_l)(x_j-x_k)(x_j-x_l)}{(x_i-x_j)^2(x_k-x_l)^2}\right)=2(d_{\mathcal{R}}&+\delta(v_{\mathcal{R}},v_{\mathfrak{s}}))+d_{\mathcal{R}}+\delta(v_{\mathcal{R}},v)+d_{\mathcal{R}}+\delta(v_{\mathcal{R}},w) \\
&-2(d_{\mathcal{R}}+\delta(v_{\mathcal{R}},v_{\mathfrak{s}}))-2(d_{\mathcal{R}}+\delta(v_{\mathcal{R}},v))=\delta(v,w)>0. 
\end{align}
If $v_{\mathcal{R}}$ is between $v'$ and $v''$ then 
\begin{align}
\textup{ord}\left(\frac{(x_i-x_k)(x_i-x_l)(x_j-x_k)(x_j-x_l)}{(x_i-x_j)^2(x_k-x_l)^2}\right)=2(d_{\mathcal{R}}&+\delta(v_{\mathcal{R}},v')+d_{\mathcal{R}}+\delta(v_{\mathcal{R}},v)+d_{\mathcal{R}}+\delta(v_{\mathcal{R}},w) \\
&-2(d_{\mathcal{R}}+\delta(v_{\mathcal{R}},v')-2(d_{\mathcal{R}}+\delta(v_{\mathcal{R}},v))=\delta(v,w)>0. 
\end{align}
This concludes the proof.
\end{proof}

\begin{corollary}\label{values}
Let $C$ be a hyperelliptic curve of genus $g$ over a local field $K$ of odd residue characteristic and fix a model $C:y^2=f(x)$, where $f(x)$ has degree $2g+2$. Fix a labelling $x_1,\dots,x_d$ of the roots of $f(x)$ and denote by $s_1,\dots,s_d$ the singletons of $T_C$ corresponding to the roots. If $\{\{s_i,s_j\},\{s_k,s_l\}\}\in S_{m}(T_C,I_C)$ then 
\begin{equation}
\textup{ord}\left(\frac{(x_i-x_k)(x_i-x_l)(x_j-x_k)(x_j-x_l)}{(x_i-x_j)^2(x_k-x_l)^2}\right)=-2\delta_m(C).
\end{equation}
\end{corollary}

\begin{proof}
By the definition of $S_{m}(T_C,I_C)$ (Definition \ref{stablemodeltreedefs}), the unique path between $C_{ij}$ and $C_{kl}$ goes between vertices $v$ and $w$ in $T_C$ for which $\delta(v,w)=\delta_m(C)$. So the result follows by Proposition \ref{distancepaths}.
\end{proof}

\begin{lemma}\label{weightsinleafgraph}
Let $(T,I)\in\mathbf{T}_g$ be a stable model tree with an ordering on the distances between the vertices and let $C:y^2=f(x)$ be a Weierstrass model for a hyperelliptic curve $C$ over a local field $K$. Fix a labelling of the singletons $s_1,\dots,s_d$ of $T$ corresponding to the roots $x_1,\dots,x_d$ of $f(x)$. Then for $n\geq 0$
\begin{equation}
    \textup{ord}(R_{T,I,n}^{\sigma})=\sum_{m=1}^{n}\sum_{\substack{ \big\{\{s_{\sigma^{-1}(i)},s_{\sigma^{-1}(j)}\},\\ \{s_{\sigma^{-1}(k)},s_{\sigma^{-1}(l)}\}\big\}\in S_{m}(T,I)}}(n-m+1)\cdot \textup{ord}\left(\frac{(x_i-x_k)(x_i-x_l)(x_j-x_k)(x_j-x_l)}{(x_i-x_j)^2(x_k-x_l)^2}\right).
\end{equation}
\end{lemma}

\begin{proof}
This follows immediately from the definition of $R_{T,I,n}^{\sigma}$.
\end{proof}

\begin{fact}[Rearrangement inequality]\label{weights}
Let $x_1\geq x_2\geq \dots \geq x_k$ and $w_1\geq w_2\geq \dots \geq w_{k}$ be two descending sequences of rational numbers. Let $\sigma$ be a permutation of $1,\dots,k$ for which $(w_1,\dots,w_k)\neq (w_{\sigma(1)},\dots,w_{\sigma(k)})$. Then 
\begin{equation}
    \sum_{i=1}^{k}w_ix_i > \sum_{i=1}^{k} w_{\sigma(i)}x_i.
\end{equation}
That is to say, the sum is maximised by allocating the highest weight $w_1$ to the highest number $x_1$, the second highest to the second highest and so on.
\end{fact}

\begin{proof}[Proof of Theorem \ref{valuationofyprime}]
Fix a labelling of the singletons of $T_C$ corresponding to the roots of $f(x)$ and let $R_{T_C,I,n}$ be the summand of $\textup{Inv}_{T_C,I_C,n}$ with this labelling. By the definition of $R_{T_C,I_C,n}$, Lemma \ref{weightsinleafgraph} above and Corollary \ref{values} that gives the valuation of the factors $(ij,kl)$ in terms of the distances $\delta_l(C)$ in $T_C$,
\begin{equation}
    \textup{ord}(R_{T_C,I_C,n}(f))=-2\sum_{m=1}^{n} K_{m}(C)(n+1-m)\delta_{m}(C).
\end{equation}
Since 
\begin{equation}
\textup{ord}(R_{T,I,n}^{\sigma}(f))=\sum_{m=1}^{n}\sum_{\substack{ \big\{\{s_{\sigma^{-1}(i)},s_{\sigma^{-1}(j)}\},\\ \{s_{\sigma^{-1}(k)},s_{\sigma^{-1}(l)}\}\big\}\in S_{m}(T,I)}}(n-m+1)\cdot \textup{ord}\left((ij,kl)\right)
\end{equation}
for $\sigma\in S_d/\textup{Stab}$ by Lemma \ref{weightsinleafgraph}, we want to know when the sum of the valuation of building blocks $(ij,kl)$ multiplied by the exponents is the lowest to know $\textup{ord}(\textup{Inv}_{T_C,I_C,n}(C))$. Note that, by Proposition \ref{distancevalues}, the negation of the valuation of the building blocks $(ij,kl)$ gives a descending sequence of rational numbers
\begin{equation}
x_1\geq x_2\geq \dots \geq x_k
\end{equation}
where $x_{m}=\delta_m(C)$ if $K_{1}(C)+\dots+ K_{m-1}(C)+1\leq m\leq K_{1}(C)+\dots+ K_{m}(C)$ and $x_{m}\leq 0$ for $m>K_1(C)+\dots+K_N(C)$, with $N$ the number of distinct distances in $T_C$ and setting $K_0(C)=0$. This is because, by the definition of $K_m(C)$ and Proposition \ref{distancevalues}, $2\delta_{m}(C)$ appears $K_m(C)$ times in the set of valuations of expressions of the form $(ij,kl)$ and, beyond the valuation of the smallest distance in the tree $\delta_N(C)$, the valuation of $(ij,kl)$ is $\geq 0$. In the sum, we have the descending sequence of exponents 
\begin{equation}
n,\dots,n,n-1,\dots,n-1,\dots,1,\dots,1.
\end{equation}
where $n+1-m$ appears $K_{m}(C)$ times as an exponent by the definition of $R_{T,I,n}$. By Fact \ref{weights}, $\textup{ord}(R_{T_C,I_C,n}(f))$ is the unique summand of $\textup{Inv}_{Y_{n+1}(C)}(C)$ with the lowest (most negative) valuation, since it is the unique summand that allocates $(ij,kl)$ weight $n+1-m$ if and only if $\textup{ord}((ij,kl))=-2\delta_m(C)$; all other summands allocate a higher exponent to a smaller distances between vertices, or to a positive distance $\delta(v,w)>0$ since, by the definition of $R_{T_C,I_C,n}$, $\sigma$ fixes $R_{T_C,I_C,n}$ if and only if it fixes $\{(ij,kl):\textup{ord}((ij,kl))=-2\delta_m(C)\}$ for $m=1,\dots,n$. Hence, 
\begin{equation}
    \textup{ord}(\textup{Inv}_{T_C,I_C,n}(C))=\textup{ord}(R_{T_C,I_C,n}(f))=-2\sum_{m=1}^{n} K_{m}(C)(n+1-m)\delta_{m}(C).
\end{equation}
\end{proof}

\section{Comparing the valuations of the absolute invariants of a hyperelliptic curve}\label{comparingvaluations}

The main theorem we prove in this section is Theorem \ref{valuationofyprime}, which compares the valuation of the absolute invariants defined in \S\ref{invariantssection} when they are evaluated on hyperelliptic curves. We compare the valuations by defining a quantity $B_{n}(T,I,C)$ associated to a possible stable model tree and a possible ordering on the distances between the vertices $(T,I)$ and a fixed hyperelliptic curve $C$. We also prove that $\delta_n(C)=B_n(T_C,I_C,C)$ (see Proposition \ref{correctone}), and we show how one can tell whether a curve has potentially good reduction from the absolute invariants defined in Definition \ref{invariantsdefinition}. 

\begin{definition}[Set of all possible stable model trees at the $(n+1)$-st step]\label{possiblen+1st}
Let $T_C$ be the stable model tree of a hyperelliptic curve $C$ of genus $g$ over a local field $K$ of odd residue characteristic. Define 
\begin{align}
    \mathcal{T}_{n+1}(C)=\{(T,I) \in\mathbf{T}_g: \ R_{T,I,n}=R_{T_C,I_C,n}^{\sigma} \textup{ for some $\sigma\in S_{2g+2}$}\}. 
\end{align}
In other words, $\mathcal{T}_{n+1}(C)$ is the set of all stable model trees with orderings on the distances between vertices for which the $n$-th summand is equal to the $n$-th summand associated to $T_C$, up to a permutation by the symmetric group. We call $\mathcal{T}_{n+1}(C)$ the \textit{set of all possible stable model trees at the $(n+1)$-st step}, since it contains all possibilities for $(T_C,I_C)$ given that $R_{T_C,I_C,n}$ is known.
\end{definition}

We can now define $B_{n+1}(T,I,C)$ for $(T,I)\in \mathcal{T}_{n+1}$, which is a normalised valuation that is shifted in relation to valuation of $\textup{Inv}_{T_C,I_C,n}(C)$ and scaled by $K_{n+1}(T,I)$. It will allow us to compare the valuations of $\textup{Inv}_{T,I,n+1}(C)$ for $(T,I)\in \mathcal{T}_{n+1}$ in order to find $R_{T_C,I_C,n+1}(C)$ and ultimately recover $T_C$.

\begin{definition}\label{bn+1}
For $(T,I)\in \mathcal{T}_{n+1}$, define 
\begin{equation}
B_1(T,I,C)=-\frac{\textup{ord}(\textup{Inv}_{T,I,1}(C))}{2\cdot K_{1}(T,I)}
\end{equation}
and for $n\geq 2$, define
\begin{equation}
    B_{n}(T,I,C)=\frac{-\textup{ord}(\textup{Inv}_{T,I,n}(C)) -2\sum_{m=1}^{n-1} (n+1-m)K_{m}(C)\delta_m(C)}{2\cdot K_{n}(T,I)},
\end{equation}
where $K_{n+1}(T,I)=\# S_{n+1}(T,I)$.
\end{definition}

The following theorem tells us that, out of elements of $\mathcal{T}_{n+1}(C)$ that maximise $B_{n+1}(T,I,C)$, $(T_C,I_C)$ is the element with the greatest value of $K_{n+1}(-,-)$. This is used to prove Theorem \ref{wholealgorithm} in \S\ref{mainthminvproof}, which tells us how the valuation of the absolute invariants associated to elements of $\mathcal{T}_{n+1}(C)$ can be used to recover the $(n+1)$-st summand associated to $T_C$ from $R_{T_C,I_C,n}$.

\begin{theorem}\label{semistablemainlemma}
Let $C$ be a hyperelliptic curve of genus $g$ over a local field $K$. Let $(T,I)\in \mathcal{T}_{n+1}(C)$, with $R_{T,I,n+1}\neq R_{T_C,I_C,n+1}^{\sigma} \textup{ for any $\sigma\in S_{2g+2}$}$.
\begin{enumerate}[(i)]
    \item If $K_{n+1}(T,I)< K_{n+1}(C)$ then $B_{n+1}(T_C,I_C,C)\geq B_{n+1}(T,I,C)$.
    \item If $K_{n+1}(T,I)\geq K_{n+1}(C)$ then $B_{n+1}(T_C,I_C,C)>B_{n+1}(T,I,C)$.
\end{enumerate}
\end{theorem}

Before proving Theorem \ref{semistablemainlemma}, we demonstrate the use of $B_{n+1}(T,I,n+1)$ in determining $R_{T_C,I_C,n+1}$ from $R_{T_C,I_C,n}$ by means of an example.

\begin{example}\label{curlytn+1}\label{bn+1example}
Let us take a genus $2$ hyperelliptic curve $C$ over a local field $K$ of odd residue characteristic with stable model tree  
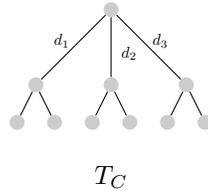
\begin{center}
\begin{figure}[H]
\begin{tikzpicture}
				[scale=0.5, auto=left,every node/.style={circle,fill=black!20,scale=0.6}]
                \node (n10) at (2,2) {};
				\node (n1) at (0,0) {};
				\node (n2) at (2,0)  {};
				\node (n3) at (4,0)  {};

                 \node (n4) at (-0.5,-1) {};
                \node (n5) at (0.5,-1) {};
                
                \node (n6) at (1.5,-1) {};
                \node (n7) at (2.5,-1) {};

                \node (n8) at (3.5,-1) {};
                \node (n9) at (4.5,-1) {};

                \draw (n1) -- (n4); 
                \draw (n1) -- (n5); 
                
                \draw (n2) -- (n6); 
                \draw (n2) -- (n7); 

                \draw (n3) -- (n8); 
                \draw (n3) -- (n9);
				
				\draw (n10) -- (n1) node [text=black, pos=0.4, left, fill=none] {$d_1$};
                \draw (n10) -- (n2) node [text=black, pos=0.6, right,fill=none] {$d_2$};
				\draw (n10) -- (n3) node [text=black, pos=0.4, right,fill=none] {$d_3$};
			\end{tikzpicture}
            \caption*{$T_C$}
            \end{figure}
\end{center}
where $d_1> d_2>d_3$. Suppose we know a priori that 
\begin{equation}
\textup{Inv}_{T_C,I_C,1}=\sum_{S_6/\textup{Stab}}\frac{(X_1-X_3)(X_1-X_4)(X_2-X_3)(X_2-X_4)}{(X_1-X_2)^2(X_3-X_4)^2},
\end{equation}
but we want to know work out $R_{T_C,I_C,2}$ using the possible second absolute invariants for $(T_C,I_C)$ contained in $\mathcal{T}_2(C)$. Then $\mathcal{T}_2(C)$ contains the following stable model trees with the described orderings on the distances between the vertices.

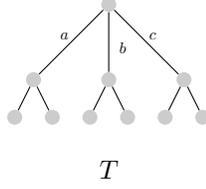
\begin{minipage}{0.4\textwidth}
\begin{figure}[H]
			\begin{tikzpicture}
				[scale=0.5, auto=left,every node/.style={circle,fill=black!20,scale=0.6}]
                \node (n10) at (2,2) {};
				\node  (n1) at (0,0) {};
				\node (n2) at (2,0)  {};
				\node (n3) at (4,0)  {};

                 \node (n4) at (-0.5,-1) {};
                \node (n5) at (0.5,-1) {};
                
                \node (n6) at (1.5,-1) {};
                \node (n7) at (2.5,-1) {};

                \node (n8) at (3.5,-1) {};
                \node (n9) at (4.5,-1) {};

                \draw (n1) -- (n4); 
                \draw (n1) -- (n5); 
                
                \draw (n2) -- (n6); 
                \draw (n2) -- (n7); 

                \draw (n3) -- (n8); 
                \draw (n3) -- (n9);

				\draw (n10) -- (n1) node [text=black, pos=0.4, left, fill=none] {$a$};
                \draw (n10) -- (n2) node [text=black, pos=0.6, right,fill=none] {$b$};
				\draw (n10) -- (n3) node [text=black, pos=0.4, right,fill=none] {$c$};
			\end{tikzpicture}
            \caption*{$T$}
   \end{figure}
   \end{minipage}\begin{minipage}{0.6\textwidth}

$I_1=\{a>b>c, \ a>b+c\}\in \textup{Ineq}_{T};$

$I_2=\{a>b>c, \ a<b+c\}\in \textup{Ineq}_{T};$

$I_3=\{a>b>c, \ a=b+c\}\in \textup{Ineq}_{T};$

$I_4=\{a=b>c\}\in \textup{Ineq}_{T}.$

   \end{minipage}

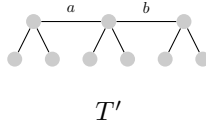
\begin{minipage}{0.4\textwidth}
\begin{figure}[H]
			\begin{tikzpicture}
				[scale=0.5, auto=left,every node/.style={circle,fill=black!20,scale=0.6}]
				\node  (n1) at (0,0) {};
				\node (n2) at (2,0)  {};
				\node (n3) at (4,0)  {};

                \node (n4) at (-0.5,-1)  {};
                \node (n5) at (0.5,-1)  {};
                \node (n6) at (1.5,-1)  {};
                \node (n7) at (2.5,-1)  {};
                \node (n8) at (3.5,-1)  {};
                 \node (n9) at (4.5,-1)  {};

                 \draw (n1) -- (n4) node [text=black, pos=0.4, left, fill=none] {};
                 \draw (n1) -- (n5) node [text=black, pos=0.4, left, fill=none] {};

                 \draw (n2) -- (n6) node [text=black, pos=0.4, left, fill=none] {};
                 \draw (n2) -- (n7) node [text=black, pos=0.4, left, fill=none] {};

                 \draw (n3) -- (n8) node [text=black, pos=0.4, left, fill=none] {};
                 \draw (n3) -- (n9) node [text=black, pos=0.4, left, fill=none] {};

				\draw (n1) -- (n2) node [text=black, pos=0.5, above, fill=none] {$a$};
                \draw (n2) -- (n3) node [text=black, pos=0.5, above,fill=none] {$b$};
			\end{tikzpicture}
            \caption*{$T'$}
   \end{figure}
   \end{minipage}\begin{minipage}{0.6\textwidth}

$I_1=\{a>b\}\in \textup{Ineq}_{T'};$

$I_2=\{a=b\}\in \textup{Ineq}_{T'}.$

   \end{minipage}

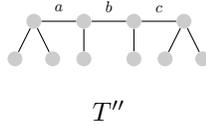
\begin{minipage}{0.4\textwidth}
\begin{figure}[H]
			\begin{tikzpicture}
				[scale=0.5, auto=left,every node/.style={circle,fill=black!20,scale=0.6}]
				\node  (n1) at (0,0) {};
				\node (n2) at (1.33,0)  {};
				\node (n3) at (2.66,0)  {};
                \node (n4) at (3.99,0)  {};

                \node (n5) at (-0.5,-1)  {};
                \node (n6) at (0.5,-1)  {};
                
                \node (n7) at (1.33,-1)  {};
                
                \node (n8) at (2.66,-1)  {};
                
                \node (n9) at (3.49,-1)  {};
                 \node (n10) at (4.49,-1)  {};

                 \draw (n1) -- (n5) node [text=black, pos=0.4, left, fill=none] {};
                 \draw (n1) -- (n6) node [text=black, pos=0.4, left, fill=none] {};

                 \draw (n2) -- (n7) node [text=black, pos=0.4, left, fill=none] {};
                 
                 \draw (n3) -- (n8) node [text=black, pos=0.4, left, fill=none] {};

                 \draw (n4) -- (n9) node [text=black, pos=0.4, left, fill=none] {};
                 \draw (n4) -- (n10) node [text=black, pos=0.4, left, fill=none] {};

				\draw (n1) -- (n2) node [text=black, pos=0.5, above, fill=none] {$a$};
                \draw (n2) -- (n3) node [text=black, pos=0.5, above,fill=none] {$b$};
				\draw (n3) -- (n4) node [text=black, pos=0.5, above,fill=none] {$c$};
			\end{tikzpicture}
            \caption*{$T''$}
   \end{figure}
   \end{minipage}\begin{minipage}{0.6\textwidth}

$I_1=\{a>c>b\}\in \textup{Ineq}_{T''}; \quad i_5=\{b=a>c\}\in \textup{Ineq}_{T''};$

$I_2=\{b>a>c\}\in \textup{Ineq}_{T''}; \quad i_6=\{a=c>b\}\in \textup{Ineq}_{T''};$

$I_3=\{a>b>c\}\in \textup{Ineq}_{T''}; \quad i_7=\{b>a=c\}\in \textup{Ineq}_{T''};$

$I_4=\{a>c=b\}\in \textup{Ineq}_{T''}; \quad i_8=\{a=b=c\}\in \textup{Ineq}_{T''}.$

\vspace{20pt}

   \end{minipage}

This is because these are the exact elements $(T,I)\in\mathbf{T}_2$ for which $R_{T,I,1}=R_{T_C,I_C,1}^{\sigma} \textup{ for some $\sigma\in S_{6}$}$, since the longest distance in each case is between two vertices that both have two singletons. In an arbitrary labelling of the singletons and using the notation
\begin{equation}
(ij,kl)=\frac{(X_i-X_k)(X_i-X_l)(X_j-X_k)(X_j-X_l)}{(X_i-X_j)^2(X_k-X_l)^2},
\end{equation}
the summands associated to the elements of $\mathcal{T}_2$ are
\begin{align}
R_{T,I_1,2}&=  (12,34)^2(12,56)=R_{T,I_2,2}=R_{T,I_3,2} \\
R_{T,I_4,2}&=  (12,34)^2(12,56)(34,56)\\
R_{T',I_1,2}&= (12,34)^2(12,53)(12,54)(12,63)(12,64)(12,56) \\
R_{T',I_2,2}&= (12,34)^2(12,53)(12,54)(12,63)(12,64)(12,56)(34,51)(34,52)(34,61)(34,62)(34,56) \\
R_{T'',I_1,2}&=  (12,34)^2(12,53)(12,54)= R_{T'',I_2,2}=R_{T'',I_3,2}=R_{T'',I_4,2}=R_{T'',I_5,2} \\
R_{T'',I_6,2}&=  (12,34)^2(12,53)(12,54)(34,61)(34,62)=R_{T'',I_7,2}=R_{T'',I_8,2}
\end{align}
By applying Proposition \ref{distancevalues}, we can write down inequalities on the valuation of the associated absolute invariants when evaluated on a hyperelliptic curve $C$ with the above stable model tree in terms of $d_1$, $d_2$ and $d_3$. We can then use this to write down an inequality for the associated `average' $B_{2}(T,I,C)$ in terms of $d_1$, $d_2$ and $d_3$, which are written in the table below.

\begin{center}
\setlength\extrarowheight{2pt}
\begin{tabular}{|>{\centering\arraybackslash}m{1.8cm}||>{\centering\arraybackslash}m{1.64cm}|>{\centering\arraybackslash}m{1.64cm}|>{\centering\arraybackslash}m{1.64cm}|>{\centering\arraybackslash}m{2.3cm}|>{\centering\arraybackslash}m{1.64cm}|>{\centering\arraybackslash}m{2.3cm}|}
    \hline
      $(T,I)$ &  $(T,I_1)$ & $(T,I_2)$ & $(T,I_3)$ & $(T,I_4)$ & $(T',I_1)$ & $(T',I_2)$  \\
        \hline
       $B_{2}(T,I,C)$  & $d_1+d_3$ & $d_1+d_3$  & $d_1+d_3$ &  $\frac{1}{2}(d_1+d_2)+d_3$ & $d_1+\frac{d_3}{5}$ & $\frac{d_1+d_2}{2}+\frac{d_3}{5}$ \\
       \hline 
\end{tabular}
\end{center}

\hspace{5pt}

\begin{center}
\setlength\extrarowheight{2pt}
\begin{tabular}{|>{\centering\arraybackslash}m{1.8cm}||>{\centering\arraybackslash}m{1.3cm}|>{\centering\arraybackslash}m{1.3cm}|>{\centering\arraybackslash}m{1.3cm}|>{\centering\arraybackslash}m{1.3cm}|>{\centering\arraybackslash}m{1.3cm}|>{\centering\arraybackslash}m{1.3cm}|>{\centering\arraybackslash}m{1.3cm}|>{\centering\arraybackslash}m{1.3cm}|}
    \hline
      $(T,I)$ & $(T'',I_1)$ &  $(T'',I_2)$ & $(T'',I_3)$ & $(T'',I_4)$ & $(T'',I_5)$ & $(T'',I_6)$ & $(T'',I_7)$ & $(T'',I_8)$  \\
        \hline
       $B_{2}(T,I,C)$ & $\leq d_1$  & $\leq d_1$ & $\leq d_1$ & $\leq d_1$ & $\leq d_1$ & $\leq \frac{d_1+d_2}{2}$  & $\leq \frac{d_1+d_2}{2}$ & $\leq \frac{d_1+d_2}{2}$ \\
       \hline 
\end{tabular}
\end{center}
Out of the elements of $\mathcal{T}_{2}(C)$, $(T,I_1)$, $(T,I_2)$ and $(T,I_3)$ have the highest value of $B_{2}(T,I,C)$. It turns out that this is telling us that $R_{T_C,I_C,2}=R_{T,I_1,2}=R_{T,I_2,2}=R_{T,I_3,2}$ up to a permutation of the singletons, and $\delta_{2}(C)=B_{2}(T,I_1,C)=d_1+d_3$, as we would expect since we already know the stable model tree. In other words, it tells us that $(T_C,I_C)=(T,I_1)$, $(T,I_2)$ or $(T,I_3)$. This is the idea behind Theorem \ref{wholealgorithm}, which gives the general statement on finding $R_{T_C,I_C,n+1}$ given that $R_{T_C,I_C,n}$ is known by looking at the value of $B_{n+1}(T,I,C)$ for each $(T,I)\in\mathcal{T}_{n+1}(C)$.

A further step using the $3$-rd absolute invariants would be needed in order to distinguish whether $(T_C,I_C)$ is $(T,I_1)$, $(T,I_2)$ or $(T,I_3)$, and it depends on whether $d_1>d_2+d_3$, $d_1<d_2+d_3$ or $d_1=d_2+d_3$. A full description for genus $2$ curves is given in \S\ref{genus2section}.
\end{example}

\begin{proof}[Proof of Theorem \ref{semistablemainlemma}]
First let us prove $(i)$. We know from Theorem \ref{valuationofyprime} that
\begin{equation}
    \textup{ord}(\textup{Inv}_{T_C,I_C,n+1}(C))=-2\sum_{m=1}^{n+1} K_m(C)(n+2-m)\delta_{m}(C).
\end{equation}
Let $d=2g+2$ and fix a labelling of the singletons of $T_C$ corresponding to the roots $x_1,\dots,x_d$ of $f(x)$, where $C:y^2=f(x)$ is a Weierstrass equation for $C$. Let $R_{T,I,n+1}^{\sigma}(f)$ be a summand of $\textup{Inv}_{T,I,n+1}(C)$ with the lowest valuation (there may be multiple summands with the same valuation). We claim that
\begin{equation}
    \textup{ord}(\textup{Inv}_{T,I,n+1}(C))\geq \textup{ord}(R_{T,I,n+1}(f)^{\sigma}) \geq -2\sum_{m=1}^{n} (n+2-m)\cdot K_m(C)\delta_{m}(C) - 2K_{n+1}(T,I)\delta_{n+1}(C).
\end{equation}
As in the proof of Theorem \ref{valuationofyprime}, the negation of the valuation of the building blocks $(ij,kl)$ give a descending sequence of rational numbers
\begin{equation}
x_1\geq x_2\geq \dots \geq x_k
\end{equation}
where $x_{m}=\delta_m(C)$ if $K_{1}(C)+\dots+ K_{m-1}(C)+1\leq m\leq K_{1}(C)+\dots+ K_{m}(C)$ and $x_{m}\leq 0$ for $m>K_1(C)+\dots+K_N(C)$, with $N$ the number of distinct distances in $T_C$ and setting $K_0(C)=0$. In the expression for $R_{T,I,n+1}$, we have the descending sequence of exponents of the factors $(ij,kl)$
\begin{equation}
n+1,\dots,n+1,n,\dots,n,n-1,\dots,n-1,\dots,1,\dots,1.
\end{equation}
where the exponent $n+2-m$ appears $K_{m}(C)$ times for $m=1,\dots,n$ since $(T,I)\in \mathcal{T}_{n+1}(C)$ and so $R_{T,I,n}=R_{T_C,I_C,n}^{\sigma}$ for some $\sigma\in S_{2g+2}$. The exponent $1$ appears $K_{n+1}(T,I)$ times by the definition of $R_{T,I,n+1}$. Hence, as in the proof of Theorem \ref{valuationofyprime} and by Fact \ref{weights}, the valuation of the expression for $R_{T,I,n+1}(f)^{\sigma}$ is greater than or equal to $-(\Sigma_{m=1}^{n} (n+2-m)K_m(C)x_m+K_{n+1}(T,I)\delta_{n+1}(C))$. So for $n\geq 0$ we have
\begin{align}
   - \textup{ord}(\textup{Inv}_{T,I,n+1}(C))-2\sum_{m=1}^n (n+2-m)K_{m}(C)\delta_m(C)\leq 2 K_{n+1}(T,I)\delta_{n+1}(C).
\end{align}
Thus, as required
\begin{equation}
     B_{n+1}(T,I,C)=\frac{- \textup{ord}(\textup{Inv}_{T,I,n+1}(C))-2\sum_{m=1}^n (n+2-m)K_{m}(C)\delta_m(C)}{2\cdot K_{n+1}(T,I)} \leq \delta_{n+1}(C)=B_{n+1}(T_C,I_C,C). 
\end{equation}

$(ii)$ Now suppose $K_{n+1}(T,I)\geq K_{n+1}(C)$ and let $R_{T,I,n+1}(f)^{\sigma}$ be a summand of $\textup{Inv}_{T,I,n+1}(C)$ with the lowest valuation. To prove the second part of the lemma, it suffices to prove that we have the strict inequality 
\begin{equation}
    \textup{ord}(R_{T,I,n+1}(f)^{\sigma}) > -2\sum_{m=1}^{n} (n+2-m)\cdot K_m(C)\delta_{m}(C) - 2K_{n+1}(T,I)\delta_{n+1}(C),
\end{equation}
since if this is the case then $B_{n+1}(T_C,I_C,C)>B_{n+1}(T,I,C)$. Note that since $R_{T,I,n}=R_{T_C,I_C,n}^{\sigma}$ for some $\sigma\in S_{2g+2}$, we have $K_m(T,I)=K_m(C)$ for $m=1,\dots,n$. If $K_{n+1}(T,I)\leq K_{n+1}(C)+\dots+K_{N}(C)$, where $N$ is the number of distinct in $T_C$, then again by Corollary \ref{values} and Fact \ref{weights}, for $n\geq 0$, 
\begin{equation}
    \textup{ord}(R_{T,I,n+1}(f)^{\sigma})\geq -2\sum_{m=1}^{n} (n+2-m)K_m(C)\delta_{m}(C) - 2\sum_{m=1}^{r}K_{n+m}(C)\delta_{m}(C)-2k_{n+r+1}\delta_{n+r+1}(C),
\end{equation}
where $K_{n+1}(T,I)=\sum_{m=1}^{r+m}K_{n+m}(C)+k_{n+r+1}$ and $k_{r+1}\leq K_{r+1}(C)$, by the same argument as $(i)$. If $K_{n+1}(T,I)>K_{n+1}(C)$, it is clear that
\begin{align}
    \textup{ord}(R_{T,I,n+1}(f)^{\sigma})&\geq -2\sum_{m=1}^{n} (n+2-m)K_m(C)\delta_{m}(C) - 2\sum_{m=1}^{r}K_{n+m}(C)\delta_{m}(C)-2k_{n+r+1}\delta_{n+r+1}(C) \\
    &>-2\sum_{m=1}^{n} (n+2-m)\cdot K_m(C)\delta_{m}(C) - 2K_{n+1}(T,I)\delta_{n+1}(C). 
\end{align}
Now suppose $K_{n+1}(T,I)=K_{n+1}(C)$ and assume $\textup{ord}(R_{T,I,n+1}(f)^{\sigma})= -2\sum_{m=1}^{n+1} (n+2-m)K_m(C)\delta_{m}(C)$. Then $R_{T,I,n+1}(f)^{\sigma}$ must allocate exponent $1$ to the $K_{n+1}(C)$ expressions $(ij,kl)$ with $\textup{ord}(ij,kl)=-2\delta_{n+1}(C)$. The only way that this can happen is if $R_{T,I,n+1}=R_{T_C,I_C,n+1}^{\sigma}$ for some $\sigma\in S_{2g+2}$, which is a contradiction since we assumed $R_{T,I,n+1}\neq R_{T_C,I_C,n+1}^{\sigma}$ for any $\sigma\in S_{2g+2}$. Thus 
\begin{align}
    \textup{ord}(R_{T,I,n+1}(f)^{\sigma})>-2\sum_{m=1}^{n} (n+1-m)\cdot K_m(C)\delta_{m}(C) - 2K_{n+1}(T,I)\delta_{n+1}(C)
\end{align}
when $K_{n+1}(T,I)=K_{n+1}(C)$. 

If $K_{n+1}(T,I)>K_{n+1}(C)+\dots+K_{N}(C)$, then since by Corollary \ref{values} there are only $K_{n+1}(C)+\dots+K_{N}(C)$ expressions $(ij,kl)$ of valuation $<0$, by Fact \ref{weights}, for $n\geq 0$, 
\begin{align}
    \textup{ord}(R_{T,I,n+1}(f)^{\sigma})&\geq -2\sum_{m=1}^{N} (n+2-m)K_m(C)\delta_{m}(C) \\
    &>-2\sum_{m=1}^{n} (n+2-m)K_m(C)\delta_{m}(C)-2K_{n+1}(T,I)\delta_{n+1}(C). 
\end{align}
This completes the proof. 
\end{proof}

\begin{proposition}\label{correctone}
Let $T_C$ be the stable model tree of a hyperelliptic curve $C$ over a local field $K$ of odd residue characteristic. Then for $n\geq 0$, $\delta_{n+1}(C)=B_{n+1}(T_C,I_C,C)$.
\end{proposition}

\begin{proof}
This follows immediately from Theorem \ref{valuationofyprime} and the fact that, by definition, $K_n(C)=K_n(T_C,I_C)$.
\end{proof}

In Theorem \ref{wholealgorithm}, we will deal with the case of potentially good reduction separately to the the other possible reduction types. The following proposition gives us a way of identifying when $C/K$ has potentially good reduction from the first absolute invariants associated to elements $(T,I)\in\mathbf{T}_g$. Note that it was already known that a hyperelliptic curve having potentially good reduction can be identified by checking the valuation of invariants of the curve (see, for example, \cite{elisapotentiallygood} Proposition 3.13), but we state the result here for the absolute invariants defined in this paper. 

\begin{proposition}\label{potgoodred}
Let $C$ be a genus $g$ hyperelliptic curve over a local field $K$ of odd residue characteristic. Then $C$ has potentially good reduction if and only if $\textup{ord}(\textup{Inv}_{T,I,1}(C))\geq 0$ for every $(T,I)\in\mathbf{T}_g$.
\end{proposition}

\begin{proof}
Let $C:y^2=f(x)$ be a Weierstrass equation for $C$ where $d=\deg(f)=2g+2$. Denote by  $x_1,\dots,x_d$ the roots of $f(x)$. We know from Proposition \ref{potgoodredlemma} that $C$ has potentially good reduction if and only if its cluster picture either has no subclusters of size less than $d$ or it has one subcluster of size $d-1$. Suppose we are in the first case and the outside cluster has depth $d_{\mathcal{R}}$. By the definition of the cluster picture, this means that we have $\textup{ord}(x_i-x_j)=d_{\mathcal{R}}$ for every $i\neq j$. Thus
\begin{equation}
    \textup{ord}\Big(\frac{(x_i-x_k)(x_i-x_l)(x_j-x_k)(x_j-x_l)}{(x_i-x_j)^2(x_k-x_l)^2}\Big)=0
\end{equation}
and so $\textup{ord}(\textup{Inv}_{T,I,1}(C))\geq 0$. Suppose there is a subcluster of size $d-1$ of depth $\delta$ and the outside depth is $d_{\mathcal{R}}$. If all the roots are in the cluster of size $d-1$, similarly to the above we have valuation $0$. Otherwise, say $x_l$ is the root not in the cluster of size $d-1$, then 
\begin{equation}
    \textup{ord}\Big(\frac{(x_i-x_k)(x_i-x_l)(x_j-x_k)(x_j-x_l)}{(x_i-x_j)^2(x_k-x_l)^2}\Big)=d_{\mathcal{R}} +\delta +d_{\mathcal{R}} +\delta - 2d_{\mathcal{R}}- 2\delta=0
\end{equation}
and so $\textup{ord}(\textup{Inv}_{T,I,1}(C))\geq 0$. For the converse, note that Theorem \ref{valuationofyprime} implies that if $C$ does not have potentially good reduction, if one takes $(T,I)=(T_C,I_C)$ then $\textup{ord}(\textup{Inv}_{T,I,1}(C))<0$.
\end{proof}

\section{Determining the dual graph from the absolute invariants}\label{mainthminvproof}
In this section we use the results of \S\ref{comparingvaluations} to prove the main result of this paper on recovering the dual graph of the special fibre of the minimal regular model of a semistable hyperelliptic curve from absolute invariants. The main theorem we prove is the following.

\begin{theorem}\label{wholealgorithm}
Let $C$ be a hyperelliptic curve of genus $g$ over a local field $K$ of odd residue characteristic, and let $d=2g+2$. The stable model tree $T_C$ is determined by the following procedure. 
\begin{enumerate}[(i)]
\item $T_C=K_{1,d}$, or equivalently $C/K$ has potentially good reduction, if and only if $\textup{ord}(\textup{Inv}_{T,I,1}(C))\geq 0$ for every $(T,I)\in\mathbf{T}_g$. 
\item If $T_C$ does not consist of a single vertex, for $n\geq 0$ given $R_{T_C,I_C,n}$ and $\delta_1(C),\dots,\delta_{n}(C)$, let 
   \begin{align}
    \mathcal{T}_{n+1}(C)=\{(T,I) \in\mathbf{T}_g: \ R_{T,I,n}=R_{T_C,I_C,n}^{\sigma} \textup{ for some $\sigma\in S_{2g+2}$} \}, 
\end{align}
and for $(T,I)\in\mathcal{T}_{n+1}(C)$ let 
\begin{equation}
    B_{n+1}(T,I,C)=\frac{-\textup{ord}(\textup{Inv}_{T,I,n+1}(C)) -2\sum_{m=1}^n (n+2-m)K_{m}(C)\delta_m(C)}{2\cdot K_{n+1}(T,I)}.
\end{equation}
Out of the elements of $\mathcal{T}_{n+1}(C)$, let $(T,I)$ be a tree and ordering with the largest value of $K_{n+1}(-,-)$ satisfying $B_{n+1}(T,I,C)=\underset{(T',I')\in\mathcal{T}_{n+1}(C)}{\emph{\text{max}}} B_{n+1}(T',I',C)$. Then 
\begin{equation}
    R_{T_C,I_C,n+1}=R_{T,I,n+1}^{\sigma} \textup{ for some $\sigma\in S_{2g+2}$} \quad \textup{and} \quad \delta_{n+1}(C)=B_{n+1}(T,I,C).
\end{equation}
\item Let $n_{C}$ denote the number of pairwise differences between vertices in $T_C$. Then $(T_C,I_C)$ is uniquely determined by $R_{T_C,I_C,n_{C}}$, and $\delta_i(T_C,I_C)=\delta(v,w)$ in $(T_C,I_C)$ if and only if $\delta(v,w)=\delta_i(C)$ in $T_C$. Thus $T_C$ is uniquely determined by $R_{T_C,I_C,n_{C}}$ and $\delta_1(C),\dots,\delta_{n_{C}}(C)$. 
\end{enumerate}
\end{theorem}

The BY tree (see \cite{m2d2} Definition D.6) can then be read off from $T_C$ using the procedure described in Proposition \ref{byfromtc}. The dual graph of the special fibre of the minimal regular model of $C/K^{\textup{unr}}$ can be determined from the BY tree using \cite{m2d2} Theorem 5.18 when $C/K$ is semistable.

We delay the proof of Theorem \ref{wholealgorithm}, and first state the following corollary.

\begin{corollary}\label{invthm}
Let $C$ be a hyperelliptic curve of genus $g$ over a local field of odd residue chatacteristic $K$. The finite set 
\begin{equation}
\left\{(T,I,\textup{ord}(\textup{Inv}_{T,I,n}(C))):(T,I)\in \mathbf{T}_{g} \textup{ and } n\leq n_{T,I}\right\},
\end{equation}
where $n_{T,I}$ is the number of distinct distances between the vertices in $T$ with ordering $I$, uniquely determines:
\begin{enumerate}[(i)]
      \item The dual graph of the special fibre of the minimal regular model of $C/K^{\textup{unr}}$ if $C$ is semistable;
    \item The dual graph of the special fibre of the potential stable model of $C$ if $C/K$ is not semistable. 
\end{enumerate}
\end{corollary}

\begin{proof}
By Theorem \ref{wholealgorithm}, $T_C$ is uniquely determined by the valuations of the absolute invariants associated to possible stable model trees and possible orderings on the distances between their vertices in $\mathbf{T}_{g}$. By Lemma \ref{finitenumberofleafgraphs}, $\mathbf{T}_{g}$ is a finite set and so is the set of absolute invariants. By Proposition \ref{nonsemistablemain}, $T_C$ uniquely determines $(i)$ and $(ii)$. 
\end{proof}

\begin{remark}\label{degrees}
We can find a crude bound for the degree\footnote{By degree, we mean the degree of an absolute invariant as a rational function in the roots of an even degree Weierstrass equation for $C$, i.e. $\deg(f_1/f_2)=\max(\deg(f_1),\deg(f_2))$.} of the absolute invariants in 
\begin{equation}
\left\{\textup{Inv}_{T,I,n}(C):(T,I)\in \mathbf{T}_{g} \textup{ and } n\leq n_{T,I}\right\}.
\end{equation}
The number of times a factor $(X_i-X_j)^2$ can appear on the denominator of $\textup{Inv}_{T,I,n}$ is $\left(\frac{d-2}{2}\right)$. Applying Euler's formula tells us that the number of vertices of $T$ is at most $d-2$, and so the exponent for each factor $(X_i-X_j)^2$ is at most $\left(\frac{d-2}{2}\right)$. Hence
\begin{equation}
   \deg(\textup{Inv}_{T,I,n})\leq2\cdot\left(\frac{d}{2}\right)\cdot\left(\frac{d-2}{2}\right)^2= \frac{d(d-1)(d-2)^2(d-3)^2}{4}.
\end{equation}
\end{remark}

Remark \ref{degrees} tells us that for a genus $2$ curve, the valuation of all absolute invariants up to degree $6\cdot 5 \cdot 4^2\cdot 3^2/4=1080$ determine the dual graph. However, in \S\ref{genus2section} we list a set of absolute invariants that uniquely determine the dual graph of a genus $2$ curve using the methods of this paper, and the maximum degree is $300$.

\begin{remark}
It is not always necessary to recover $R_{T_C,I_C,n_{T_C,I_C}}$ in order to determine the stable model tree; we see in \S\ref{genus2section} that $T_C$ is uniquely determined by an earlier summand for genus $2$ curves. At what point the distances between vertices in a tree uniquely determines the tree appears to be an open problem in graph theory (see, for example, \cite{trees} for partial results on the subject).
\end{remark}

\begin{example}\label{ellipticcurvesex}
Let $E:y^2=x^3+ax+b$ be an elliptic curve over a local field $K$ of odd residue characteristic, and write $x_1, x_2$ and $x_3$ for the roots of the cubic. Adding in an extra root by making a change of model, there are two possible stable model trees for $E$, corresponding to $E$ having potentially multiplicative or potentially good reduction, where $\delta\in \mathbb{Q}$:
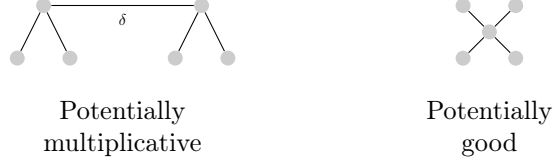
\begin{center}
\begin{minipage}[b]{0.3\textwidth}
\begin{center}
\begin{figure}[H]
\begin{tikzpicture}
				[scale=0.7, auto=left,every node/.style={circle,fill=black!20,scale=0.6}]
				
				\node (n1) at (0,0)  {};
                \node (n2) at (3,0)  {};

                 \node (n4) at (-0.5,-1)  {};
                \node (n5) at (0.5,-1)  {};
                \node (n6) at (2.5,-1)  {};
                \node (n7) at (3.5,-1)  {};

                \draw (n1) -- (n4) node [text=black, pos=0.4, left, fill=none] {};
                \draw (n1) -- (n5) node [text=black, pos=0.4, left, fill=none] {};
              
                \draw (n2) -- (n6) node [text=black, pos=0.4, left, fill=none] {};
                \draw (n2) -- (n7) node [text=black, pos=0.4, left, fill=none] {};

				\draw  (n1) -- (n2) node [text=black, pos=0.5, below, fill=none] {$\delta$};
                
			\end{tikzpicture}
        \caption*{Potentially multiplicative}
		\end{figure}
  \end{center}
		\end{minipage}\begin{minipage}[b]{0.3\textwidth}
\begin{center}
\begin{figure}[H]
			\begin{tikzpicture}
				[scale=0.7, auto=left,every node/.style={circle,fill=black!20,scale=0.6}]
				
				\node (n1) at (0,0)  {};

                \node (n4) at (-0.5,-0.5)  {};
                \node (n5) at (-0.5,0.5)  {};
                \node (n6) at (0.5,0.5)  {};
                \node (n7) at (0.5,-0.5)  {};

                \draw (n1) -- (n4) node [text=black, pos=0.4, left, fill=none] {};
                \draw (n1) -- (n5) node [text=black, pos=0.4, left, fill=none] {};
              
                \draw (n1) -- (n6) node [text=black, pos=0.4, left, fill=none] {};
                \draw (n1) -- (n7) node [text=black, pos=0.4, left, fill=none] {};
                
			\end{tikzpicture}
		 
        \caption*{Potentially good} 
        
		\end{figure}
  \end{center}
		\end{minipage}
  \end{center}
Applying Theorem \ref{wholealgorithm}, we must check the valuation of the absolute invariant associated to the stable model tree equivalent to potentially multiplicative reduction $I_1^1$, since by Proposition \ref{potgoodred} $E/K$ has potentially good reduction if and only if $\textup{ord}(I_1^1)\geq 0$. We can write out $I_1^1$ in terms of $X_1$, $X_2$, $X_3$ and $X_4$, and let $X_4\rightarrow\infty$ to get an expression\footnote{One can check that the value of the resulting invariant is consistent with starting with a Weierstrass model of the curve that has even degree.} in terms of $X_1$, $X_2$ and $X_3$. Evaluating this on the roots of $f(x)$, we obtain
\begin{equation}
    I_1^1=\frac{(x_1-x_3)(x_2-x_3)}{(x_1-x_2)^2}+\frac{(x_1-x_2)(x_3-x_2)}{(x_1-x_3)^2}+\frac{(x_3-x_1)(x_2-x_1)}{(x_3-x_2)^2}.
\end{equation}
Theorem \ref{wholealgorithm} tells us that $E/K$ has potentially multiplicative reduction if and only if $\textup{ord}(I_1^1)<0$. Writing in terms of the coefficients $a$ and $b$, we obtain
\begin{equation}
    I_1^1= \frac{3^3\cdot 2^4a^3}{4a^3+27b^2}-3=\frac{j_E}{16}-3. 
\end{equation}
Since $K$ was assumed to have odd residue characteristic, Theorem \ref{wholealgorithm} tells us that $E/K$ has potentially multiplicative reduction if and only if $\textup{ord}(j_E)<0$, and recovers this well known fact about elliptic curves. Theorem \ref{wholealgorithm} tells us that if $\textup{ord}(I_1^1)<0$ then $\delta=-\textup{ord}(I_1^1)/2=-\textup{ord}(j_E)/2$. Applying Proposition \ref{byfromtc} to obtain the BY tree and using this to construct the dual graph of the special fibre of the minimal regular model would tell us that the special fibre has $2\delta=-\textup{ord}(j_E)$ components, which is also well known. However, we cannot readily apply the BY tree construction to elliptic curves since Theorem 5.18 of \cite{m2d2} requires $g\geq 2$.
\end{example}

We now proceed to prove Theorem \ref{wholealgorithm}. We will need the following proposition. 

\begin{proposition}\label{uniquelydetermines}
Let $(T,I)\in \mathbf{T}_g$ and let $n_{T,I}$ be the number of distinct distances between the vertices of $T$ with the ordering $I$ on the distances between the vertices. Suppose $(T',I')\in \mathbf{T}_g$ with $(T,I)\neq (T',I')$ and $n_{T',I'}\geq n_{T,I}$. Then 
\begin{equation}
R_{T,I,n_{T,I}}\neq R_{T',I',n_{T,I}}^{\sigma} \textup{ for any $\sigma\in S_{2g+2}$}. 
\end{equation}
\end{proposition}

In order to prove Proposition \ref{uniquelydetermines}, we will first need the following lemmata. 

\begin{lemma}\label{singletonslemma}
Let $(T,I)\in \mathbf{T}_g$ and let $n_{T,I}$ be the number of distinct distances between the vertices of $T$ with ordering $I$. Fix $R_{T,I,n_{T,I}}^\sigma$ a summand of $\textup{Inv}_{T,I,n_{T,I}}$. Then $R_{T,I,n_{T,I}}^\sigma$ uniquely determines 
\begin{align}
P_{T}&=\{v:\textup{ $v$ is a proper vertex of $T$ and has at least one singleton vertex attached}\},\\
L_{T} &=\{v:\textup{ $v$ is a proper leaf of $T$}\} \textup{ and } \\
S_{T,v}&=\{s_i:s_i\textup{ is attached to $v$ in $T$}\} \textup{ for all $v\in P_{T}$}
\end{align}
for a fixed labelling of the singletons $s_1,\dots,s_{2g+2}$ of $T$.
\end{lemma}

\begin{proof}
It suffices to prove that $R_{T,I,n_{T,I}}$ uniquely determines $(T,I)$ for a fixed labelling of the singletons, since acting by $\sigma$ corresponds to relabelling the singletons of $T$. 

We will prove the following three statements, from which the lemma follows. Let $s_r$ be a singleton of $T$ and let $(rj,kl)^{w_r}$ be a factor of greatest exponent where $r$ appears. Then
\begin{enumerate}[(i)]
\item $s_r$ is attached to a proper leaf vertex $v_L$ in $T$ if and only if there does not exist a factor $(rj',kl)^{w_r}$ where $j'$ appears in a factor of exponent greater than $w_r$.
\item If $s_r$ is attached to a proper leaf vertex $v_L$, the set of other singleton vertices attached to $v_L$ is
\begin{equation}
\{s_q : (rq,kl)^{w_r} \textup{ is a factor of }R_{T,I,n_{T,I}}\}.
\end{equation} 
\item If $s_r$ is attached to a non-leaf proper vertex $v_P$, the set of other singleton vertices attached to $v_L$ is 
\begin{equation}
\{s_q : (rq,kl)^{w_i} \textup{ is a factor of }R_{T,I,n_{T,I}}\textup{ and $q$ does not appear in a factor of exponent $>w_r$}\}.
\end{equation}
\end{enumerate}

We start by proving two claims.

\textit{Claim 1:} If $s_r$ is attached to a proper leaf vertex $v_L$ in $T$ and $(rj,kl)^{w_r}$ is the factor of greatest exponent where $r$ appears in $R_{T,I,n_{T,I}}$ then $s_j$ must be attached to $v_L$. To see this, let $v$ denote the vertex in $C_{rj}$ on the shortest (in terms of number of edges) path from $C_{rj}$ to $C_{kl}$ in $T$ and $w$ the vertex in $C_{kl}$ on this path. Suppose for a contradiction that $s_j$ is not attached to $v_L$. Then if $s_{j'}$ is attached to $v_L$ we have $\delta(C_{rj'},C_{kl})=\delta(v_L,w)>\delta(v,w)=\delta(C_{rj},C_{kl})$ since $v$ lies on the path from $v_L$ to $w$, which by the definition of $R_{T,I,n_{T,I}}$ contradicts the fact that $(ri,kl)^{w_r}$ is the factor of greatest exponent where $r$ appears. 

\textit{Claim 2:} Suppose $s_r$ is attached to a non-leaf proper vertex $v_P$ in $T$ and let $(rj,kl)^{w_r}$ be the factor of greatest exponent where $r$ appears. Let $v$ denote the vertex in $C_{rj}$ on the shortest path from $C_{rj}$ to $C_{kl}$ and $w$ the vertex in $C_{kl}$ on this path. We claim that $v=v_P$. For a contradiction, suppose $v\neq v_P$ and let $v_L$ be a proper leaf such that $v_P$ lies on the path from $v_L$ to $v$. Suppose $s_{j'}$ is attached to $v_L$. Then $\delta(C_{rj'},C_{kl})=\delta(v_P,w)>\delta(v,w)=\delta(C_{rj},C_{kl})$, contradicting the fact that $(rj,kl)^{w_r}$ is the factor of greatest exponent where $r$ appears. 

For the forwards direction of $(i)$, suppose $s_r$ is attached to a proper leaf vertex $v_L$ in $T$ and let $(rj,kl)^{w_r}$ be the factor of greatest exponent where $r$ appears. Suppose for a contradiction that $(rj',kl)^{w_r}$ appears as a factor of $R_{T,I,n_{T,I}}$ and that $j'$ appears in the factor $(ij',k'l')^{w_{j'}}$ with $w_{j'}>w_r$. Since $(rj',kl)^{w_r}$ appears as a factor of $R_{T,I,n_{T,I}}$, by Claim 1 $s_j'$ is attached to $v_L$ in $T$. Then since $s_r$ and $s_j'$ are attached to $v_L$, we have $\delta(C_{ir},C_{k'l'})=\delta(C_{ij'},C_{k'l'})$ and so $(ir,k'l')^{w_{j'}}$ appears as a factor of $R_{T,I,n_{T,I}}$, contradicting the fact that $(rj,kl)^{w_r}$ is the factor of greatest exponent where $r$ appears.

For the converse of $(i)$, suppose $s_r$ is attached to a non-leaf proper vertex $v_P$ in $T$ and let $(rj,kl)^{w_r}$ be the factor of greatest exponent where $r$ appears. By Claim 2, $v_P$ is the vertex in $C_{rj}$ on the shortest path from $C_{rj}$ to $C_{kl}$. Let $w$ denote the vertex in $C_{kl}$ on this path and let $v_L$ be a proper leaf such that $v_P$ lies on the path from $v_L$ to $w$. Then if $s_{j'}$ is attached to $v_L$ we have $\delta(C_{rj'},C_{kl})=\delta(v_P,w)=\delta(C_{rj},C_{kl})$ and so $(rj',kl)^{w_r}$ appears as a factor. But for $s_i$ attached to $v_L$, we have $\delta(C_{ij'},C_{kl})=\delta(v_L,w)>\delta(v_P,w)=\delta(C_{rj},C_{kl})$, and so $j'$ appears in a factor with a greater exponent.

For $(ii)$, if $s_r$ is attached to a proper leaf vertex $v_L$, the set of singleton vertices attached to $v_L$ is
\begin{equation}
S_{T,v_L}=\{s_r\}\cup\{s_q : (rq,kl)^{w_r} \textup{ is a factor of }R_{T,I,n_{T,I}}\},
\end{equation} 
by Claim 1 and the fact that if $s_q\in S_{T,v_L}$ then $(rq,kl)^{w_r}$ appears as a factor of $R_{T,I,n_{T,I}}$ since $s_r\in S_{T,v_L}$.  

For $(iii)$, we want to show that if $s_r$ is attached to a non-leaf proper vertex $v_P$, the set of singleton vertices attached to $v_P$ is
\begin{equation}
S_{T,v_P}=\{s_r\}\cup\{s_q : (rq,kl)^{w_r} \textup{ is a factor of }R_{T,I,n_{T,I}}\textup{ and $q$ does not appear in a factor of exponent $>w_r$}\}.
\end{equation}
The forward inclusion follows from the fact that if $s_q\in S_{T,v_P}$ then $(rq,kl)^{w_r}$ appears as a factor of $R_{T,I,n_{T,I}}$. For the reverse inclusion, we will show that if $s_q\not\in S_{T,v_P}$ then either $(rq,kl)^{w_r}$ is not a factor of $R_{T,I,n_{T,I}}$ or $q$ appears in a factor of exponent $>w_r$. Suppose $s_q\not\in S_{T,v_P}$. Since $s_r\in v_P$, either $s_q$ is attached to a vertex $v'$ such that $v_P$ lies on the path from $v'$ to $C_{kl}$, or $s_q$ is attached to a vertex such that the path between $C_{rq}$ and $C_{kl}$ goes from a vertex $v$ in $C_{rq}$ which lies on the path between $v_P$ and $C_{kl}$. In the former case, there exists an $s_i$ attached to a proper leaf $v_L$ (possibly with $v_L=v'$) such that $v'$ and $v_P$ lie on the path from $v_L$ to $C_{kl}$. As such, $(iq,kl)$ appears with exponent greater than $w_r$. In the latter case, $(rq,kl)$ appears with exponent smaller than $w_r$ and so $(rq,kl)^{w_r}$ is not a factor of $R_{T,I,n_{T,I}}$.

In this fixed labelling of the singletons, this uniquely determines $P_{T}$, $L_{T}$ and $S_{T,v}$ for $v\in P_{T}$ since $n_{T,I}$ is the number of distinct distances between the proper vertices of $(T,I)$, so all variables $X_1,\dots,X_{2g+2}$ appear in the expression for $R_{T,I,n_{T,I}}$ and the position of all corresponding singletons can be identified.
\end{proof}

\begin{lemma}\label{summandslemma}
Let $(T,I)\in \mathbf{T}_g$ and let $n_{T,I}$ be the number of distinct distances between the vertices of $T$ with ordering $I$. Then $R_{T,I,n_{T,I}}^\sigma$ uniquely determines $(T,I)$. 
\end{lemma}

\begin{proof}
It suffices to prove that $R_{T,I,n_{T,I}}$ uniquely determines $(T,I)$, since acting by $\sigma$ corresponds to relabelling the singletons of $T$. We will first prove that $R_{T,I,n_{T,I}}$ uniquely determines $T$ with a fixed labelling of the singletons. We proceed by induction on $n=\# P_{T}$, the number of proper vertices with at least one singleton attached, which can be obtained from $R_{T,I,n_{T,I}}$ by Lemma \ref{singletonslemma}.

\textit{Base case:} If $n=2$, $T$ consists of two proper leaf vertices attached by an edge. By Lemma \ref{singletonslemma}, the set of singletons attached to each proper leaf is uniquely determined by $R_{T,I,1}$, hence $T$ with a corresponding labelling of the singletons is uniquely determined by $R_{T,I,1}$.

\textit{Inductive hypothesis:} We will assume that $R_{T,I,n_{T,I}}$ uniquely determines $T$ with a fixed labelling of the singletons for $\# P_{T}=N$. 

\textit{Inductive step:}
Suppose $\# P_{T}=N+1$. By Lemma \ref{singletonslemma}, $R_{T,I,n_{T,I}}$ uniquely determines $P_{T}$, $L_{T}$ and $S_{T,v}$ for all $v\in P_{T}$. Fix $v_L\in L_{T}$ a leaf of $T$ and assume, without loss of generality, that 
\begin{equation}
S_{T,v_L}=\{s_1,\dots,s_a\}. 
\end{equation}
In the expression for $R_{T,I,n_{T,I}}$, delete any factors containing the variables $X_2,\dots,X_a$ and denote by $R'$ the resulting expression. 

Let $e$ denote the number of distinct exponents of the factors $(ij,kl)$ in the expression for $R'$ and define $e_{ijkl}=e-m+1$ for $(ij,kl)$ a factor of $R' $ of $m$-th largest exponent. Now define
\begin{equation}
R'' = \prod_{\substack{(ij,kl) \textup{ a factor } \\\textup{of } R' }}(ij,kl)^{e_{ijkl}}. 
\end{equation}
Denote by $v$ the vertex attached to the fixed proper leaf $v_L$ in $T$. Considering $T$ as a graph with a fixed labelling of the singletons, let $(T_N,I_N)$ denote $T$ but with $v_L$ removed and $s_1$ attached to $v$, and with the ordering on the distances between the vertices $I_N$ induced by $I$. We claim that 
\begin{equation}
R''=R_{T_N,I_N,n_{T_N,I_N}}.
\end{equation}
Henceforth, let $C_{ij}$ denote the shortest path between singletons $s_i$ and $s_j$ in $T$ and let $\Tilde{C}_{ij}$ denote the shortest path between singletons $s_i$ and $s_j$ in $T_N$. In order to prove the claim, we will first show that $(ij,kl)$ appears as a factor in $R'' $ if and only if it appears in $R_{T_N,I_N,n_{T_N,I_N}}$. For $i,j,k,l\neq 1$, it is clear that $(ij,kl)$ appears as a factor of $R''$ if and only if it appears as a factor of $R_{T_N,I_N,n_{T_N,I_N}}$. This is due to the fact that for $i,j,k,l\not\in\{1,\dots, a\}$, $C_{ij}=\Tilde{C}_{ij}$ and $C_{kl}=\Tilde{C}_{kl}$ as paths in $T$ and $T_N$ by the definition of $T_N$. 

We will now show that $(1j,kl)$ appears as a factor in $R'' $ if and only if it appears in $R_{T_N,I_N,n_{T_N,I_N}}$. For the forwards direction, suppose $(1j,kl)$ appears as a factor in $R'' $. We claim that the shortest path between $C_{1j}$ and $C_{kl}$ goes between $w$ and $w'$ in $T$ if and only if the shortest path between $\Tilde{C}_{1j}$ and $\Tilde{C}_{kl}$ goes between $w$ and $w'$ in $T_N$. We first note that $s_j$, $s_k$ and $s_l$ are not attached to $v_L$ in $T$ since otherwise this factor would have been deleted. We also note that $s_k$ and $s_l$ are not attached to $v$ in $T$ since otherwise $s_j$ is attached to $v_L$ which is a contradiction. Thus, $C_{kl}=\Tilde{C}_{kl}$ as paths in $T$ and $T_N$, and $\Tilde{C}_{1j}$ is $C_{1j}$ but with $v_L$ removed, since $v$ is the second vertex on any path from $v_L$ to another vertex in $T$ and $s_1$ is assumed to be attached to $v$ in $T_N$. Since $v_L$ is a proper leaf vertex, this means that the shortest path between $C_{1j}$ and $C_{kl}$ goes between $w$ and $w'$ in $T$ if and only if the shortest path between $\Tilde{C}_{1j}$ and $\Tilde{C}_{kl}$ goes between $w$ and $w'$ in $T_N$. Hence $(1j,kl)$ appears as a factor in $R_{T_N,I_N,n_{T_N}}$. 

For the converse, suppose $(1j,kl)$ appears as a factor in $R_{T_N,I_N,n_{T_N}}$. Then $s_j$, $s_k$ and $s_l$ are not attached to $v_L$ since $v_L$ is not a vertex in $T_N$ and $s_k$ and $s_l$ are not attached to $v$ since otherwise $\delta(\Tilde{C}_{1j},\Tilde{C}_{kl})=0$ meaning that the factor would not show up in $R_{T_N,I_N,n_{T_N}}$. Then similarly to the above, $C_{kl}=\Tilde{C}_{kl}$ as paths in $T$ and $T_N$ and $\Tilde{C}_{1j}$ is $C_{1j}$ but with $v_L$ removed. This means that the shortest path between $\Tilde{C}_{1j}$ and $\Tilde{C}_{kl}$ goes between $w$ and $w'$ in $T_N$ if and only if the shortest path between $C_{1j}$ and $C_{kl}$ goes between $w$ and $w'$ in $T$. Hence $(1j,kl)$ appears as a factor in $R_{T,I,n_{T,I}}$, and thus appears in $R''$, since $s_j$, $s_k$ and $s_l$ are not attached to $v_L$ so this factor is not deleted when forming $R''$ from $R_{T,I,n_{T,I}}$.

By definition, the ordering on the distances between the vertices in $T_N$ is the same as the ordering on distances between the those vertices in $T$. The above shows that for $(ij,kl)$ a factor of $R''$ and $R_{T_N,I_N,n_{T_N,I_N}}$, the shortest path between $\Tilde{C}_{ij}$ and $\Tilde{C}_{kl}$ goes between $w$ and $w'$ in $T_N$ if and only if the shortest path between $C_{ij}$ and $C_{kl}$ goes between $w$ and $w'$ in $T$. Thus $\delta(C_{ij},C_{kl})$ is the $m$-th largest remaining distance in $T$ (once $v_L$ is removed) if and only if $\delta(\Tilde{C}_{ij},\Tilde{C}_{kl})$ is the $m$-th largest distance in $T_N$. This means that $(ij,kl)$ has the same exponent in  $R''$ as in $R_{T_N,I_N,n_{T_N,I_N}}$ and so 
\begin{equation}
R'' =R_{T_N,I_N,n_{T_N,I_N}}.
\end{equation}

By the inductive hypothesis, $R_{T_N,I_N,n_{T_N,I_N}}$ uniquely determines $T_N$ with a fixed labelling of the singletons. By the definition of $T_N$ coming from $T$, this uniquely determines $T$ as being $T_N$ with an extra vertex $v_L$ and an edge between $v_L$ and $v$, where $v$ is the vertex of $T$ attached to the singleton $s_1$, and with $s_1$ then removed from $v$ attached to $v_L$. The set of other vertices attached to $v_L$, is determined by $R_{T,I,n_{T,I}}$ by Lemma \ref{singletonslemma}. Thus $R_{T,I,n_{T,I}}$ uniquely determines $T$ as an unweighted tree and with a fixed labelling of the singletons of $T$.

By the definition of $R_{T,I,n_{T,I}}$, if $v$ is the vertex on the path from $C_{ij}$ to $C_{kl}$ in $T$ and $w$ is the vertex in $C_{kl}$ then $\delta(v,w)=\delta(C_{ij},C_{kl})=\delta_m(T,I)$ if and only if the exponent of $(ij,kl)$ in $R_{T,I,n_{T,I}} $ is $n_{T,I}-m+1$, and $\delta(C_{ij},C_{kl})=\delta(v,w)>\delta(v',w')=\delta(C_{i'j'},C_{k'l'})$ if and only if the exponent of $(ij,kl)$ is greater than the exponent of $(i'j',k'l')$. Hence, $R_{T,I,n_{T,I}}$ uniquely determines $I$, the ordering on the distances between the vertices in $T$.
\end{proof}

\begin{proof}[Proof of Proposition \ref{uniquelydetermines}]
By Lemma \ref{summandslemma}, if $R_{T,I,n_{T,I}} =R_{T',I',n_{T,I}}^\sigma $ then $T\cong T'$ as unweighted graphs and $I$ is equivalent to $I'$. 
\end{proof}

\begin{proof}[Proof of Theorem \ref{wholealgorithm}]
By Propositions \ref{potgoodredlemma} and \ref{potgoodred}, $T_C$ consists of a single vertex, or equivalently $C/K$ has potentially good reduction, if and only if $\textup{ord}(\textup{Inv}_{T,I,1}(C))\geq 0$ for every $(T,I)\in\mathbf{T}_g$. 

We know that $(T_C,I_C)\in \mathcal{T}_{n+1}(C)$ since $\mathcal{T}_{n+1}(C)$ is the set of all stable model trees with orderings for which $\textup{Inv}_{T,I,n}=\textup{Inv}_{T_C,I_C,n}$. By Theorem \ref{valuationofyprime}, Proposition \ref{correctone} and Theorem \ref{semistablemainlemma}, for $(T,I)\in \mathcal{T}_{n+1}(C)$ a tree and ordering with the largest value of $K_{n+1}(-,-)$ satisfying $B_{n+1}(T,I,C)=\underset{(T',I')\in\mathcal{T}_{n+1}(C)}{\textup{max}} B_{n+1}(T',I',C)$,
\begin{equation}
R_{T_C,I_C,n+1}=R_{T,I,n+1}^{\sigma}\textup{ for some $\sigma\in S_{2g+2}$} \quad \textup{ and } \quad \delta_{n+1}(C)=B_{n+1}(T,I,C). 
\end{equation}
By Proposition \ref{uniquelydetermines}, $(T_C,I_C)$ is uniquely determined by $R_{T_C,I_C,n_{T_C}}$. By the definition of $(T_C,I_C)$, if $\delta_i(T_C,I_C)=\delta(v,w)$ then $\delta(v,w)=\delta_i(C)$ in $T_C$.
\end{proof}

\section{Full description of the list of absolute invariants and algorithm for genus 2}\label{genus2section}
In this section we give a complete description of the absolute invariants presented in this paper that recover the dual graph of the special fibre of the minimal regular model of a semistable genus $2$ curve, since all genus $2$ curves are hyperelliptic. We present a theorem and a table describing how the series of valuations of absolute invariants associated to a genus $2$ curve uniquely determines the dual graph of the special fibre. We explicitly list the absolute invariants needed to recover the dual graph, and we give an example showing how to calculate the special fibre of a specific genus $2$ curve using this list. 

\begin{theorem}\label{genus2theorem}
Let $C$ be a semistable genus $2$ curve over local field $K$ of odd residue characteristic, and denote by $v$ the valuation with respect to a uniformiser of $K$.
\begin{enumerate}[1.]
        \item $C/K$ has potentially good reduction if and only if $v(X(C))\geq 0$ for every absolute invariant $X$ listed below Table \ref{genus2table} labelled with one letter. 
        \item If $C/K$ does not have potentially good reduction, the dual graph of the special fibre of the minimal regular model of $f(x)$ over $K^{\textup{unr}}$ is uniquely determined by the valuations of the absolute invariants listed below Table \ref{genus2table} using the following procedure. 
        \begin{enumerate}[(i)]
            \item Evaluate the value of 
        \begin{equation}
            B_1(X,C)=-\frac{v(X(C))}{2\cdot K_1(X)}
        \end{equation}
         for each absolute invariant $X$ in Table \ref{genus2table} labelled with one letter. Choose the absolute invariant labelled with one letter that has the greatest value of $K_1(-)$ out of those that maximise the value of $B_1(-,C)$ and call this $I_1$.
        \item If the `$I_1I_2$' column is not empty, evaluate the value of 
        \begin{equation}
            B_2(I_1Y,C)=\frac{-v(I_1Y(C))-4K(I_1)\delta_1(C)}{2\cdot K_2(I_1Y)},
        \end{equation}
        for each absolute invariant $I_1Y$ in Table \ref{genus2table} labelled with two letters and with $I_1$ as the first character. Choose the absolute invariant labelled with two letters and with $I_1$ as the first letter that has the greatest value of $K_2(-)$ out of those that maximise the value of $B_2(-,C)$ and call this $I_1I_2$. 
        \item If the `$I_1I_2I_3$' column is not empty, evaluate the value of 
        \begin{equation}
            B_3(I_1I_2Z,C)=\frac{-v(I_1I_2Z(C))-4K_2(I_1I_2)\delta_2(C)-6K_1(I_1)\delta_1(C)}{2\cdot K_3(I_1I_2Z)},
        \end{equation}
        for each absolute invariant $I_1I_2Z$ in Table \ref{genus2table} labelled with three letters and with $I_1 I_2$ as the first two letters. Choose the graph labelled with three letters and with $I_1 I_2$ as the first two letters that has the greatest value of $K_3(-)$ out of those that maximise the value of $B_3(-,C)$ and call this $I_1 I_2 I_3$.
        \item The special fibre is given in the row associated to the final calculated absolute invariant, and lengths of the chains of $\mathbb{P}^1$'s are written in the `Lengths of chains' column in terms of $v(I_1(C))$, $v(I_1I_2(C))$ and $v(I_1I_2I_3(C))$.
        \end{enumerate}  
    \end{enumerate}
\end{theorem}

\begin{remark}
We present $24$ absolute invariants from which the dual graph of the special fibre of the minimal regular model of a semistable genus $2$ curve can be obtained, but it is shown in \cite{mestre} and \cite{liu} that the special fibre of the potential stable model of genus $2$ curves can be described in terms of the valuations of combinations of Igusa invariants, of which there are four. Moreover, it is known that the Igusa invariants generate the ring of invariants of genus $2$ curves in odd residue characteristic (\cite{igusa} Proposition 3), so each of the $24$ invariants can be written in terms of Igusa invariants, as illustrated in Appendix \ref{appendix}.
\end{remark}

In the table, the genus $1$ components of the special fibre are labelled by $g1$, and the other components are isomorphic to $\mathbb{P}^1$. We use the notation 

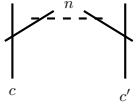
\begin{center}
\begin{minipage}[t]{0.45\textwidth}
\begin{center}
\begin{tikzpicture}
				[scale=0.2, auto=left,every node/.style={circle,fill=black!20,scale=0.6}]

                \draw[black, thick] (0.5,2.5) -- (-2.75,0.5) node [text=black, left, pos=0.5,fill=none] {};
                \draw[black, thick] (2.5,2.5) -- (5.75,0.5) node [text=black, left, pos=0.5,fill=none] {};
                \draw[black, thick, dashed] (-1,2) -- (4,2) node [text=black, above, pos=0.5,fill=none] {$n$};

                \draw[black, thick] (-2.25,3) -- (-2.25,-2) node [text=black, below, pos=1,fill=none] {$c$};
                 \draw[black, thick] (5.25,3) -- (5.25,-2) node [text=black, below, pos=1,fill=none] {$c'$};

			\end{tikzpicture}
            \end{center}
            
$n$ components between two side components $c$ and $c'$. If $n=0$ then $c$ and $c'$ intersect.
\end{minipage}\hspace{20pt}\begin{minipage}[t]{0.45\textwidth}
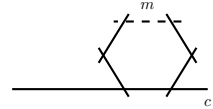
\begin{center}
\begin{tikzpicture}
				[scale=0.2, auto=left,every node/.style={circle,fill=black!20,scale=0.6}]

                \draw[black, thick] (10.5,-2.5) -- (8.5,-5.75) node [text=black, left, pos=0.5,fill=none] {};
                \draw[black, thick] (10.5,-3.5) -- (8.5,-0.25) node [text=black, left, pos=0.5,fill=none] {};

                \draw[black, thick] (4,-2.5) -- (6,-5.75) node [text=black, left, pos=0.5,fill=none] {};
                \draw[black, thick] (4,-3.5) -- (6,-0.25) node [text=black, left, pos=0.5,fill=none] {};

                \draw[black, thick, dashed] (9.5,-0.75) -- (5,-0.75) node [text=black, above, pos=0.5,fill=none] {$m$};
        
                 \draw[black, thick] (11.25,-5.25) -- (-1.75,-5.25) node [text=black, below, pos=0,fill=none] {$c$};
                 
			\end{tikzpicture}
            \end{center}
$m$ components forming an $(m+1)$-gon with the component $c$. 
\end{minipage}
\end{center}
For each special fibre, we have given an example of a cluster picture of a Weierstrass equation $C:y^2=f(x)$ for a curve that has the corresponding special fibre; the cluster picture we present is a canonical representative called the `balanced cluster picture' (see \cite{m2d2} Lemma 15.1). In the notation of the paper, for a fixed genus $2$ curve over a local field $K$, $\textup{Inv}_{T_C,I_C,1}=I_1$, $\textup{Inv}_{T_C,I_C,2}=I_1I_2$ and $\textup{Inv}_{T_C,I_C,3}=I_1I_2I_3$. The absolute invariants referred to in Table \ref{genus2table} are listed below using the shorthand notation 
\begin{equation}
(ij,kl)=\frac{(X_i-X_k)(X_i-X_l)(X_j-X_k)(X_j-X_l)}{(X_i-X_j)^2(X_k-X_l)^2}.
\end{equation}
By $\sum_{S_6/\textup{Stab}}R$, we mean $\sum_{\sigma\in S_6/\textup{Stab}(R)}R^{\sigma}$ where $R^\sigma$ takes the variable $X_i$ to $X_{\sigma(i)}$ in the expression $R$.

\begin{figure}[H]
\caption*{\underline{List of genus $2$ absolute invariants}}
\begin{scriptsize}\begin{align}
&A=\sum_{S_6/\textup{Stab}} (12,34)(12,35)(12,45)  && K_1(A)=3\\
&B=\sum_{S_6/\textup{Stab}} (12,34)(12,35)(12,36)(12,45)(12,46)(12,56)  &&K_1(B)= 6\\
&C=\sum_{S_6/\textup{Stab}} (12,45)(12,46)(12,56)(13,45)(13,46)(13,56)(23,45)(23,46)(23,56)  &&K_1(C)=9\\
&D=\sum_{S_6/\textup{Stab}} (12,34)  &&K_1(D)=1\\
&E=\sum_{S_6/\textup{Stab}} (12,34)(12,56)(34,56)  &&K_1(E)=3\\
&F=\sum_{S_6/\textup{Stab}} (12,34)(12,56) &&K_1(F)=2\\
&AA=\sum_{S_6/\textup{Stab}} (12,34)^2(12,35)^2(12,45)^2(34,16)(35,16)(45,16)(34,26)(35,26)(45,26) &&K_2(AA)=6\\
&AB=\sum_{S_6/\textup{Stab}} (12,34)^2(12,35)^2(12,45)^2(12,63)(12,64)(12,65) &&K_2(AB)=3\\
&AC=\sum_{S_6/\textup{Stab}} (12,34)^2(12,35)^2(12,45)^2(34,16)(35,16)(45,16)(34,26)(35,26)(45,26)(12,63)(12,64)(12,65) &&K_2(AC)= 9\\
&DA=\sum_{S_6/\textup{Stab}} (12,34)^2(12,56) &&K_2(DA)=1 \\
&DB=\sum_{S_6/\textup{Stab}} (12,34)^2(12,56)(34,56) &&K_2(DB)= 2\\
&DC=\sum_{S_6/\textup{Stab}} (12,34)^2(12,56)(12,35)(12,36)(12,45)(12,46) &&K_2(DC)= 5\\
&DD=\sum_{S_6/\textup{Stab}} (12,34)^2(12,56)(12,35)(12,36)(12,45)(12,46)(34,56)(34,15)(34,16)(34,25)(34,26) &&K_2(DD)= 10\\
&DE=\sum_{S_6/\textup{Stab}} (12,34)^2(12,35)(12,45) &&K_2(DE)= 2\\
&DF=\sum_{S_6/\textup{Stab}} (12,34)^2(12,35)(12,45)(34,16)(34,26) &&K_2(DF)= 4\\
&FA=\sum_{S_6/\textup{Stab}} (12,34)^2(12,56)^2(34,56) &&K_2(FA)=1 \\
&FB=\sum_{S_6/\textup{Stab}} (12,34)^2(12,56)^2(12,35)(12,36)(12,45)(12,46) &&K_2(FB)= 4\\
&FC=\sum_{S_6/\textup{Stab}} (12,34)^2(12,56)^2(34,56)(12,35)(12,36)(12,45)(12,46) &&K_2(FC)= 5\\
&DAA=\sum_{S_6/\textup{Stab}} (12,34)^3(12,56)^2(34,56) &&K_3(DAA)=1 \\
&DAB=\sum_{S_6/\textup{Stab}} (12,34)^3(12,56)^2(12,35)(12,36)(12,45)(12,46) &&K_3(DAB)=4 \\
&DAC=\sum_{S_6/\textup{Stab}} (12,34)^3(12,56)^2(34,56)(12,35)(12,36)(12,45)(12,46) &&K_3(DAC)=5 \\
&DEA=\sum_{S_6/\textup{Stab}} (12,34)^3(12,35)^2(12,45)^2(12,36)(12,46)(12,56) &&K_3(DEA)= 3\\
&DEB=\sum_{S_6/\textup{Stab}} (12,34)^3(12,35)^2(12,45)^2(34,16)(34,26) && K_3(DEB)=2 \\
&DEC=\sum_{S_6/\textup{Stab}} (12,34)^3(12,35)^2(12,45)^2(12,36)(12,46)(12,56)(34,16)(34,26) &&K_3(DEC)=5 
\end{align}\end{scriptsize}
\end{figure}

\newpage

\setlength\extrarowheight{2pt}
\begin{center}
   {\setlength{\LTleft}{-0.05in}
 \setlength{\LTright}{0in}
\begin{longtable}{|>{\centering\arraybackslash}m{0.5cm}|>{\centering\arraybackslash}m{0.6cm}|>{\centering\arraybackslash}m{0.9cm}|>{\centering\arraybackslash}m{3.3cm}|>{\centering\arraybackslash}m{5.8cm}|>{\centering\arraybackslash}m{3.4cm}|} \caption{Genus $2$ algorithm}\label{genus2table} \\
\hline $I_1$ & $I_1I_2$ & $I_1I_2I_3$ & Special fibre & Length of chains & Balanced cluster picture \\ \hline \hline 
$A$
    & $AA$ & & 
    \begin{tikzpicture}
				[scale=0.2, auto=left,every node/.style={circle,fill=black!20,scale=0.6}]

                \node[white] (n1) at (3,8) {};
                
                \draw[black, thick] (0.5,5) -- (-2.75,3) node [text=black, left, pos=0.5,fill=none] {};
                \draw[black, thick] (2.5,5) -- (5.75,3) node [text=black, left, pos=0.5,fill=none] {};
                \draw[black, thick, dashed] (-1,4.5) -- (4,4.5) node [text=black, above, pos=0.5,fill=none] {$a-1$};

                \draw[black, thick] (2.5,2.5) -- (5.75,0.5) node [text=black, left, pos=0.5,fill=none] {};
                \draw[black, thick] (3.5,2.5) -- (0.25,0.5) node [text=black, left, pos=0.5,fill=none] {};

                \draw[black, thick] (2.5,-4) -- (5.75,-2) node [text=black, left, pos=0.5,fill=none] {};
                \draw[black, thick] (3.5,-4) -- (0.25,-2) node [text=black, left, pos=0.5,fill=none] {};

                \draw[black, thick, dashed] (0.75,1.5) -- (0.75,-3) node [text=black, right, pos=0.5,fill=none] {$b-1$};

                \draw[black, thick] (-2.25,4.25) -- (-2.25,-3) node [text=black, left, pos=1,fill=none] {$g1$};
                 \draw[black, thick] (5.25,4.25) -- (5.25,-3) node [text=black, left, pos=0.5,fill=none] {};
                 
			\end{tikzpicture}
            
            $4a>b$
    
    & 
    
    $a= \frac{1}{24} (2 v(A) - v(AA))$ 
    
    $b = \frac{1}{6} (-4 v(A) + v(AA))$ & 
    
    \scalebox{1.5}{
\clusterpicture            
  \Root[D] {1} {first} {r1};
  \Root[D] {} {r1} {r2};
  \Root[D] {} {r2} {r3};
  \Root[D] {3.5} {r3} {r4};
    \Root[D] {} {r4} {r5};
    \Root[D] {1.5} {r5} {r6};
  \ClusterLD c1[][a] = (r1)(r2)(r3);
    \ClusterLD c2[][\frac{b}{2}] = (r4)(r5);
    \ClusterLD c3[][a] = (c2)(r6);
  \ClusterD c4[0] = (c1)(c2)(c3);
\endclusterpicture} \\ \hline 

$A$ & $AB$ & &
\begin{tikzpicture}
				[scale=0.2, auto=left,every node/.style={circle,fill=black!20,scale=0.6}]

                \node[white] (n1) at (3,8) {};
                
                \draw[black, thick] (0.5,5) -- (-2.75,3) node [text=black, left, pos=0.5,fill=none] {};
                \draw[black, thick] (2.5,5) -- (5.75,3) node [text=black, left, pos=0.5,fill=none] {};
                \draw[black, thick, dashed] (-1,4.5) -- (4,4.5) node [text=black, above, pos=0.5,fill=none] {$a-1$};

                \draw[black, thick] (2.5,2.5) -- (5.75,0.5) node [text=black, left, pos=0.5,fill=none] {};
                \draw[black, thick] (3.5,2.5) -- (0.25,0.5) node [text=black, left, pos=0.5,fill=none] {};

                \draw[black, thick] (2.5,-4) -- (5.75,-2) node [text=black, left, pos=0.5,fill=none] {};
                \draw[black, thick] (3.5,-4) -- (0.25,-2) node [text=black, left, pos=0.5,fill=none] {};

                \draw[black, thick, dashed] (0.75,1.5) -- (0.75,-3) node [text=black, right, pos=0.5,fill=none] {$b-1$};

                \draw[black, thick] (-2.25,4.25) -- (-2.25,-3) node [text=black, left, pos=1,fill=none] {$g1$};
                 \draw[black, thick] (5.25,4.25) -- (5.25,-3) node [text=black, left, pos=0.5,fill=none] {};
                 
			\end{tikzpicture}

$b>4a$
    
& 
$a = \frac{1}{12} (-3 v( A) + v( {AB}))$

$ b =\frac{1}{3} (2 v( A) - v( {AB}))$
& \scalebox{1.5}{
\clusterpicture            
  \Root[D] {1} {first} {r1};
  \Root[D] {} {r1} {r2};
  \Root[D] {} {r2} {r3};
  \Root[D] {3.5} {r3} {r4};
    \Root[D] {} {r4} {r5};
    \Root[D] {1.5} {r5} {r6};
  \ClusterLD c1[][a] = (r1)(r2)(r3);
    \ClusterLD c2[][\frac{b}{2}] = (r4)(r5);
    \ClusterLD c3[][a] = (c2)(r6);
  \ClusterD c4[0] = (c1)(c2)(c3);
\endclusterpicture} \\ \hline 

$A$ & $AC$
                & & 
                \begin{tikzpicture}
				[scale=0.2, auto=left,every node/.style={circle,fill=black!20,scale=0.6}]

                \node[white] (n1) at (3,8) {};
                
                \draw[black, thick] (0.5,5) -- (-2.75,3) node [text=black, left, pos=0.5,fill=none] {};
                \draw[black, thick] (2.5,5) -- (5.75,3) node [text=black, left, pos=0.5,fill=none] {};
                \draw[black, thick, dashed] (-1,4.5) -- (4,4.5) node [text=black, above, pos=0.5,fill=none] {$a-1$};

                \draw[black, thick] (2.5,2.5) -- (5.75,0.5) node [text=black, left, pos=0.5,fill=none] {};
                \draw[black, thick] (3.5,2.5) -- (0.25,0.5) node [text=black, left, pos=0.5,fill=none] {};

                \draw[black, thick] (2.5,-4) -- (5.75,-2) node [text=black, left, pos=0.5,fill=none] {};
                \draw[black, thick] (3.5,-4) -- (0.25,-2) node [text=black, left, pos=0.5,fill=none] {};

                \draw[black, thick, dashed] (0.75,1.5) -- (0.75,-3) node [text=black, right, pos=0.5,fill=none] {$b-1$};

                \draw[black, thick] (-2.25,4.25) -- (-2.25,-3) node [text=black, left, pos=1,fill=none] {$g1$};
                 \draw[black, thick] (5.25,4.25) -- (5.25,-3) node [text=black, left, pos=0.5,fill=none] {};
                 
			\end{tikzpicture}
   
    $4a=b$
                
                & 
                $a=\frac{1}{36} (2 v( A) - v( {AC}))$

                $b=\frac{1}{9} (2 v( A) - v( {AC}))$
                
                & \scalebox{1.5}{
\clusterpicture            
  \Root[D] {1} {first} {r1};
  \Root[D] {} {r1} {r2};
  \Root[D] {} {r2} {r3};
  \Root[D] {3.5} {r3} {r4};
    \Root[D] {} {r4} {r5};
    \Root[D] {1.5} {r5} {r6};
  \ClusterLD c1[][a] = (r1)(r2)(r3);
    \ClusterLD c2[][\frac{b}{2}] = (r4)(r5);
    \ClusterLD c3[][a] = (c2)(r6);
  \ClusterD c4[0] = (c1)(c2)(c3);
\endclusterpicture} \\ \hline 
                
$B$ & & & 
\begin{tikzpicture}
				[scale=0.2, auto=left,every node/.style={circle,fill=black!20,scale=0.6}]

                \node[white] (n1) at (3,8) {};              
                
                \draw[black, thick] (1,4) -- (5.75,2) node [text=black, left, pos=0.5,fill=none] {};
                \draw[black, thick] (2,4) -- (-2.75,2) node [text=black, left, pos=0.5,fill=none] {};

                \draw[black, thick] (1,-4) -- (5.75,-2) node [text=black, left, pos=0.5,fill=none] {};
                \draw[black, thick] (2,-4) -- (-2.75,-2) node [text=black, left, pos=0.5,fill=none] {};

                \draw[black, thick, dashed] (-2.25,3) -- (-2.25,-3) node [text=black, left, pos=0.5,fill=none] {$a-1$};

                 \draw[black, thick] (5.25,4.75) -- (5.25,-4.75) node [text=black, right, pos=1,fill=none] {$g1$};
                 
			\end{tikzpicture}

& 

$a=-\frac{v( B)}{6}$

& \scalebox{1.5}{
\clusterpicture            
  \Root[D] {1} {first} {r1};
  \Root[D] {} {r1} {r2};
  \Root[D] {2} {r2} {r3};
  \Root[D] {} {r3} {r4};
    \Root[D] {} {r4} {r5};
    \Root[D] {} {r5} {r6};
  \ClusterLD c1[][\frac{a}{2}] = (r1)(r2);
  \ClusterD c4[0] = (c1)(r3)(r4)(r5)(r6);
\endclusterpicture} \\ \hline 

$C$ & & & 

\begin{tikzpicture}
				[scale=0.2, auto=left,every node/.style={circle,fill=black!20,scale=0.6}]

                \node[white] (n1) at (3,5) {};

                \draw[black, thick] (0.5,2.5) -- (-2.75,0.5) node [text=black, left, pos=0.5,fill=none] {};
                \draw[black, thick] (2.5,2.5) -- (5.75,0.5) node [text=black, left, pos=0.5,fill=none] {};
                \draw[black, thick, dashed] (-1,2) -- (4,2) node [text=black, above, pos=0.5,fill=none] {$a-1$};

                \draw[black, thick] (-2.25,3) -- (-2.25,-2) node [text=black, left, pos=1,fill=none] {$g1$};
                 \draw[black, thick] (5.25,3) -- (5.25,-2) node [text=black, left, pos=1,fill=none] {$g1$};

                  \node[white] (n2) at (3,-2) {};
                 
			\end{tikzpicture}
&
$a=-\frac{v( C)}{36}$

& \scalebox{1.5}{
\clusterpicture            
  \Root[D] {1} {first} {r1};
  \Root[D] {} {r1} {r2};
  \Root[D] {} {r2} {r3};
  \Root[D] {2} {r3} {r4};
    \Root[D] {} {r4} {r5};
    \Root[D] {} {r5} {r6};
  \ClusterLD c1[][a] = (r1)(r2)(r3);
  \ClusterLD c2[][a] = (r4)(r5)(r6);
  \ClusterD c4[0] = (c1)(c2);
\endclusterpicture} \\ \hline 

$D$
& 
$DA$
& $DAA$ &

\begin{tikzpicture}
				[scale=0.2, auto=left,every node/.style={circle,fill=black!20,scale=0.6}]

                \node[white] (n1) at (3,8) {};
                
                \draw[black, thick] (0.5,5) -- (-2.75,3) node [text=black, left, pos=0.5,fill=none] {};
                \draw[black, thick] (2.5,5) -- (5.75,3) node [text=black, left, pos=0.5,fill=none] {};
                \draw[black, thick, dashed] (-1,4.5) -- (4,4.5) node [text=black, above, scale=0.8,pos=0.5,fill=none] {$a-1$};

                \draw[black, thick] (0.5,2.5) -- (-2.75,0.5) node [text=black, left, pos=0.5,fill=none] {};
                \draw[black, thick] (2.5,2.5) -- (5.75,0.5) node [text=black, left, pos=0.5,fill=none] {};
                \draw[black, thick, dashed] (-1,2) -- (4,2) node [text=black, above, scale=0.8,pos=0.5,fill=none] {$b-1$};

                \draw[black, thick] (0.5,0) -- (-2.75,-2) node [text=black, left, pos=0.5,fill=none] {};
                \draw[black, thick] (2.5,0) -- (5.75,-2) node [text=black, left, pos=0.5,fill=none] {};
                \draw[black, thick, dashed] (-1,-0.5) -- (4,-0.5) node [text=black, above, scale=0.8,pos=0.5,fill=none] {$c-1$};

                \draw[black, thick] (-2.25,4.25) -- (-2.25,-2.75) node [text=black, left, pos=0.5,fill=none] {};
                 \draw[black, thick] (5.25,4.25) -- (5.25,-2.75) node [text=black, left, pos=0.5,fill=none] {};
                 
			\end{tikzpicture}

$b+c>a>b>c$

& 

$a =\frac{1}{2} (2 v( D) - 3 v( {DA}) + v( {DAA}))$

$b =\frac{1}{2} (-4 v( D) + 3 v( {DA}) - v( {DAA}))$

$c =\frac{1}{2} (2 v( D) + v( {DA}) - v( {DAA}))$

& \scalebox{1.5}{
\clusterpicture            
  \Root[D] {1} {first} {r1};
  \Root[D] {} {r1} {r2};
  \Root[D] {2.5} {r2} {r3};
  \Root[D] {} {r3} {r4};
    \Root[D] {2.5} {r4} {r5};
    \Root[D] {} {r5} {r6};
  \ClusterLD c1[][\frac{a}{2}] = (r1)(r2);
  \ClusterLD c2[][\frac{b}{2}] = (r3)(r4);
  \ClusterLD c3[][\frac{c}{2}] = (r5)(r6);
  \ClusterD c4[0] = (c1)(c2)(c3);
\endclusterpicture} \\ \hline 

$D$
& 
$DA$
& $DAB$ &

\begin{tikzpicture}
				[scale=0.2, auto=left,every node/.style={circle,fill=black!20,scale=0.6}]

                \node[white] (n1) at (3,8) {};
                
                \draw[black, thick] (0.5,5) -- (-2.75,3) node [text=black, left, pos=0.5,fill=none] {};
                \draw[black, thick] (2.5,5) -- (5.75,3) node [text=black, left, pos=0.5,fill=none] {};
                \draw[black, thick, dashed] (-1,4.5) -- (4,4.5) node [text=black, above, scale=0.8,pos=0.5,fill=none] {$a-1$};

                \draw[black, thick] (0.5,2.5) -- (-2.75,0.5) node [text=black, left, pos=0.5,fill=none] {};
                \draw[black, thick] (2.5,2.5) -- (5.75,0.5) node [text=black, left, pos=0.5,fill=none] {};
                \draw[black, thick, dashed] (-1,2) -- (4,2) node [text=black, above, scale=0.8,pos=0.5,fill=none] {$b-1$};

                \draw[black, thick] (0.5,0) -- (-2.75,-2) node [text=black, left, pos=0.5,fill=none] {};
                \draw[black, thick] (2.5,0) -- (5.75,-2) node [text=black, left, pos=0.5,fill=none] {};
                \draw[black, thick, dashed] (-1,-0.5) -- (4,-0.5) node [text=black, above, scale=0.8,pos=0.5,fill=none] {$c-1$};

                \draw[black, thick] (-2.25,4.25) -- (-2.25,-2.75) node [text=black, left, pos=0.5,fill=none] {};
                 \draw[black, thick] (5.25,4.25) -- (5.25,-2.75) node [text=black, left, pos=0.5,fill=none] {};
                 
			\end{tikzpicture}

$a>b+c$

$b>c$

& 

$a=\frac{1}{4}(-v( D) + 2 v( {DA}) - v( {DAB}))$

$b=\frac{1}{4} (-3 v( D) - 2 v( {DA}) + v( {DAB}))$

$c =\frac{1}{4} (9 v( D) - 6 v( {DA}) + v( {DAB}))$

& \scalebox{1.5}{
\clusterpicture            
  \Root[D] {1} {first} {r1};
  \Root[D] {} {r1} {r2};
  \Root[D] {2.5} {r2} {r3};
  \Root[D] {} {r3} {r4};
    \Root[D] {2.5} {r4} {r5};
    \Root[D] {} {r5} {r6};
  \ClusterLD c1[][\frac{a}{2}] = (r1)(r2);
  \ClusterLD c2[][\frac{b}{2}] = (r3)(r4);
  \ClusterLD c3[][\frac{c}{2}] = (r5)(r6);
  \ClusterD c4[0] = (c1)(c2)(c3);
\endclusterpicture} \\ \hline 

$D$
& 
$DA$
& $DAC$ &

\begin{tikzpicture}
				[scale=0.2, auto=left,every node/.style={circle,fill=black!20,scale=0.6}]

                \node[white] (n1) at (3,8) {};
                
                \draw[black, thick] (0.5,5) -- (-2.75,3) node [text=black, left, pos=0.5,fill=none] {};
                \draw[black, thick] (2.5,5) -- (5.75,3) node [text=black, left, pos=0.5,fill=none] {};
                \draw[black, thick, dashed] (-1,4.5) -- (4,4.5) node [text=black, above,scale=0.8, pos=0.5,fill=none] {$a-1$};

                \draw[black, thick] (0.5,2.5) -- (-2.75,0.5) node [text=black, left, pos=0.5,fill=none] {};
                \draw[black, thick] (2.5,2.5) -- (5.75,0.5) node [text=black, left, pos=0.5,fill=none] {};
                \draw[black, thick, dashed] (-1,2) -- (4,2) node [text=black, above, scale=0.8,pos=0.5,fill=none] {$b-1$};

                \draw[black, thick] (0.5,0) -- (-2.75,-2) node [text=black, left, pos=0.5,fill=none] {};
                \draw[black, thick] (2.5,0) -- (5.75,-2) node [text=black, left, pos=0.5,fill=none] {};
                \draw[black, thick, dashed] (-1,-0.5) -- (4,-0.5) node [text=black, above, scale=0.8,pos=0.5,fill=none] {$c-1$};

                \draw[black, thick] (-2.25,4.25) -- (-2.25,-2.75) node [text=black, left, pos=0.5,fill=none] {};
                 \draw[black, thick] (5.25,4.25) -- (5.25,-2.75) node [text=black, left, pos=0.5,fill=none] {};
                 
			\end{tikzpicture}

$b+c=a>b>c$

& 

$a=\frac{1}{2} (-2 v( D) + 3 v( {DA}) - v( {DAC}))$

$b=\frac{1}{2} (-3 v( {DA}) + v( {DAC}))$

$c=\frac{1}{2} (6 v( D) - 5 v( {DA}) + v( {DAC}))$

& \scalebox{1.5}{
\clusterpicture            
  \Root[D] {1} {first} {r1};
  \Root[D] {} {r1} {r2};
  \Root[D] {2.5} {r2} {r3};
  \Root[D] {} {r3} {r4};
    \Root[D] {2.5} {r4} {r5};
    \Root[D] {} {r5} {r6};
  \ClusterLD c1[][\frac{a}{2}] = (r1)(r2);
  \ClusterLD c2[][\frac{b}{2}] = (r3)(r4);
  \ClusterLD c3[][\frac{c}{2}] = (r5)(r6);
  \ClusterD c4[0] = (c1)(c2)(c3);
\endclusterpicture} \\ \hline 

$D$ & $DB$ & &

\begin{tikzpicture}
				[scale=0.2, auto=left,every node/.style={circle,fill=black!20,scale=0.6}]

                \node[white] (n1) at (3,8) {};
                
                \draw[black, thick] (0.5,5) -- (-2.75,3) node [text=black, left, pos=0.5,fill=none] {};
                \draw[black, thick] (2.5,5) -- (5.75,3) node [text=black, left, pos=0.5,fill=none] {};
                \draw[black, thick, dashed] (-1,4.5) -- (4,4.5) node [text=black, above, scale=0.8,pos=0.5,fill=none] {$a-1$};

                \draw[black, thick] (0.5,2.5) -- (-2.75,0.5) node [text=black, left, pos=0.5,fill=none] {};
                \draw[black, thick] (2.5,2.5) -- (5.75,0.5) node [text=black, left, pos=0.5,fill=none] {};
                \draw[black, thick, dashed] (-1,2) -- (4,2) node [text=black, above,scale=0.8, pos=0.5,fill=none] {$a-1$};

                \draw[black, thick] (0.5,0) -- (-2.75,-2) node [text=black, left, pos=0.5,fill=none] {};
                \draw[black, thick] (2.5,0) -- (5.75,-2) node [text=black, left, pos=0.5,fill=none] {};
                \draw[black, thick, dashed] (-1,-0.5) -- (4,-0.5) node [text=black, above, scale=0.8,pos=0.5,fill=none] {$b-1$};

                \draw[black, thick] (-2.25,4.25) -- (-2.25,-2.75) node [text=black, left, pos=0.5,fill=none] {};
                 \draw[black, thick] (5.25,4.25) -- (5.25,-2.75) node [text=black, left, pos=0.5,fill=none] {};
                 
			\end{tikzpicture}

$a>b$

& 
$a = -\frac{v( D)}{2} $

$ b =\frac{1}{2} (3 v( D) - v( {DB}))$
& \scalebox{1.5}{
\clusterpicture            
  \Root[D] {1} {first} {r1};
  \Root[D] {} {r1} {r2};
  \Root[D] {2.5} {r2} {r3};
  \Root[D] {} {r3} {r4};
    \Root[D] {2.5} {r4} {r5};
    \Root[D] {} {r5} {r6};
  \ClusterLD c1[][\frac{a}{2}] = (r1)(r2);
  \ClusterLD c2[][\frac{a}{2}] = (r3)(r4);
  \ClusterLD c3[][\frac{b}{2}] = (r5)(r6);
  \ClusterD c4[0] = (c1)(c2)(c3);
\endclusterpicture} \\ \hline 

$D$ & $DC$ & & 

\begin{tikzpicture}
				[scale=0.2, auto=left,every node/.style={circle,fill=black!20,scale=0.6}]
                
                \node[white] (n1) at (3,2) {};

                \draw[black, thick] (2.5,-2.5) -- (0.5,-5.75) node [text=black, left, pos=0.5,fill=none] {};
                \draw[black, thick] (2.5,-3.5) -- (0.5,-0.25) node [text=black, left, pos=0.5,fill=none] {};

                \draw[black, thick] (-4,-2.5) -- (-2,-5.75) node [text=black, left, pos=0.5,fill=none] {};
                \draw[black, thick] (-4,-3.5) -- (-2,-0.25) node [text=black, left, pos=0.5,fill=none] {};

                \draw[black, thick, dashed] (1.5,-0.75) -- (-3,-0.75) node [text=black, above, pos=0.5,fill=none] {$b-1$};

                \draw[black, thick] (10.5,-2.5) -- (8.5,-5.75) node [text=black, left, pos=0.5,fill=none] {};
                \draw[black, thick] (10.5,-3.5) -- (8.5,-0.25) node [text=black, left, pos=0.5,fill=none] {};

                \draw[black, thick] (4,-2.5) -- (6,-5.75) node [text=black, left, pos=0.5,fill=none] {};
                \draw[black, thick] (4,-3.5) -- (6,-0.25) node [text=black, left, pos=0.5,fill=none] {};

                \draw[black, thick, dashed] (9.5,-0.75) -- (5,-0.75) node [text=black, above, pos=0.5,fill=none] {$a-1$};
        
                 \draw[black, thick] (11.25,-5.25) -- (-4.75,-5.25) node [text=black, left, pos=0.5,fill=none] {};
                 
			\end{tikzpicture}

   $a>b$

& 

$a =\frac{1}{5} (2 v( D) - v( {DC}))$

$b =\frac{1}{5} (-7 v( D) + v( {DC}))$
& \scalebox{1.5}{
\clusterpicture            
  \Root[D] {1} {first} {r1};
  \Root[D] {} {r1} {r2};
  \Root[D] {2.5} {r2} {r3};
  \Root[D] {} {r3} {r4};
    \Root[D] {2} {r4} {r5};
    \Root[D] {} {r5} {r6};
  \ClusterLD c1[][\frac{a}{2}] = (r1)(r2);
  \ClusterLD c2[][\frac{b}{2}] = (r3)(r4);
  \ClusterD c4[0] = (c1)(c2)(r5)(r6);
\endclusterpicture} \\ \hline 

$D$ & $DD$ & & 

\begin{tikzpicture}
				[scale=0.2, auto=left,every node/.style={circle,fill=black!20,scale=0.6}]

                \node[white] (n1) at (3,2) {};
                \node[white] (n1) at (3,-7) {};

                \draw[black, thick] (2.5,-2.5) -- (0.5,-5.75) node [text=black, left, pos=0.5,fill=none] {};
                \draw[black, thick] (2.5,-3.5) -- (0.5,-0.25) node [text=black, left, pos=0.5,fill=none] {};

                \draw[black, thick] (-4,-2.5) -- (-2,-5.75) node [text=black, left, pos=0.5,fill=none] {};
                \draw[black, thick] (-4,-3.5) -- (-2,-0.25) node [text=black, left, pos=0.5,fill=none] {};

                \draw[black, thick, dashed] (1.5,-0.75) -- (-3,-0.75) node [text=black, above, pos=0.5,fill=none] {$a-1$};

                \draw[black, thick] (10.5,-2.5) -- (8.5,-5.75) node [text=black, left, pos=0.5,fill=none] {};
                \draw[black, thick] (10.5,-3.5) -- (8.5,-0.25) node [text=black, left, pos=0.5,fill=none] {};

                \draw[black, thick] (4,-2.5) -- (6,-5.75) node [text=black, left, pos=0.5,fill=none] {};
                \draw[black, thick] (4,-3.5) -- (6,-0.25) node [text=black, left, pos=0.5,fill=none] {};

                \draw[black, thick, dashed] (9.5,-0.75) -- (5,-0.75) node [text=black, above, pos=0.5,fill=none] {$a-1$};
        
                 \draw[black, thick] (11.25,-5.25) -- (-4.75,-5.25) node [text=black, left, pos=0.5,fill=none] {};
                 
			\end{tikzpicture}
& 

$a=-\frac{v( D)}{2}$

& \scalebox{1.5}{
\clusterpicture            
  \Root[D] {1} {first} {r1};
  \Root[D] {} {r1} {r2};
  \Root[D] {2.5} {r2} {r3};
  \Root[D] {} {r3} {r4};
    \Root[D] {2} {r4} {r5};
    \Root[D] {} {r5} {r6};
  \ClusterLD c1[][\frac{a}{2}] = (r1)(r2);
  \ClusterLD c2[][\frac{a}{2}] = (r3)(r4);
  \ClusterD c4[0] = (c1)(c2)(r5)(r6);
\endclusterpicture}  \\ \hline 

$D$ & 
$DE$
& $DEA$ & 

\begin{tikzpicture}
				[scale=0.2, auto=left,every node/.style={circle,fill=black!20,scale=0.6}]

                \node[white] (n1) at (3,8) {};
                
                \draw[black, thick] (0.5,5) -- (-2.75,3) node [text=black, left, pos=0.5,fill=none] {};
                \draw[black, thick] (2.5,5) -- (5.75,3) node [text=black, left, pos=0.5,fill=none] {};
                \draw[black, thick, dashed] (-1,4.5) -- (4,4.5) node [text=black, above, pos=0.5,fill=none] {$b-1$};

                \draw[black, thick] (2.5,2.5) -- (5.75,0.5) node [text=black, left, pos=0.5,fill=none] {};
                \draw[black, thick] (3.5,2.5) -- (0.25,0.5) node [text=black, left, pos=0.5,fill=none] {};

                \draw[black, thick] (2.5,-4) -- (5.75,-2) node [text=black, left, pos=0.5,fill=none] {};
                \draw[black, thick] (3.5,-4) -- (0.25,-2) node [text=black, left, pos=0.5,fill=none] {};

                \draw[black, thick, dashed] (0.75,1.5) -- (0.75,-3) node [text=black, right, pos=0.5,fill=none] {$c-1$};

                \draw[black, thick] (-5,2.5) -- (-1.75,0.5) node [text=black, left, pos=0.5,fill=none] {};
                \draw[black, thick] (-4,2.5) -- (-7.25,0.5) node [text=black, left, pos=0.5,fill=none] {};

                \draw[black, thick] (-5,-4) -- (-1.75,-2) node [text=black, left, pos=0.5,fill=none] {};
                \draw[black, thick] (-4,-4) -- (-7.25,-2) node [text=black, left, pos=0.5,fill=none] {};

                \draw[black, thick, dashed] (-6.75,1.5) -- (-6.75,-3) node [text=black, right, pos=0.5,fill=none] {$a-1$};

                \draw[black, thick] (-2.25,4.25) -- (-2.25,-3) node [text=black, left, pos=0.5,fill=none] {};
                 \draw[black, thick] (5.25,4.25) -- (5.25,-3) node [text=black, left, pos=0.5,fill=none] {};
                 
			\end{tikzpicture}

   $a>c+4b$

& 

$a=\frac{1}{3} (-v( D) + 2 v( {DE}) - v( {DEA}))$

$b=\frac{1}{24} (8 v( D) - 7 v( {DE}) + 2 v( {DEA}))$

$ c=\frac{1}{2} (-4 v( D) + v( {DE}))$

& \scalebox{1.5}{
\clusterpicture            
  \Root[D] {1} {first} {r1};
  \Root[D] {} {r1} {r2};
  \Root[D] {2} {r2} {r3};
  \Root[D] {3} {r3} {r4};
    \Root[D] {} {r4} {r5};
    \Root[D] {2} {r5} {r6};
  \ClusterLD c1[][\frac{a}{2}] = (r1)(r2);
  \ClusterLD c2[][\frac{c}{2}] = (r4)(r5);
  \ClusterLD c3[][b] = (c1)(r3);
  \ClusterLD c4[][b] = (c2)(r6);
  \ClusterD c5[0] = (c3)(c4);
\endclusterpicture}

\\ \hline

$D$ & $DE$ & $DEB$ & 

\begin{tikzpicture}
				[scale=0.2, auto=left,every node/.style={circle,fill=black!20,scale=0.6}]

                \node[white] (n1) at (3,8) {};
                
                \draw[black, thick] (0.5,5) -- (-2.75,3) node [text=black, left, pos=0.5,fill=none] {};
                \draw[black, thick] (2.5,5) -- (5.75,3) node [text=black, left, pos=0.5,fill=none] {};
                \draw[black, thick, dashed] (-1,4.5) -- (4,4.5) node [text=black, above, pos=0.5,fill=none] {$b-1$};

                \draw[black, thick] (2.5,2.5) -- (5.75,0.5) node [text=black, left, pos=0.5,fill=none] {};
                \draw[black, thick] (3.5,2.5) -- (0.25,0.5) node [text=black, left, pos=0.5,fill=none] {};

                \draw[black, thick] (2.5,-4) -- (5.75,-2) node [text=black, left, pos=0.5,fill=none] {};
                \draw[black, thick] (3.5,-4) -- (0.25,-2) node [text=black, left, pos=0.5,fill=none] {};

                \draw[black, thick, dashed] (0.75,1.5) -- (0.75,-3) node [text=black, right, pos=0.5,fill=none] {$c-1$};

                \draw[black, thick] (-5,2.5) -- (-1.75,0.5) node [text=black, left, pos=0.5,fill=none] {};
                \draw[black, thick] (-4,2.5) -- (-7.25,0.5) node [text=black, left, pos=0.5,fill=none] {};

                \draw[black, thick] (-5,-4) -- (-1.75,-2) node [text=black, left, pos=0.5,fill=none] {};
                \draw[black, thick] (-4,-4) -- (-7.25,-2) node [text=black, left, pos=0.5,fill=none] {};

                \draw[black, thick, dashed] (-6.75,1.5) -- (-6.75,-3) node [text=black, right, pos=0.5,fill=none] {$a-1$};

                \draw[black, thick] (-2.25,4.25) -- (-2.25,-3) node [text=black, left, pos=0.5,fill=none] {};
                 \draw[black, thick] (5.25,4.25) -- (5.25,-3) node [text=black, left, pos=0.5,fill=none] {};
                 
			\end{tikzpicture}

   $c+4b>a>c$

& 

$a=\frac{1}{2} (-v( D) - 2 v( {DE}) + v( {DEB}))$

$b=\frac{1}{8} (3 v( D) + v( {DE}) - v( {DEB}))$

$c=\frac{1}{2} (-4 v( D) + v( {DE}))$

& \scalebox{1.5}{
\clusterpicture            
  \Root[D] {1} {first} {r1};
  \Root[D] {} {r1} {r2};
  \Root[D] {2} {r2} {r3};
  \Root[D] {3} {r3} {r4};
    \Root[D] {} {r4} {r5};
    \Root[D] {2} {r5} {r6};
  \ClusterLD c1[][\frac{a}{2}] = (r1)(r2);
  \ClusterLD c2[][\frac{c}{2}] = (r4)(r5);
  \ClusterLD c3[][b] = (c1)(r3);
  \ClusterLD c4[][b] = (c2)(r6);
  \ClusterD c5[0] = (c3)(c4);
\endclusterpicture} 

\\ \hline 

$D$ & $DE$ & $DEC$ & 

\begin{tikzpicture}
				[scale=0.2, auto=left,every node/.style={circle,fill=black!20,scale=0.6}]

                \node[white] (n1) at (3,8) {};
                
                \draw[black, thick] (0.5,5) -- (-2.75,3) node [text=black, left, pos=0.5,fill=none] {};
                \draw[black, thick] (2.5,5) -- (5.75,3) node [text=black, left, pos=0.5,fill=none] {};
                \draw[black, thick, dashed] (-1,4.5) -- (4,4.5) node [text=black, above, pos=0.5,fill=none] {$b-1$};

                \draw[black, thick] (2.5,2.5) -- (5.75,0.5) node [text=black, left, pos=0.5,fill=none] {};
                \draw[black, thick] (3.5,2.5) -- (0.25,0.5) node [text=black, left, pos=0.5,fill=none] {};

                \draw[black, thick] (2.5,-4) -- (5.75,-2) node [text=black, left, pos=0.5,fill=none] {};
                \draw[black, thick] (3.5,-4) -- (0.25,-2) node [text=black, left, pos=0.5,fill=none] {};

                \draw[black, thick, dashed] (0.75,1.5) -- (0.75,-3) node [text=black, right, pos=0.5,fill=none] {$c-1$};

                \draw[black, thick] (-5,2.5) -- (-1.75,0.5) node [text=black, left, pos=0.5,fill=none] {};
                \draw[black, thick] (-4,2.5) -- (-7.25,0.5) node [text=black, left, pos=0.5,fill=none] {};

                \draw[black, thick] (-5,-4) -- (-1.75,-2) node [text=black, left, pos=0.5,fill=none] {};
                \draw[black, thick] (-4,-4) -- (-7.25,-2) node [text=black, left, pos=0.5,fill=none] {};

                \draw[black, thick, dashed] (-6.75,1.5) -- (-6.75,-3) node [text=black, right, pos=0.5,fill=none] {$a-1$};

                \draw[black, thick] (-2.25,4.25) -- (-2.25,-3) node [text=black, left, pos=0.5,fill=none] {};
                 \draw[black, thick] (5.25,4.25) -- (5.25,-3) node [text=black, left, pos=0.5,fill=none] {};
                 
			\end{tikzpicture}

  $a=c+4b$

& 

$a=\frac{1}{5}(- v( D) + 2 v( {DE}) -v( {DEC})$

$b=\frac{1}{40} (12 v( D) - 9 v( {DE}) + 2 v( {DEC}))$

$c=\frac{1}{2} (-4 v( D) + v( {DE}))$

& \scalebox{1.5}{
\clusterpicture            
  \Root[D] {1} {first} {r1};
  \Root[D] {} {r1} {r2};
  \Root[D] {2} {r2} {r3};
  \Root[D] {3} {r3} {r4};
    \Root[D] {} {r4} {r5};
    \Root[D] {2} {r5} {r6};
  \ClusterLD c1[][\frac{a}{2}] = (r1)(r2);
  \ClusterLD c2[][\frac{c}{2}] = (r4)(r5);
  \ClusterLD c3[][b] = (c1)(r3);
  \ClusterLD c4[][b] = (c2)(r6);
  \ClusterD c5[0] = (c3)(c4);
\endclusterpicture}

\\ \hline 

$D$ & $DF$ & & 

\begin{tikzpicture}
				[scale=0.2, auto=left,every node/.style={circle,fill=black!20,scale=0.6}]

                \node[white] (n1) at (3,8) {};
                
                \draw[black, thick] (0.5,5) -- (-2.75,3) node [text=black, left, pos=0.5,fill=none] {};
                \draw[black, thick] (2.5,5) -- (5.75,3) node [text=black, left, pos=0.5,fill=none] {};
                \draw[black, thick, dashed] (-1,4.5) -- (4,4.5) node [text=black, above, pos=0.5,fill=none] {$b-1$};

                \draw[black, thick] (2.5,2.5) -- (5.75,0.5) node [text=black, left, pos=0.5,fill=none] {};
                \draw[black, thick] (3.5,2.5) -- (0.25,0.5) node [text=black, left, pos=0.5,fill=none] {};

                \draw[black, thick] (2.5,-4) -- (5.75,-2) node [text=black, left, pos=0.5,fill=none] {};
                \draw[black, thick] (3.5,-4) -- (0.25,-2) node [text=black, left, pos=0.5,fill=none] {};

                \draw[black, thick, dashed] (0.75,1.5) -- (0.75,-3) node [text=black, right, pos=0.5,fill=none] {$a-1$};

                \draw[black, thick] (-5,2.5) -- (-1.75,0.5) node [text=black, left, pos=0.5,fill=none] {};
                \draw[black, thick] (-4,2.5) -- (-7.25,0.5) node [text=black, left, pos=0.5,fill=none] {};

                \draw[black, thick] (-5,-4) -- (-1.75,-2) node [text=black, left, pos=0.5,fill=none] {};
                \draw[black, thick] (-4,-4) -- (-7.25,-2) node [text=black, left, pos=0.5,fill=none] {};

                \draw[black, thick, dashed] (-6.75,1.5) -- (-6.75,-3) node [text=black, right, pos=0.5,fill=none] {$a-1$};

                \draw[black, thick] (-2.25,4.25) -- (-2.25,-3) node [text=black, left, pos=0.5,fill=none] {};
                 \draw[black, thick] (5.25,4.25) -- (5.25,-3) node [text=black, left, pos=0.5,fill=none] {};
                 
			\end{tikzpicture}

& 

$a=\frac{1}{4} (-6 v( D) + v( {DF}))$

$b=\frac{1}{8} (4 v( D) - v( {DF}))$

& \scalebox{1.5}{
\clusterpicture            
  \Root[D] {1} {first} {r1};
  \Root[D] {} {r1} {r2};
  \Root[D] {2} {r2} {r3};
  \Root[D] {3} {r3} {r4};
    \Root[D] {} {r4} {r5};
    \Root[D] {2} {r5} {r6};
  \ClusterLD c1[][\frac{a}{2}] = (r1)(r2);
  \ClusterLD c2[][\frac{a}{2}] = (r4)(r5);
  \ClusterLD c3[][b] = (c1)(r3);
  \ClusterLD c4[][b] = (c2)(r6);
  \ClusterD c5[0] = (c3)(c4);
\endclusterpicture} \\ \hline 

$E$ & & &
\begin{tikzpicture}
				[scale=0.2, auto=left,every node/.style={circle,fill=black!20,scale=0.6}]

                \node[white] (n1) at (3,8) {};
                
                \draw[black, thick] (0.5,5) -- (-2.75,3) node [text=black, left, pos=0.5,fill=none] {};
                \draw[black, thick] (2.5,5) -- (5.75,3) node [text=black, left, pos=0.5,fill=none] {};
                \draw[black, thick, dashed] (-1,4.5) -- (4,4.5) node [text=black, above, scale=0.8,pos=0.5,fill=none] {$a-1$};

                \draw[black, thick] (0.5,2.5) -- (-2.75,0.5) node [text=black, left, pos=0.5,fill=none] {};
                \draw[black, thick] (2.5,2.5) -- (5.75,0.5) node [text=black, left, pos=0.5,fill=none] {};
                \draw[black, thick, dashed] (-1,2) -- (4,2) node [text=black, above, scale=0.8,pos=0.5,fill=none] {$a-1$};

                \draw[black, thick] (0.5,0) -- (-2.75,-2) node [text=black, left, pos=0.5,fill=none] {};
                \draw[black, thick] (2.5,0) -- (5.75,-2) node [text=black, left, pos=0.5,fill=none] {};
                \draw[black, thick, dashed] (-1,-0.5) -- (4,-0.5) node [text=black, above, scale=0.8,pos=0.5,fill=none] {$a-1$};

                \draw[black, thick] (-2.25,4.25) -- (-2.25,-2.75) node [text=black, left, pos=0.5,fill=none] {};
                 \draw[black, thick] (5.25,4.25) -- (5.25,-2.75) node [text=black, left, pos=0.5,fill=none] {};
                 
			\end{tikzpicture}

& $a=-\frac{v( E)}{6}$ & \scalebox{1.5}{
\clusterpicture            
  \Root[D] {1} {first} {r1};
  \Root[D] {} {r1} {r2};
  \Root[D] {2.5} {r2} {r3};
  \Root[D] {} {r3} {r4};
    \Root[D] {2.5} {r4} {r5};
    \Root[D] {} {r5} {r6};
  \ClusterLD c1[][\frac{a}{2}] = (r1)(r2);
  \ClusterLD c2[][\frac{a}{2}] = (r3)(r4);
  \ClusterLD c3[][\frac{a}{2}] = (r5)(r6);
  \ClusterD c4[0] = (c1)(c2)(c3);
\endclusterpicture} \\ \hline

$F$
                & $FA$ & &
\begin{tikzpicture}
				[scale=0.2, auto=left,every node/.style={circle,fill=black!20,scale=0.6}]

                \node[white] (n1) at (3,8) {};
                
                \draw[black, thick] (0.5,5) -- (-2.75,3) node [text=black, left, pos=0.5,fill=none] {};
                \draw[black, thick] (2.5,5) -- (5.75,3) node [text=black, left, pos=0.5,fill=none] {};
                \draw[black, thick, dashed] (-1,4.5) -- (4,4.5) node [text=black, above, scale=0.8,pos=0.5,fill=none] {$a-1$};

                \draw[black, thick] (0.5,2.5) -- (-2.75,0.5) node [text=black, left, pos=0.5,fill=none] {};
                \draw[black, thick] (2.5,2.5) -- (5.75,0.5) node [text=black, left, pos=0.5,fill=none] {};
                \draw[black, thick, dashed] (-1,2) -- (4,2) node [text=black, above, scale=0.8,pos=0.5,fill=none] {$b-1$};

                \draw[black, thick] (0.5,0) -- (-2.75,-2) node [text=black, left, pos=0.5,fill=none] {};
                \draw[black, thick] (2.5,0) -- (5.75,-2) node [text=black, left, pos=0.5,fill=none] {};
                \draw[black, thick, dashed] (-1,-0.5) -- (4,-0.5) node [text=black, above, scale=0.8,pos=0.5,fill=none] {$b-1$};

                \draw[black, thick] (-2.25,4.25) -- (-2.25,-2.75) node [text=black, left, pos=0.5,fill=none] {};
                 \draw[black, thick] (5.25,4.25) -- (5.25,-2.75) node [text=black, left, pos=0.5,fill=none] {};
                 
			\end{tikzpicture}

$2b>a>b$

& 

$a=\frac{1}{2} (-3 v( F) + v( {FA}))$

$b=\frac{1}{2} (2 v( F) - v( {FA}))$

& \scalebox{1.5}{
\clusterpicture            
  \Root[D] {1} {first} {r1};
  \Root[D] {} {r1} {r2};
  \Root[D] {2.5} {r2} {r3};
  \Root[D] {} {r3} {r4};
    \Root[D] {2.5} {r4} {r5};
    \Root[D] {} {r5} {r6};
  \ClusterLD c1[][\frac{a}{2}] = (r1)(r2);
  \ClusterLD c2[][\frac{b}{2}] = (r3)(r4);
  \ClusterLD c3[][\frac{b}{2}] = (r5)(r6);
  \ClusterD c4[0] = (c1)(c2)(c3);
\endclusterpicture} \\ \hline 

$F$
                & $FB$ & &
\begin{tikzpicture}
				[scale=0.2, auto=left,every node/.style={circle,fill=black!20,scale=0.6}]

                \node[white] (n1) at (3,8) {};
                
                \draw[black, thick] (0.5,5) -- (-2.75,3) node [text=black, left, pos=0.5,fill=none] {};
                \draw[black, thick] (2.5,5) -- (5.75,3) node [text=black, left, pos=0.5,fill=none] {};
                \draw[black, thick, dashed] (-1,4.5) -- (4,4.5) node [text=black, above, scale=0.8,pos=0.5,fill=none] {$a-1$};

                \draw[black, thick] (0.5,2.5) -- (-2.75,0.5) node [text=black, left, pos=0.5,fill=none] {};
                \draw[black, thick] (2.5,2.5) -- (5.75,0.5) node [text=black, left, pos=0.5,fill=none] {};
                \draw[black, thick, dashed] (-1,2) -- (4,2) node [text=black, above, scale=0.8,pos=0.5,fill=none] {$b-1$};

                \draw[black, thick] (0.5,0) -- (-2.75,-2) node [text=black, left, pos=0.5,fill=none] {};
                \draw[black, thick] (2.5,0) -- (5.75,-2) node [text=black, left, pos=0.5,fill=none] {};
                \draw[black, thick, dashed] (-1,-0.5) -- (4,-0.5) node [text=black, above, scale=0.8,pos=0.5,fill=none] {$b-1$};

                \draw[black, thick] (-2.25,4.25) -- (-2.25,-2.75) node [text=black, left, pos=0.5,fill=none] {};
                 \draw[black, thick] (5.25,4.25) -- (5.25,-2.75) node [text=black, left, pos=0.5,fill=none] {};
                 
			\end{tikzpicture}

$a>2b$

& 

$a=\frac{1}{4} (2 v( F) - v( {FB}))$

$b=\frac{1}{4} (-4 v( F) + v( {FB}))$

& \scalebox{1.5}{
\clusterpicture            
  \Root[D] {1} {first} {r1};
  \Root[D] {} {r1} {r2};
  \Root[D] {2.5} {r2} {r3};
  \Root[D] {} {r3} {r4};
    \Root[D] {2.5} {r4} {r5};
    \Root[D] {} {r5} {r6};
  \ClusterLD c1[][\frac{a}{2}] = (r1)(r2);
  \ClusterLD c2[][\frac{b}{2}] = (r3)(r4);
  \ClusterLD c3[][\frac{b}{2}] = (r5)(r6);
  \ClusterD c4[0] = (c1)(c2)(c3);
\endclusterpicture} \\ \hline 

$F$
                & $FC$ & &
\begin{tikzpicture}
				[scale=0.2, auto=left,every node/.style={circle,fill=black!20,scale=0.6}]

                \node[white] (n1) at (3,8) {};
                
                \draw[black, thick] (0.5,5) -- (-2.75,3) node [text=black, left, pos=0.5,fill=none] {};
                \draw[black, thick] (2.5,5) -- (5.75,3) node [text=black, left, pos=0.5,fill=none] {};
                \draw[black, thick, dashed] (-1,4.5) -- (4,4.5) node [text=black, above, scale=0.8,pos=0.5,fill=none] {$a-1$};

                \draw[black, thick] (0.5,2.5) -- (-2.75,0.5) node [text=black, left, pos=0.5,fill=none] {};
                \draw[black, thick] (2.5,2.5) -- (5.75,0.5) node [text=black, left, pos=0.5,fill=none] {};
                \draw[black, thick, dashed] (-1,2) -- (4,2) node [text=black, above,scale=0.8, pos=0.5,fill=none] {$b-1$};

                \draw[black, thick] (0.5,0) -- (-2.75,-2) node [text=black, left, pos=0.5,fill=none] {};
                \draw[black, thick] (2.5,0) -- (5.75,-2) node [text=black, left, pos=0.5,fill=none] {};
                \draw[black, thick, dashed] (-1,-0.5) -- (4,-0.5) node [text=black, above, scale=0.8,pos=0.5,fill=none] {$b-1$};

                \draw[black, thick] (-2.25,4.25) -- (-2.25,-2.75) node [text=black, left, pos=0.5,fill=none] {};
                 \draw[black, thick] (5.25,4.25) -- (5.25,-2.75) node [text=black, left, pos=0.5,fill=none] {};
                 
			\end{tikzpicture}

$2b=a>b$

& 

$a=\frac{1}{2} (3 v( F) - v( {FC}))$ 

$b=\frac{1}{2} (-4 v( F) + v( {FC}))$

& \scalebox{1.5}{
\clusterpicture            
  \Root[D] {1} {first} {r1};
  \Root[D] {} {r1} {r2};
  \Root[D] {2.5} {r2} {r3};
  \Root[D] {} {r3} {r4};
    \Root[D] {2.5} {r4} {r5};
    \Root[D] {} {r5} {r6};
  \ClusterLD c1[][\frac{a}{2}] = (r1)(r2);
  \ClusterLD c2[][\frac{b}{2}] = (r3)(r4);
  \ClusterLD c3[][\frac{b}{2}] = (r5)(r6);
  \ClusterD c4[0] = (c1)(c2)(c3);
\endclusterpicture} \\ \hline 

\end{longtable}}

\end{center}

\begin{proof}[Proof of Theorem \ref{genus2theorem}]
The possible stable model trees and possible orderings on the distances between the vertices are listed below. In the orderings, we have only included the distances up to which we will show that the stable model tree is uniquely determined by the corresponding absolute invariants. We have labelled the edges consistently with the labelling of the lengths of chains in Table \ref{genus2table} above.

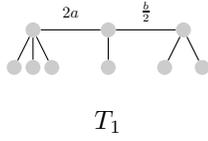
\begin{center}
\begin{minipage}{0.4\textwidth}
\begin{figure}[H]
			\begin{tikzpicture}
				[scale=0.5, auto=left,every node/.style={circle,fill=black!20,scale=0.6}]
				\node  (n1) at (0,0) {};
				\node (n2) at (2,0)  {};
				\node (n3) at (4,0)  {};

                \node (n4) at (-0.5,-1)  {};
                \node (n5) at (0.5,-1)  {};
                \node (n6) at (0,-1)  {};
                
                \node (n7) at (2,-1)  {};
                
                \node (n8) at (3.5,-1)  {};
                 \node (n9) at (4.5,-1)  {};

                 \draw (n1) -- (n4) node [text=black, pos=0.4, left, fill=none] {};
                 \draw (n1) -- (n5) node [text=black, pos=0.4, left, fill=none] {};

                 \draw (n1) -- (n6) node [text=black, pos=0.4, left, fill=none] {};
                 
                 \draw (n2) -- (n7) node [text=black, pos=0.4, left, fill=none] {};

                 \draw (n3) -- (n8) node [text=black, pos=0.4, left, fill=none] {};
                 \draw (n3) -- (n9) node [text=black, pos=0.4, left, fill=none] {};

				\draw (n1) -- (n2) node [text=black, pos=0.5, above, fill=none] {$2a$};
                \draw (n2) -- (n3) node [text=black, pos=0.5, above,fill=none] {$\frac{b}{2}$};
			\end{tikzpicture}
            \caption*{$T_1$}
   \end{figure}
   \end{minipage}\begin{minipage}{0.6\textwidth}

$I_1=\{4a>b\};$

$I_2=\{b>4a\};$

$I_3=\{4a=b\}.$

   \end{minipage}
   \end{center}
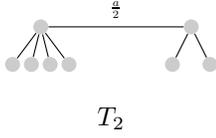
\begin{center}
\begin{minipage}{0.4\textwidth}
\begin{figure}[H]
			\begin{tikzpicture}
				[scale=0.5, auto=left,every node/.style={circle,fill=black!20,scale=0.6}]
				\node  (n1) at (0,0) {};
				\node (n3) at (4,0)  {};

                \node (n4) at (-0.25,-1)  {};
                \node (n5) at (0.25,-1)  {};
                \node (n6) at (3.5,-1)  {};
                \node (n7) at (4.5,-1)  {};

                \node (n8) at (-0.75,-1)  {};
                \node (n9) at (0.75,-1)  {};

				\draw (n1) -- (n4) node [text=black, pos=0.4, left, fill=none] {};
                \draw (n1) -- (n5) node [text=black, pos=0.4, left, fill=none] {};
                \draw (n1) -- (n8) node [text=black, pos=0.4, left, fill=none] {};
                \draw (n1) -- (n9) node [text=black, pos=0.4, left, fill=none] {};
                \draw (n3) -- (n6) node [text=black, pos=0.4, left, fill=none] {};
                \draw (n3) -- (n7) node [text=black, pos=0.4, left, fill=none] {};

				\draw (n1) -- (n3) node [text=black, pos=0.5, above, fill=none] {$\frac{a}{2}$};
			\end{tikzpicture}
            \caption*{$T_2$}
   \end{figure}
   \end{minipage}\begin{minipage}{0.6\textwidth}

\hspace{10pt}

   \end{minipage}
\end{center}
\begin{center}
\begin{minipage}{0.4\textwidth}
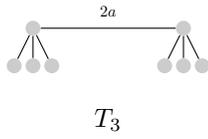
\begin{figure}[H]
			\begin{tikzpicture}
				[scale=0.5, auto=left,every node/.style={circle,fill=black!20,scale=0.6}]
				\node  (n1) at (0,0) {};
				\node (n2) at (4,0)  {};

                \node (n4) at (-0.5,-1)  {};
                \node (n5) at (0.5,-1)  {};
                \node (n6) at (3.5,-1)  {};
                \node (n7) at (4.5,-1)  {};
                \node (n8) at (0,-1)  {};
                 \node (n9) at (4,-1)  {};

                \draw (n1) -- (n4) node [text=black, pos=0.4, left, fill=none] {};
                \draw (n1) -- (n5) node [text=black, pos=0.4, left, fill=none] {};
                 \draw (n1) -- (n8) node [text=black, pos=0.4, left, fill=none] {};
                \draw (n2) -- (n6) node [text=black, pos=0.4, left, fill=none] {};
                \draw (n2) -- (n7) node [text=black, pos=0.4, left, fill=none] {};
			\draw (n2) -- (n9) node [text=black, pos=0.4, left, fill=none] {};

				\draw (n1) -- (n2) node [text=black, pos=0.5, above, fill=none] {$2a$};
			\end{tikzpicture}
            \caption*{$T_3$}
   \end{figure}
   \end{minipage}\begin{minipage}{0.6\textwidth}
   
\hspace{10pt}

   \end{minipage}
\end{center}
\begin{center}
\begin{minipage}{0.4\textwidth}
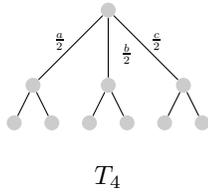
\begin{figure}[H]
			\begin{tikzpicture}
				[scale=0.5, auto=left,every node/.style={circle,fill=black!20,scale=0.6}]
                \node (n10) at (2,2) {};
				\node  (n1) at (0,0) {};
				\node (n2) at (2,0)  {};
				\node (n3) at (4,0)  {};

                \node (n4) at (-0.5,-1) {};
                \node (n5) at (0.5,-1) {};
                
                \node (n6) at (1.5,-1) {};
                \node (n7) at (2.5,-1) {};

                \node (n8) at (3.5,-1) {};
                \node (n9) at (4.5,-1) {};

                \draw (n1) -- (n4); 
                \draw (n1) -- (n5); 
                
                \draw (n2) -- (n6); 
                \draw (n2) -- (n7); 

                \draw (n3) -- (n8); 
                \draw (n3) -- (n9);

				\draw (n10) -- (n1) node [text=black, pos=0.4, left, fill=none] {$\frac{a}{2}$};
                \draw (n10) -- (n2) node [text=black, pos=0.6, right,fill=none] {$\frac{b}{2}$};
				\draw (n10) -- (n3) node [text=black, pos=0.4, right,fill=none] {$\frac{c}{2}$};
			\end{tikzpicture}
            \caption*{$T_4$}
   \end{figure}
   \end{minipage}\begin{minipage}{0.6\textwidth}

$I_1=\{a>b>c, \ a<b+c\}; \quad I_5=\{a=b=c\};$

$I_2=\{a>b>c, \ a>b+c\}; \quad I_6=\{a>b=c,a<2b\};$

$I_3=\{a>b>c, \ a=b+c\}; \quad I_7=\{a>b=c,a>2b\};$

$I_4=\{a=b>c\}; \hspace{57pt} I_8=\{a>b=c,a=2b\}.$

   \end{minipage}
\end{center}
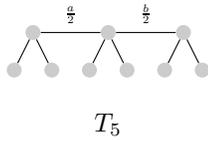
\begin{center}
\begin{minipage}{0.4\textwidth}
\begin{figure}[H]
			\begin{tikzpicture}
				[scale=0.5, auto=left,every node/.style={circle,fill=black!20,scale=0.6}]
				\node  (n1) at (0,0) {};
				\node (n2) at (2,0)  {};
				\node (n3) at (4,0)  {};

                \node (n4) at (-0.5,-1)  {};
                \node (n5) at (0.5,-1)  {};
                \node (n6) at (1.5,-1)  {};
                \node (n7) at (2.5,-1)  {};
                \node (n8) at (3.5,-1)  {};
                 \node (n9) at (4.5,-1)  {};

                 \draw (n1) -- (n4) node [text=black, pos=0.4, left, fill=none] {};
                 \draw (n1) -- (n5) node [text=black, pos=0.4, left, fill=none] {};

                 \draw (n2) -- (n6) node [text=black, pos=0.4, left, fill=none] {};
                 \draw (n2) -- (n7) node [text=black, pos=0.4, left, fill=none] {};

                 \draw (n3) -- (n8) node [text=black, pos=0.4, left, fill=none] {};
                 \draw (n3) -- (n9) node [text=black, pos=0.4, left, fill=none] {};

				\draw (n1) -- (n2) node [text=black, pos=0.5, above, fill=none] {$\frac{a}{2}$};
                \draw (n2) -- (n3) node [text=black, pos=0.5, above,fill=none] {$\frac{b}{2}$};
			\end{tikzpicture}
            \caption*{$T_5$}
   \end{figure}
   \end{minipage}\begin{minipage}{0.6\textwidth}

$I_1=\{a>b\};$

$I_2=\{a=b\}.$

   \end{minipage}
\end{center}
\begin{center}
\begin{minipage}{0.4\textwidth}
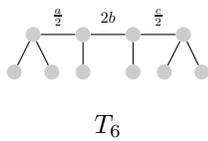
\begin{figure}[H]
			\begin{tikzpicture}
				[scale=0.5, auto=left,every node/.style={circle,fill=black!20,scale=0.6}]
				\node  (n1) at (0,0) {};
				\node (n2) at (1.33,0)  {};
				\node (n3) at (2.66,0)  {};
                \node (n4) at (4,0)  {};

                \node (n5) at (-0.5,-1)  {};
                \node (n6) at (0.5,-1)  {};
                
                \node (n7) at (1.33,-1)  {};
                
                \node (n8) at (2.66,-1)  {};
                
                \node (n9) at (3.49,-1)  {};
                 \node (n10) at (4.49,-1)  {};

                 \draw (n1) -- (n5) node [text=black, pos=0.4, left, fill=none] {};
                 \draw (n1) -- (n6) node [text=black, pos=0.4, left, fill=none] {};

                 \draw (n2) -- (n7) node [text=black, pos=0.4, left, fill=none] {};
                 
                 \draw (n3) -- (n8) node [text=black, pos=0.4, left, fill=none] {};

                 \draw (n4) -- (n9) node [text=black, pos=0.4, left, fill=none] {};
                 \draw (n4) -- (n10) node [text=black, pos=0.4, left, fill=none] {};

				\draw (n1) -- (n2) node [text=black, pos=0.5, above, fill=none] {$\frac{a}{2}$};
                \draw (n2) -- (n3) node [text=black, pos=0.5, above,fill=none] {$2b$};
				\draw (n3) -- (n4) node [text=black, pos=0.5, above,fill=none] {$\frac{c}{2}$};
			\end{tikzpicture}
            \caption*{$T_6$}
   \end{figure}
   \end{minipage}\begin{minipage}{0.6\textwidth}

$I_1=\{a>c+4b\};$

$I_2=\{c+4b>a>c\};$

$I_3=\{a=c+4b\};$

$I_4=\{a=c\}.$

   \end{minipage}
\end{center}

By Proposition \ref{potgoodred}, $C/K$ has potentially good reduction if and only if $v(X(C))\geq 0$ for every absolute invariant listed above labelled with one letter. Following Definition \ref{invariantsdefinition}, we have 

\begin{center}
\begin{minipage}[t]{0.4\textwidth}
\begin{align}
A&=\textup{Inv}_{T_1,I_j,1} \textup{ for $j=1,2,3$};\\
B&=\textup{Inv}_{T_2,I,1};\\
C&=\textup{Inv}_{T_3,I,1};\\
D&=\textup{Inv}_{T_4,I_j,1} \textup{ for $j=1,2,3,4$};\\
D&=\textup{Inv}_{T_5,I_j,1} \textup{ for $j=1$ and $2$};\\
D&=\textup{Inv}_{T_6,I_j,1} \textup{ for $j=1,2,3,4$};\\
E&=\textup{Inv}_{T_4,I_5,1}; \\
F&=\textup{Inv}_{T_4,I_j,1} \textup{ for $j=6,7,8$};
\end{align}
\end{minipage}\begin{minipage}[t]{0.4\textwidth}
\begin{align}
AA&=\textup{Inv}_{T_1,I_1,2};\\
AB&=\textup{Inv}_{T_1,I_2,2};\\
AC&=\textup{Inv}_{T_1,I_3,2};\\
DA&=\textup{Inv}_{T_4,I_j,2} \textup{ for $j=1,2,3$};\\
DB&=\textup{Inv}_{T_4,I_4,2};\\
DC&=\textup{Inv}_{T_5,I_1,2};\\
DD&=\textup{Inv}_{T_5,I_2,2}; \\
DE&=\textup{Inv}_{T_6,I_j,2} \textup{ for $j=1,2,3$}; \\
DF&=\textup{Inv}_{T_6,I_4,2}; \\
FA&=\textup{Inv}_{T_4,I_6,2}; \\
FB&=\textup{Inv}_{T_4,I_7,2}; \\
FC&=\textup{Inv}_{T_4,I_8,2};\\
\end{align}
\end{minipage}\begin{minipage}[t]{0.2\textwidth}
\begin{align}
DAA&=\textup{Inv}_{T_4,I_1,3};\\
DAB&=\textup{Inv}_{T_4,I_2,3};\\
DAC&=\textup{Inv}_{T_4,I_3,3};\\
DEA&=\textup{Inv}_{T_6,I_1,3};\\
DEB&=\textup{Inv}_{T_6,I_2,3};\\
DEC&=\textup{Inv}_{T_6,I_3,3}. 
\end{align}
\end{minipage}
\end{center}

Theorem \ref{genus2theorem} then follows immediately from Theorem \ref{wholealgorithm}, and the fact that the final absolute invariant in the series of absolute invariants for a stable model tree with an ordering on the distances between the vertices above uniquely determines the stable model trees with such an ordering. The special fibres corresponding to the possible stable model trees above can be written down by constructing the BY tree using Proposition \ref{byfromtc} and applying Theorem 5.18 of \cite{m2d2}. 
\end{proof}

\begin{example}\label{genus2example}
Let us take the semistable genus $2$ curve given by the Weierstrass equation 
\begin{equation}
\mathcal{C}:y^2=x^6+6x^5-386x^4-1668x^3+17539x^2+67326x-50274
\end{equation}
over $\mathbb{Q}_7$. We will use Theorem \ref{genus2theorem} to recover the dual graph of the special fibre of the minimal regular model of $\mathcal{C}/\mathbb{Q}_7^{\textup{unr}}$. The valuations of absolute invariants above were calculated using SageMath \cite{sage}.

Denote by $v$ the $7$-adic valuation. First, we need the valuation of all absolute invariants that are labelled with one letter. We have
\begin{align}
    v(A(\mathcal{C}))=-11, \quad v(B(\mathcal{C}))=-21&, \quad v(C(\mathcal{C}))=-13,
    \quad v(D(\mathcal{C}))=-5, \\
    \quad v(E(\mathcal{C}))=-12, \quad &\textup{and} \quad  v(F(\mathcal{C}))=-9. 
\end{align}
Since at least one of them has valuation $< 0$, $\mathcal{C}/K$ does not have potentially good reduction. Since $K_1(A)=3$, $K_1(B)=6$, $K_1(C)=9$, $K_1(D)=1$, $K_1(E)=3$ and $K_1(F)=2$, we have
\begin{equation}
    B_1(A,\mathcal{C})=\frac{11}{6},\quad  B_1(B,\mathcal{C})=\frac{7}{4},  \quad B_1(C,\mathcal{C})=\frac{13}{18}, \quad B_1(D,\mathcal{C})=\frac{5}{2},  \quad B_1(E,\mathcal{C})=2, \quad \textup{and}\quad B_1(F,\mathcal{C})=\frac{9}{4}.
\end{equation}
Since $B_1(D,\mathcal{C})$ is the largest out of these, Theorem \ref{genus2theorem} tells us that $I_1=D$ and $\delta_{1}(\mathcal{C})=\frac{5}{2}$. Given this, in order to find the second absolute invariant of $C$ we need to calculate the valuation of $DA$, $DB$, $DC$, $DD$, $DE$ and $DF$. We find that 
\begin{align}
    v(DA(\mathcal{C}))=-14, \quad v(DB(\mathcal{C}))=-17&, \quad v(DC(\mathcal{C}))=-26,
    \quad v(DD(\mathcal{C}))=-37, \\
    \quad v(DE(\mathcal{C}))=-16, \quad &\textup{and} \quad  v(DF(\mathcal{C}))=-20.
\end{align}
Since $K_2(DA)=1$, $K_2(DB)=2$, $K_2(DC)=5$, $K_2(DD)=10$, $K_2(DE)=2$, $K_2(DF)=4$ and $\delta_1(\mathcal{C})=\frac{5}{2}$, we have 
\begin{align}
    B_2(DA,\mathcal{C})=2,\quad  B_2(DB,\mathcal{C})=\frac{7}{4},&  \quad B_2(DC,\mathcal{C})=\frac{8}{5}, \quad B_2(DD,\mathcal{C})=\frac{27}{20},  \\
    \quad B_2(DE,\mathcal{C})=\frac{3}{4}, \quad &\textup{and}\quad B_2(DF,\mathcal{C})=\frac{5}{4}.
\end{align}
Since $B_2(DA,\mathcal{C})$ is the largest, this tells us that $I_1I_2=DA$ and $\delta_2(\mathcal{C})=2$. Now, the possible $3$-rd absolute invariants are $DAA$, $DAB$ and $DAC$. We have 
\begin{equation}
    v(DAA(\mathcal{C}))=-26, \quad v(DAB(\mathcal{C}))=-35 \quad \textup{ and } \quad v(DAC(\mathcal{C}))=-38
\end{equation}
and so since $K_3(DAA)=1$, $K_3(DAB)=4$, $K_3(DAC)=5$ and $\delta_2(\mathcal{C})=2$, we obtain
\begin{equation}
    B_3(DAA,\mathcal{C})=\frac{3}{2}, \quad B_3(DAB,\mathcal{C})=\frac{3}{2} \quad \textup{ and } \quad B_3(DAC,\mathcal{C})=\frac{3}{2}.
\end{equation}
Thus $I_1I_2I_3=DAC$ and $\delta_3(C)=B_3(DAC,\mathcal{C})=\frac{3}{2}$. By Theorem \ref{genus2theorem}, this uniquely determines the special fibre of the minimal regular model of $C/K^{\textup{unr}}$ from Table \ref{genus2table} as the following, where each line represents a component that is isomorphic to $\mathbb{P}^1$.
\begin{center}
\begin{minipage}[b]{1\textwidth}
		\begin{center}
		\begin{figure}[H]
			\begin{tikzpicture}
				[scale=0.3, auto=left,every node/.style={circle,fill=black!20,scale=0.6}]
               
                \draw[black, thick] (2.25,2.5) -- (-2.75,1) node [text=black, left, pos=0.5,fill=none] {};
                \draw[black, thick] (0.75,2.5) -- (5.75,1) node [text=black, left, pos=0.5,fill=none] {};
                
                \draw[black, thick] (-2,-1) -- (5,-1) node [text=black, left, pos=0.5,fill=none] {};

                \draw[black, thick] (-2.25,2.25) -- (2,-5.25) node [text=black, left, pos=0.5,fill=none] {};
                 \draw[black, thick] (5.25,2.25) -- (1,-5.25) node [text=black, left, pos=0.5,fill=none] {};
                 
			\end{tikzpicture}
 \caption*{Special fibre of the minimal regular model of $C/K^{\textup{unr}}$}
		\end{figure}
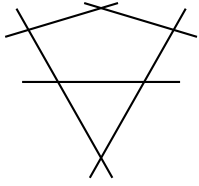
				\end{center}
    \end{minipage}
\end{center}
The lengths of the chains can be calculated using the formula displayed in the table 
\begin{align}
a&=\frac{1}{2} (-2 v(D(\mathcal{C})) + 3 v(DA(\mathcal{C}) - v(DAC(\mathcal{C})))=\frac{1}{2}(-2\cdot(-5)+3\cdot(-14)-(-38))=3; \\
b&=\frac{1}{2} (-3 v(DA(\mathcal{C})) + v(DAC(\mathcal{C})))=\frac{1}{2}(-3\cdot(-14)+(-38))=2; \\
c&=\frac{1}{2} (6 v(D(\mathcal{C})) - 5 v(DA(\mathcal{C})) + v(DAC(\mathcal{C})))=\frac{1}{2}(6\cdot (-5)-5\cdot(-14)+(-38))=1. 
\end{align}
\end{example}

\appendix

\section{On invariants of genus 2 curves by Elisa Lorenzo Garc\'ia}\label{appendix}

Absolutely irreducible, smooth curves of genus $2$ over an algebraically closed field $K$ are hyperelliptic. Hence, if $K$ is a local field with odd characteristic, they are given by an affine hyperelliptic model $$y^2=f(x)=c_f\prod_{i=1}^d(x-x_i)$$ with $f\in K[x]$ a polynomial of degree $d=5$ or $6$. The projective model, in weighted projective space $\mathbb{P}^2_{1,3,1}(K)$ can be written as $$Y^2=f(X,Z)=\prod_{i=1}^6(\alpha_i X-\beta_i Z).$$
If $d=6$, then $c_f=\prod_{i=1}^6\alpha_i$ and $x_i=\beta_i/\alpha_i$ and if $d=5$ then $\alpha_6=0$, $c_f=\prod_{i=1}^5\alpha_i\cdot \beta_6$ and $x_i=\beta_i/\alpha_i$. 

The moduli space $\mathcal{M}_2$ of smooth genus 2 curves was first studied and parametrised by Igusa in \cite{igusa}. He defined the Igusa invariants $J_2$, $J_4$, $J_6$ and $J_{10}$ in \cite{igusa} p.621-622 as
\begin{eqnarray*}
J_2 &=&2^{-6}\sum(12)^2(34)^2(56)^2,\\
J_4 & =& 2^{-9}3^{-2}\left(48J_2^2-\sum(123)^2(456)^2\right),\\
J_6 & = & 2^{-7}3^{-2}\left(16J_2^3-320J_2J_4-\sum(123)^2(456)^2(14)^2(25)^2(36)^2\right),\\
J_{10} & =&2^{-12}\prod_{i<j}(ij)^2,
\end{eqnarray*}
where $(S)=\prod_{\substack{i,j\in S\\i\neq j}}(\alpha_j\beta_i-\alpha_i\beta_j)$ and the sums run through all permutations in the symmetric group $S_6$.

Although it may appear otherwise, these expressions are well-defined in characteristic $2$ and $3$ since they can be written as integer combinations of the coefficients of $f$ (see \S4 of \cite{igusa}). Prior to Igusa, Gordan \cite{gordan}, Clebsch \cite{clebsch} and Bolza \cite{bolza} had already studied some of these invariants for binary sextics.

For a general introduction to invariants of hyperelliptic curves (or binary forms), we invite the reader to consult \cite{LR11}. Briefly, two binary forms $f(X,Z)$ and $g(X,Z)$ of the same degree $d$ are said to be equivalent if there exists a matrix $M\in \operatorname{GL}_2(\bar{K})$ such that $f=g\circ M$, that is, if they define isomorphic hyperelliptic curves. A polynomial $P$ on the coefficients of a binary form of degree $d$ is said to be an \textit{invariant} of \textit{weight} $k$ if for all $M\in \operatorname{GL}_2(\bar{K})$, one has $P(f\circ M)=\operatorname{det}(M)^k\cdot P(f)$. An \textit{absolute invariant} is a quotient of same weight invariants, so that $\frac{P}{Q}(f\circ M)=\frac{P}{Q}(f)$.     

\begin{theorem}[Corollary on p. 632 of \cite{igusa}] Two sextic binary forms $f$ and $g$ define isomorphic genus 2 curves if and only if $(J_2(f):J_4(f):J_6(f):J_{10}(f))=(J_2(g):J_4(g):J_6(g):J_{10}(g))\in \mathbb{P}^3_{1,2,3,5}(\bar{K})$ as points is a weighted projective space. 
\end{theorem}

\begin{remark} An equivalent result holds in characteristic $2$, but a fifth invariant $J_8$ is needed.
\end{remark}

Igusa invariants are indeed invariants according to previous definition (see \cite[I.11 Lecture XXIX]{hilbert}) and are easy to compute from the coefficients\footnote{This is not true for genus greater than $3$ since the expression of the invariants in terms of the coefficients of the binary form is too large. In this case, a basis of the invariants is computed using differential operators called \textit{transvectants} applied to the equation of the curve. Invariants obtained from transvectants are useful for the reconstruction process of a curve from its invariants, see e.g. \cite{thomas}.} of $f$, so there is no need to factor $f$ and compute its roots over $\Bar{K}$. Igusa invariants are implemented in SageMath \cite{sage} and Magma \cite{magma}. They determine isomorphism classes of genus $2$ curves, and given a point $(J_2:J_4:J_6:J_{10})\in \mathbb{P}^3_{1,2,3,5}(\bar{K})$ with $J_{10}\neq 0$, there exists a genus $2$ curve with those given invariants (see \cite{mestre}). For this reason, we say that they parametrise the moduli space of irreducible smooth projective curves of genus $2$. Cardona and Quer used these invariants to describe the stratification of $\mathcal{M}_2$ by automorphism groups \cite{cardona}. 

Another application of Igusa invariants is describing reduction types: first Mestre \cite{mestre} in large characteristic (using a similar and simplified argument to the one in this paper) and then Liu \cite{liu2} in all characteristics, described the special fibre of the stable model of genus $2$ curves in terms of the valuations of some combinations of Igusa invariants. In this paper, Cowland Kellock goes further and describes the special fibre of the minimal regular model in the semistable case in terms of valuations of invariants (see Theorem \ref{wholealgorithm}). 

The goal of this Appendix is to give expressions of the genus $2$ absolute invariants defined in \S\ref{genus2section} in terms of Igusa invariants. Although the absolute invariants are quotients of polynomials that are symmetric in the roots of $f(x)$ so can be written in terms of the coefficients of $f(x)$, the expressions are large and this would be computationally expensive; writing them in terms of the Igusa invariants avoids this issue. This avoids expressing the absolute invariants in terms of the coefficients of the equation of the curve, which are very large. The most compact way to write them down is as a linear combination of products of Igusa invariants.

\begin{example} In terms of the roots of a Weierstrass equation for the curve, the absolute invariant $D$ from \S\ref{genus2section} can be written as 
$$D=\frac{\frac{1}{|\textup{Stab}|}\sum_{S_6}(12,34)^3(1234,56)^2(56)^2}{\Delta},$$ 
where $(S,T)=\prod_{i\in S, j\in T}(\alpha_j\beta_i-\alpha_i\beta_j)$, $(S)^2=\prod_{\substack{i,j\in S\\i\neq j}}(\alpha_j\beta_i-\alpha_i\beta_j)^2$, $\Delta=\prod_{i<j}(ij)^2$ and $|\textup{Stab}|$ is the stabiliser of $(12,34)^3(1234,56)^2(56)^2$ under the natural action of $S_6$. The numerator and denominator are invariants of weight $10$ and we can write
$$\sum_{S_6}(12,34)^3(1234,56)^2(56)^2=-1/2J_2J_4^2 + 2J_4J_6$$ 
in terms of the Igusa invariants $J_2$, $J_4$ and $J_6$. However, in terms of the coefficients of the curve
$$y^2=ax^6+bx^5+cx^4+dx^3+ex^2+fx+g,$$ the sum takes $53$ lines, the first $2$ being:
\footnotesize{\begin{align*}-7383200a^5g^5 + 61236000a^4bfg^4 + 12830400a^4ceg^4 -
    15552000a^4cf^2g^3 + 7873200a^4d^2g^4 - 16912800a^4defg^3 &+\\
7290000a^4df^3g^2 + 6842880a^4e^3g^3 - 2916000a^4e^2f^2g^2 -
    15552000a^3b^2eg^4 - 13932000a^3b^2f^2g^3 - 16912800a^3bcdg^4 &+\dots
\end{align*}}
\normalsize 
Notice that the absolute invariants $DEA$ and $DEC$ in \S\ref{genus2section} are the ones with numerator and denominator of the largest weight, and this weight is $100$.  
\end{example}

\begin{theorem}(Prop. 3 in \cite{igusa}) The ring of invariants of binary forms of degree $6$ is generated by the Igusa invariants, which are algebraically independent.
\end{theorem}

\begin{corollary} The dimension of the vector space of invariants of weight $k$ of smooth genus 2 curves is equal to the numbers of restricted partitions of $k$ where the parts are restricted to being elements of $\{2,4,6,10\}$.
\end{corollary}

The idea of this Appendix is to use interpolation to express the numerators of the absolute invariants from \S\ref{genus2section} terms of a linear combination of products of Igusa invariants. This technique is already used in \cite[Thm. 4.1]{elisa}. For an invariant $I$ of weight $k$ we proceed as follows:
\begin{itemize}
\item[i)] Compute the partitions of $k$ restricted to $\{2,4,6,10\}$.
\item[ii)] Compute $I$ and the Igusa Invariants of many curves of the form $y^2=x(x^2-1)(x-a)(x-b)(x-c)$ and the corresponding products $J_2^{e_2}J_4^{e_4}J_6^{e_6}J_{10}^{e_{10}}$ with $2e_2+4e_4+6e_6+10e_{10}=k$.
\item[iii)] Solve the system $I=\sum_e \lambda_e\cdot J_2^{e_2}J_4^{e_4}J_6^{e_6}J_{10}^{e_{10}}$.
\item[iv)] Check that the solution is unique.
\end{itemize}

The Magma code used to compute the numerators of the absolute invariants in terms of Igusa invariants as well as their description is available in \cite{ancillaryfile}. We caution the reader that this is the numerator of the invariants where the sum over the whole symmetric group is taken, rather than the sum modulo stabilisers. For each of the $24$ absolute invariants we had to solve a linear system over the rational numbers, the largest ones being of size $713\times2500$ and $947\times1500$. Each of them took less than a couple of minutes using the servers at \textit{Université de Rennes 1}. When producing several curves to perform the interpolation it was important to produce enough linearly independent equations, otherwise solving the linear system took very long and we had to stop the program after a couple of hours. In this Appendix, we have only written down the absolute invariant $D$ in terms of Igusa invariants because the other absolute invariants contain too many terms and would take a large amount of space. 

Using the absolute invariants from \S\ref{genus2section} written in terms of the Igusa invariants, we implemented the algorithm from Theorem \ref{genus2theorem} that outputs the `balanced cluster picture' with depths from Table \ref{genus2table} of a given degree $5$ or $6$ univariate polynomial and the corresponding stable reduction type \cite{ancillaryfile}. The algorithm was tested and double checked with several tens of curves having each of the possible reduction types.

\end{document}